\documentclass[english,12pt,oneside]{amsproc}
\usepackage[english]{babel}
\usepackage{a4wide}
\usepackage{amsthm}
\usepackage{graphics}
\usepackage{biblatex}
\usepackage{amsfonts, amssymb, amscd, amsmath}
\usepackage{latexsym}
\usepackage[matrix,arrow,curve]{xy}
\usepackage{mathabx}
\usepackage{color}
\usepackage{pbox}
\usepackage{tikz}
\usepackage{stmaryrd}
\usepackage{hyperref}
\usepackage{tikz,tikz-cd}
\usepackage{algorithm}
\usepackage{algorithmic}
\usepackage[normalem]{ulem}
\usepackage{array}
\usepackage{mathrsfs}

\usetikzlibrary{arrows,arrows.meta, decorations.pathreplacing, calc,positioning,trees,arrows,chains,shapes.geometric,%
    decorations.pathmorphing,shapes, matrix, shapes.symbols}

\tikzcdset{arrow style=tikz, diagrams={>=stealth}}

\tikzset{
  altstackar/.style={decorate, decoration={show path construction,
    lineto code={
      \path (\tikzinputsegmentfirst); \pgfgetlastxy{\xstart}{\ystart}
      \path (\tikzinputsegmentlast); \pgfgetlastxy{\xend}{\yend}
      \path ($(0,0)!1.5pt!(\ystart-\yend,\xend-\xstart)$); \pgfgetlastxy{\xperp}{\yperp}
      \foreach \n[evaluate=\n as \k using .5*#1-\n+.5] in {1,...,#1}{
        \ifodd\n{\draw[->, shorten <=2pt, shift={($\k*(\xperp,\yperp)$)}](\xstart,\ystart)--(\xend,\yend);}
        \else{\draw[<-, shorten >=2pt, shift={($\k*(\xperp,\yperp)$)}](\xstart,\ystart)--(\xend,\yend);}\fi
      }
    }
  }}, altstackar/.default={1}
}

\DeclareMathOperator{\Cone}{Cone}
\DeclareMathOperator{\pt}{pt}
\DeclareMathOperator{\Ker}{Ker}
\DeclareMathOperator{\Coker}{Coker}
\DeclareMathOperator{\id}{id}

\DeclareMathOperator{\im}{Im}

\DeclareMathOperator{\Hom}{Hom}
\DeclareMathOperator{\Ext}{Ext}
\DeclareMathOperator{\rk}{rk}
\DeclareMathOperator{\diag}{diag}
\DeclareMathOperator{\Cy}{Cy}
\DeclareMathOperator{\End}{End}
\DeclareMathOperator{\OpSets}{OpSets}
\DeclareMathOperator{\ConvHull}{\mathbf{ConvHull}}
\DeclareMathOperator{\Ob}{Ob}
\DeclareMathOperator{\ord}{ord}
\DeclareMathOperator{\op}{op}
\DeclareMathOperator{\Cells}{Cells}

\DeclareMathOperator{\sd}{sd}
\DeclareMathOperator{\pr}{pr}
\DeclareMathOperator{\Iso}{Iso}

\DeclareMathOperator{\ChCpx}{ChCpx}
\DeclareMathOperator{\CochCpx}{CochCpx}
\DeclareMathOperator{\AChCpx}{AugChCpx}
\DeclareMathOperator{\Dif}{Dif}

\newcommand{\low}[2]{{_{\lceil}#1_{\rceil #2}}}

\newcommand{\ST}[1]{\mathbf{#1}}

\newcommand{\cat}{\ST{cat}}
\newcommand{\Sets}{\ST{Sets}}
\newcommand{\Top}{\ST{Top}}

\newcommand{\PrePos}{\ST{PrePos}}
\newcommand{\AlexTop}{\ST{AlexTop}}
\newcommand{\Diag}{\ST{Diag}}
\newcommand{\Vect}{\ST{Vect}}
\newcommand{\Abel}{\ST{Abel}}
\newcommand{\Mod}{\ST{Mod}}
\newcommand{\Cochain}{\ST{Cochain}}
\newcommand{\PreShvs}{\ST{PreShvs}}
\newcommand{\Shvs}{\ST{Shvs}}
\newcommand{\FinSets}{\ST{FinSets}}
\newcommand{\FinVect}{\ST{FinVect}}
\newcommand{\FinMod}{\ST{FinMod}}
\newcommand{\FinAbel}{\ST{FinAbel}}
\newcommand{\Topoi}{\ST{Topoi}}
\newcommand{\AbCat}{\ST{AbCat}}
\newcommand{\ConDiag}{\ST{ConDiag}}
\newcommand{\Repr}{\ST{Repr}}

\newcommand{\ko}{\Bbbk}
\newcommand{\Zo}{\mathbb{Z}}
\newcommand{\Ro}{\mathbb{R}}

\newcommand{\Qo}{\mathbb{Q}}

\newcommand{\Rg}{\mathbb{R}_{\geqslant 0}}
\newcommand{\Zg}{\mathbb{Z}_{\geqslant 0}}

\newcommand{\ca}[1]{\mathcal{#1}}
\newcommand{\inc}[2]{[#1\colon #2]}
\newcommand{\face}{\trianglelefteqslant}
\newcommand{\Fc}{\mathcal{F}}
\newcommand{\Der}{\mathscr{D}}

\newcommand{\Vv}{\mathbf{V}}
\newcommand{\Ww}{\mathbf{W}}
\newcommand{\Cc}{\mathbf{C}}

\newcommand{\Fu}{\mathfrak{F}}
\newcommand{\dd}{\partial}
\newcommand{\Hr}{\tilde{H}}
\newcommand{\br}{\tilde{\beta}}
\newcommand{\mult}{\mu}

\newcounter{stmcounter}[section]
\newcounter{problcounter}

\numberwithin{equation}{section}

\theoremstyle{plain}
\newtheorem{thm}[stmcounter]{Theorem}
\newtheorem{cor}[stmcounter]{Corollary}

\newtheorem{prop}[stmcounter]{Proposition}
\newtheorem{lem}[stmcounter]{Lemma}
\newtheorem{req}[stmcounter]{Requirement}

\newtheorem{probl}[problcounter]{Problem}

\theoremstyle{definition}
\newtheorem{defin}[stmcounter]{Definition}
\newtheorem{difin}[stmcounter]{Informal definition}

\theoremstyle{remark}
\newtheorem{ex}[stmcounter]{Example}
\newtheorem{rem}[stmcounter]{Remark}
\newtheorem{con}[stmcounter]{Construction}

\begin{document}

\begin{center}
    \rule{\linewidth}{1.5pt} \\[10pt]
    {\Large\textbf{Sheaf theory: from deep geometry to deep learning}} \\[2pt]
    \rule{\linewidth}{1.5pt}
\end{center}

\title[Sheaf theory: from deep geometry to deep learning]{}

%
%
%

\author{
\begin{center}
\vspace{1cm}
    \begin{tabular}{ccc}
        \textbf{Anton Ayzenberg} & \hspace{2cm} & \textbf{German Magai} \\
        Noeon Research, Japan & \hspace{2cm} & Noeon Research, Japan \\
        anton@noeon.ai & \hspace{2cm} & german@noeon.ai\\
        \\
        \textbf{Thomas Gebhart} & \hspace{2cm} & \textbf{Grigory Solomadin} \\
        University of Minnesota, USA & \hspace{2cm} & University of Strasbourg, France \\
        gebharttom8@gmail.com & \hspace{2cm} & grigorysolomadin@gmail.com \\
    \end{tabular}
\end{center}
\vspace{1cm}
}


\thanks{This work is supported by Noeon Research}



\begin{abstract}
This paper provides an overview of the applications of sheaf theory in deep learning, data science, and computer science in general. The primary text of this work serves as a friendly introduction to applied and computational sheaf theory accessible to those with modest mathematical familiarity. We describe intuitions and motivations underlying sheaf theory shared by both theoretical researchers and practitioners, bridging classical mathematical theory and its more recent implementations within signal processing and deep learning. We observe that most notions commonly considered specific to cellular sheaves translate to sheaves on arbitrary posets, providing an interesting avenue for further generalization of these methods in applications, and we present a new algorithm to compute sheaf cohomology on arbitrary finite posets in response. By integrating classical theory with recent applications, this work reveals certain blind spots in current machine learning practices. We conclude with a list of problems related to sheaf-theoretic applications that we find mathematically insightful and practically instructive to solve. To ensure the exposition of sheaf theory is self-contained, a rigorous mathematical introduction is provided in appendices which moves from an introduction of diagrams and sheaves to the definition of derived functors, higher order cohomology, sheaf Laplacians, sheaf diffusion, and interconnections of these subjects therein.
\end{abstract}

\maketitle

\newpage

\tableofcontents

\section{Introduction}\label{secIntro}

Scientific development and the enrichment of human knowledge are predicated on the existence of rigorous mathematical foundations from which new insights may be constructed. This foundation not only systematizes navigation in a particular field by fixing terminology and sharpening intuitions but also links practices from across fields, leading to novel research directions and the emergence of new techniques. Applied mathematics, when practiced correctly, facilitates this multiplicative growth within and across scientific domains. 

Sheaf theory is a highly technical branch of mathematics. The foundation of this discipline was laid by Jean Leray when he was a prisoner of war camp in Austria, and later published in~\cite{Leray}. Sheaf theory was conceptualized by Alexander Gro\-then\-dieck in 1957: his celebrated Tohoku paper~\cite{grothendieck1957tohoku} provided a general categorical perspective on many phenomena known in algebraic geometry, algebraic topology, and commutative algebra. Geometrical structures are ubiquitous in both theoretical mathematics and computer science. It may seem at first glance that geometrical structures studied in mathematics such as topologies, manifolds, varieties, and schemes differ significantly from those studied in applied mathematics: graphs, metric spaces, point clouds, probability distributions. However, it seems that the power of sheaf theory is universal. Not only is this theory formally applicable to finite data structures, but it also solves the very same problem as the one addressed in the classical papers: to put geometrical intuitions on a firm ground by using algebraical tools.

In recent years, there has been a fascinating fusion of methods of sheaf theory with the methods of spectral graph theory, random walks, and classical graph algorithms, leading to the notions of sheaf Laplacian and heat diffusion on a sheaf~\cite{hansen2019toward}. The methods and practices used to design new neural networks architectures have enriched sheaf theory with the ideas of sheaf learning~\cite{hansen2019learning}, message passing on sheaves~\cite{bodnar2023topological}, and sheaf attention~\cite{barbero2022sheaf1}. Sheaf neural networks~\cite{hansensheaf} were proposed as a novel architecture to overcome some problems inherent to classical graph neural architectures. In the past decades, sheaf theory found applications in a variety of other areas of computer science and applied mathematics away from deep learning (DL). The theory provides a powerful tool for the design of algorithms, such as normalization by evaluation~\cite{NormByEvalCoproducts2001} and pattern matching~\cite{Srinivas1993,Conghaile2022CohomologyIC}; it provides an expressive language to describe signal propagation~\cite{robinson2014topological}, obstructions to consistency in databases and models' behavior~\cite{AbramskyContextuality,Abramsky2022presh}, ambiguities in natural languages~\cite{SheafLanguageAmbig}, causal nets~\cite{Robinson2016SheafAC,Rosiak}, scheduling in production chains~\cite{Macfarlane04072014}, structural engineering~\cite{CooperbandThesis}, and many more.

Schematically, the path from the formal definition of a sheaf to the construction of sheaf neural architectures is shown on the left vertical trunk in Fig.~\ref{figPipeline}. The hierarchy goes from more abstract theories to more concrete, but there is no way to strictly define notions towards the bottom of the diagram without defining their top-level predecessors. The same scheme roughly resembles the chronological order in which the mentioned ideas have found applications. The application and contextualization of sheaf theory within data science may also be traced chronologically. Curry's thesis~\cite{Curry} perfectly covers the cohomological part of the scheme with a focus on cellular structures, Hansen's thesis equips cellular sheaves with a spectral theory~\cite{hansen2020laplacians}, Gebhart's thesis discusses sheaf-theoretic inductive biases in the context of deep learning theory~\cite{gebhart2023sheaf}, and Bodnar's thesis~\cite{bodnar2023topological} covers the passage from the mathematical theory of cellular sheaves to the design of neural networks architectures. However, to our knowledge, no existing work contains a detailed description of the entire pipeline from category theory to application in sufficient generality. 

\begin{figure}
  \centering
  \includegraphics[scale=0.24]{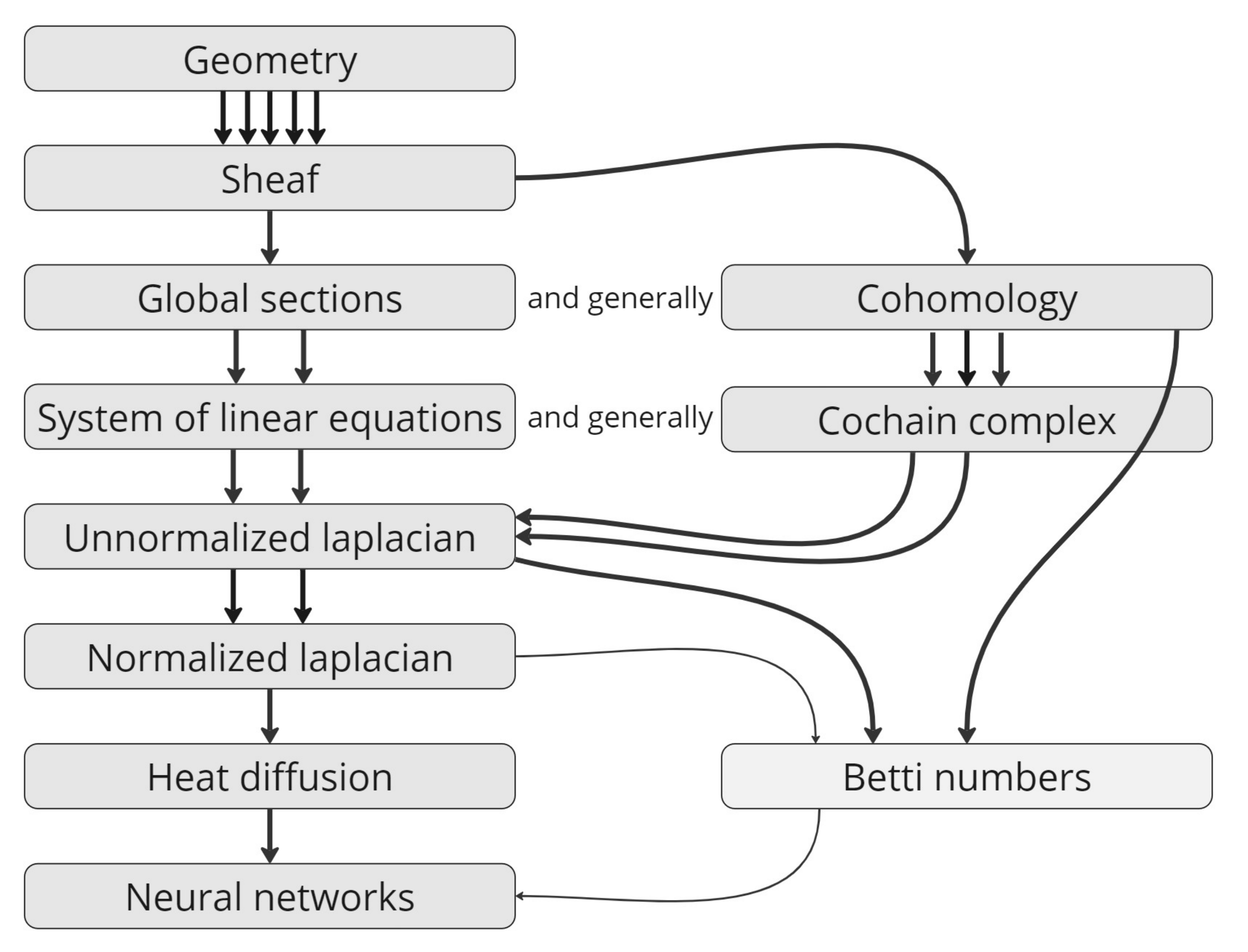}
  \caption{A way from geometry to neural networks}\label{figPipeline}
\end{figure}

In this review we center partially ordered sets (posets) as the fundamental data structure supporting a sheaf. All items shown on Fig.~\ref{figPipeline} exist and make perfect sense for general posets, not just for cell complexes as is assumed in much of the existing applied literature. In fact, most applications of sheaves in deep learning deal with graphs, which form a highly restricted subclass of cell complexes and hence can be studied with cellular sheaves. However, a number of applications have been recently proposed which seek to extend sheaves to hypergraphs~\cite{duta2024sheaf,nguyen2024sheaf} which, while combinatorial in structure, are not formally cell complexes. Hypergraph sheaf constructions are redesigned each time from scratch by analogy with graphs, but thus far lack a precise mathematical foundation. In a number of applied fields, such as formal concept analysis, knowledge representation, causal nets, and complex systems, general posets play important role. We expect that potential applications of sheaf theory are not bound merely to cellular structures, and this review considers such generalizations accordingly. 

We pursue several goals which together motivate the paper's structure.
\begin{enumerate}
  \item We give an overview of basic notions and positions of sheaf theory. This exposition is given in Section~\ref{secIntuitions} and targeted at a general audience.
  \item In Sections~\ref{secReview} and~\ref{secReviewShvsML} we survey the literature applying sheaf theory in deep learning and computer science in general. This survey may be of interest to researchers working in these areas. Section~\ref{secReviewShvsML} is devoted to sheaves applications in modern deep learning and design of neural networks. Section~\ref{secReview} surveys other ideas of sheaves' applications, most of them preceded the applications in deep learning. 
  \item A number of open research problems is gathered in Section~\ref{secProblems}. We address them to a broad community of researchers, including both homology fans and ML enjoyers.
  \item The rigorous and (mostly) self-contained exposition of the mathematical theory is placed in the Appendix sections~\ref{secMathStructures}--\ref{secMathSpaceRestoredFromShvs}. This exposition serves two goals. First, it formalizes some known practices by providing precise claims. Most of the claims are known but scattered in the mathematical literature; we gather the claims, the proofs, and references in a single place. Second, we introduce several new notions and prove a number of new results, mainly in Appendix~\ref{secMathCohomology}. 
\end{enumerate}

See the details on the overall contribution of this paper in subsection~\ref{subsecContribution}. The narration generally mirrors Fig.~\ref{figPipeline}, moving from more abstract topics to more specific topics; from earlier applications to later. Readers interested primarily in applications of sheaves within deep learning are free to start with subsection~\ref{subsecSheafLearning} having the necessary definitions and proceed to Section~\ref{secReviewShvsML}.

The preparation of this paper was greatly influenced by discussions with our colleagues and collaborators: Grigory Kondyrev, Nikita Repeev, Gregory Gelfond, Daniel Rogozin, and Andrei Krutikov. We thank Mikhail Mironov, Nadezhda Khoroshavkina, and Artem Malko for their help in writing some parts of this review. We also thank Gregory Kirgizov for bringing our attention to several extremely important papers and Andrey Filchenkov for providing valuable feedback, which helped improve the presentation.

\subsection{Fundamental notions and ideas}\label{subsecFundNotions}

As a first glimpse into sheaf theory, we provide a brief explanation of the basic terms used in the area, in particular those appearing on the scheme~\ref{figPipeline}. This description aims to introduce a common terminological basis used in sheaf-related papers, either theoretical or applied. More details and examples are provided in Section~\ref{secIntuitions}. The formal claims and proofs are gathered in the Appendices. 

By a \textbf{geometrical structure}, we mean a directed graph (digraph), i.e. a set $S$ together with a collection of ordered pairs of elements $e=(s_1,s_2)$, denoted by $e\colon s_1\to s_2$. Some further restrictions or additional structures are usually imposed on a digraph $S$. In this review we restrict ourselves to directed acyclic graphs of posets, see subsection~\ref{subsecGeometryIsPoset} for the definition and the explanation of this choice. It should be noted however that most theoretical notions described in our paper are well defined in a more general case when $S$ is a finite category.

In order to define sheaves over a structure $S$, one needs to specify a \textbf{target category} $\Vv$. We will be mainly concerned with specific categories: (1) the category $\Sets$ whose objects are sets, and morphisms are maps between sets; (2) the category $\Ro\Vect$, whose objects are real vector spaces, and morphisms are $\Ro$-linear maps. We look at such categories as theories where certain types of calculations can be performedOne can interpret such categories as theories where certain types of calculations can be performed. For example, in the category $\Ro\Vect$ we have systems of linear equations, while equations in the category $\Sets$ are constraint satisfaction problems. 

A sheaf on a geometrical structure $S$ valued in a target category $\Vv$ is a rule $D$, which associates, with any element $s\in S$ an object $D(s)$ of $\Vv$, and with any arrow $e\colon s_1\to s_2$ a morphism $D(e)\colon D(s_1)\to D(s_2)$ in $\Vv$ (i.e. a map of sets in $\Sets$ or a linear map in $\Ro\Vect$). It is assumed that compositionality restrictions (see~\eqref{eqCompositionalityMainPart}) hold for $D$, so that $D$ is a \textbf{functor} from $S$ to $\Vv$. An assignment of a sheaf to a structure $S$ may be understood as turning $S$ into a medium capable of transmitting signals of the type specified by a category $\Vv$.

A given geometrical structure $S$ supports infinitely many sheaves. The whole collection $\Shvs(S,\Vv)$, viewed as a category, faithfully remembers $S$ for meaningful choices of $\Vv$, see Appendix~\ref{secMathSpaceRestoredFromShvs}. However, this collection is too huge to process in practice. Usually, in each particular application, researchers manually choose a sheaf (or a feasible collection of sheaves) which is believed to faithfully resemble properties of $S$. 

Picking a specific sheaf on $S$ allows one to borrow computational techniques available in the category $\Vv$ to describe the properties of $S$. Given a $\Vv$-valued sheaf $D$ on $S$, we can define the \textbf{coherence equations} for $D$, see~\eqref{eqCoherentStates}. The set of solutions $\Gamma(S;D)\in \Vv$ to these equations is called \textbf{the set of global sections}, or the space of global sections if the target category is $\Ro\Vect$. The elements of $\Gamma(S;D)$ can be interpreted as consensus, or equilibrium, states of the medium represented by a sheaf $D$.

Given a sheaf $D$ on $S$, a \textbf{sheaf diffusion} on $D$ is a dynamical system which brings an arbitrary state of the medium to an equilibrium state. For $\Ro\Vect$-valued sheaves $D$ on graphs, the standard way to formally define sheaf diffusion is by applying a gradient descent to a certain quadratic function, called the \textbf{Dirichlet energy} of a sheaf. This functions measures how far is the given state from being an equilibrium state. The symmetric operator corresponding to this quadratic function is called \textbf{sheaf Laplacian}. The convergence rate of sheaf diffusion is determined by the spectrum of the sheaf Laplacian, see subsection~\ref{subsecLaplaciansMainPart}. Spectral characteristics of Laplacians reflect quantitative (i.e. metric) properties of $S$. A variation of a Laplacian, the \textbf{normalized Laplacian}, is often used in practice.

For an $\Ro\Vect$-valued sheaf $D$ on $S$, the space $\Gamma(S;D)$ of global sections enumerates consensus states of the medium $D$. This space belongs to a potentially infinite sequence of vector spaces $H^0(S;D)=\Gamma(S;D)$, $H^1(S;D)$, $H^2(S;D)$, etc., called \textbf{sheaf cohomology}. These spaces describe, respectively, consensus states, relations between consensus states, relations between relations between consensus states, and so on. The sheaf Betti numbers, $\beta_j(S;D)=\dim H^j(S;D)$ reflect qualitative (i.e. topological) properties of $S$.

In order to compute either sheaf Laplacian or cohomology of sheaves, one needs to use the device from homological algebra called the \textbf{cochain complex} $(C^*,d)$, i.e. a sequence of linear maps $C^0\stackrel{d_0}{\rightarrow} C^1\stackrel{d_1}{\rightarrow} C^2 \stackrel{d_2}{\rightarrow}\cdots$, satisfying $d_j\circ d_{j-1}=0$, see Definition~\ref{definCochainCpx}. If the cochain complex is chosen in a topologically correct way, then the sheaf Laplacian is defined by the operator $d_0^*d_0$, and cohomology $H^j(S;D)$ is isomorphic to $\Ker d_j/\im d_{j-1}$. \textbf{Discrete Hodge theory} claims that $\Ker(d_j^*d_j+d_{j-1}d_{j-1}^*)$ is isomorphic to $H^j(S;D)$. The operator $d_j^*d_j+d_{j-1}d_{j-1}^*$ appearing in this expression is called \textbf{higher order Laplacian}. Its spectral characteristics and the corresponding diffusion processes reflect both metric and topological properties of $S$ at the same time.

The story above provides the basic vocabulary of the mathematical side of sheaf theory. Recent advances in deep learning enriched applied sheaf theory with a number of important notions and ideas.

Given a sheaf $D$ on $S$, a single step of sheaf diffusion, may be used as a single layer or a building block in the construction of \textbf{sheaf-type neural networks}. Given a task formulated over a certain geometrical structure, one seeks a sheaf which can provide informative inductive bias for solving the task. One can either construct such a sheaf over the geometrical structure manually, use a predefined sheaf (if one exists), or perform \textbf{sheaf learning} to infer the sheaf structure from task data during training. To learn a sheaf, one often fixes a finite-dimensional family of potential candidate sheaves along with a suitable loss function which is minimized through gradient descent to arrive at a particular sheaf within this family, see subsection~\ref{subsecSheafLearning}. 

Sheaf-type neural networks are particularly effective for tasks on \textbf{heterophilic graphs} and are more resilient to \textbf{oversmoothing} compared to predecessor graph neural networks.


\subsection{Overall contribution of the paper}\label{subsecContribution}

We provide an informal and application-oriented exposition of sheaf theory in Section~\ref{secIntuitions}. This part is targeted at a general audience: it explains the basic philosophy behind the subject, its core concepts, what can and what cannot be done with sheaves. 

In order to make the exposition self-contained, we standardize basic terminology and provide a mathematical background for each part of Fig.~\ref{figPipeline}. This is done in the Appendices. This mathematical part is necessary to accurately transition between the various levels of generality at which sheaf theory may be expressed and applied. 

Sections~\ref{secReview} and~\ref{secReviewShvsML} act as a review of the papers on the applications of sheaf theory within computer science and deep learning. Many fascinating applications of sheaf theory in theoretical computer science and natural sciences have been proposed in the literature; Section~\ref{secReview} surveys this class of papers. We believe that some of these papers attracted less attention in applied sheaf community than they actually deserve; sheaf-theoretical ideas emerged in a variety of areas of computer science. We provide brief exposition of these areas, and how sheaves work in each particular case. 

More recent and data-driven applications of sheaves in deep learning are described in Section~\ref{secReviewShvsML}. Most applications of this sort are based, explicitly or implicitly, on the prior works in applied topology. This is the main reason why this section is placed after the description of non-DL applications. In this section we survey the papers introducing various sheaf-type neural networks as well as the papers that apply sheaf-type neural networks for solving practical problems. We concentrate on the verbal explanations of the novelty of such papers and reduce the number of formulas to a minimum: the precise details on implementation of architectures can be found in the original papers.

The content of this review motivates a number of open problems related to the application of sheaf theory which is the focus of Section~\ref{secProblems}. Some of these problems have emerged from our attempts to formalize commonly used practices, while others are natural questions which arose from mathematical and industrial practice. We address them towards a broad community of mathematicians, computer scientists, and specialists in geometric and topological deep learning. 

It is highly likely that the list of applications and ideas presented in our survey is incomplete. We would be happy to know if some important papers or ideas are missing, especially if they already address or solve the stated problems. 

As already mentioned above, most modern applications of sheaves are restricted to the class of cellular sheaves: sheaves defined on posets of cells of regular CW-complexes. In this case there is a canonical\footnote{up to change of cells' orientations} choice for a cochain complex, called the cellular cochain complex, that is suitable for computations of laplacians and cohomology. The mathematical contributions of this view were motivated, in part, by the question ``Does the pipeline of Fig.~\ref{figPipeline} make sense for arbitrary finite posets instead of CW-complexes?''. We answer this question in the affirmative and prove a number of results, mainly within algorithmic homological algebra. These results are present in Appendix~\ref{secMathCohomology}. The main theoretical contributions of the paper are described below.

We introduce the following notions.
\begin{itemize}
  \item Morse cell poset, Definition~\ref{definMorseHomologyCellPoset}.
  \item Cochain complex which honestly computes cohomology, Definition~\ref{definHonestlyComputes}.
  \item One-shot cohomology computation, Definition~\ref{definOneShotComplex}.
\end{itemize}
Every cell poset, in particular a poset of cells of a regular CW-complex, is a Morse cell poset. However, there exist geometrically meaningful examples of Morse cell posets which are not cell posets. See Example ~\ref{exMorseExample} for an example. We note the similarity of the two classical constructions:
\begin{enumerate}
  \item Cellular cochain complex defined over any cell poset (subsection~\ref{subsecMathCohomologyCellular}).
  \item Roos cochain complex (also called the standard simplicial resolution, or bar-construction) is defined over arbitrary posets (Subsection~\ref{subsecMathCohomologySimplicial}).
\end{enumerate} 
Both complexes honestly compute cohomology, but the visible advantage of cellular cochain complexes is that they provide one-shot computation.
\begin{enumerate}
  \item A poset $S$ supports a one-shot cohomology computation if and only if $S$ is a Morse cell poset, see Theorem~\ref{thmOneShotTheorem}.
  \item For any poset $S$ we prove a lower bound on complexity of sheaf cohomology computation (Theorem~\ref{thmMorseBound}) and prove that this bound is exact (Theorem~\ref{thmExactBound}) by providing the construction of a minimal cochain complex (Algorithm~\ref{algIncMatrixMain}). 
\end{enumerate}
The latter construction allows for the computation of sheaf cohomology over general finite posets in a way that is more optimal than with the Roos complex. In the case of cell posets (or Morse cell posets), the minimal cochain complex coincides with the cellular cochain complex, and the provided algorithm reduces to an algorithm which outputs incidence numbers between adjacent cells. In general, this algorithm defines an analogue to incidence numbers for non-cellular posets.

Finally, in Appendix~\ref{secMathLaplacians}, we develop a theory of sheaf Laplacians and diffusion in the generality of finite posets. It is shown that several known non-cellular laplacians, such as the hypergraph laplacian, fit in the constructed formalism.

\section{An intuitive introduction to sheaves}\label{secIntuitions}

\subsection{Geometrical structures}\label{subsecGeometryIsPoset}

\begin{defin}\label{definPoset}
A \textbf{partially ordered set (poset)} is a set $S$ together with a binary relation $\leq$, called the (non-strict) partial order, which satisfies
\begin{enumerate}
  \item $s\leq s$ for any $s\in S$;
  \item $s_1\leq s_2$ and $s_2\leq s_3$ imply $s_1\leq s_3$;
  \item $s_1\leq s_2$ and $s_2\leq s_1$ imply $s_1=s_2$.
\end{enumerate}  
\end{defin}

If $s_1\leq s_2$ and $s_1\neq s_2$ we simply write $s_1<s_2$.

We focus on posets as our primary geometric basis. There are several reasons for this choice.
\begin{itemize}
  \item Common geometrical data structures can be universally encoded as posets, see Examples~\ref{exGraphToPoset}--\ref{exGridToPoset} below.
  \item The majority of applied works covered in this review can be described as sheaf-like objects on posets.
  \item Sheaf theory in mathematics is originally formulated for topological spaces~\cite{BredonSheaves}. A finite topological space can be encoded as a finite poset and vice versa. See Proposition~\ref{propPosTop} for the precise equivalence.
\end{itemize}


\begin{ex}\label{exGraphToPoset}
A simple undirected \emph{graph} $G=(V,E)$ is treated as a poset $\Cells(G)=(V\sqcup E,\leq)$ where the partial order is given by inclusion: if $v\in V$, $e\in E$, then $v<e$ if $v$ is a vertex\footnote{A symbol $v\face e$ is used in modern applied papers in this context.} of $e$. See Figure~\ref{figPosets}.
\end{ex}

Simple graphs are restricted to two levels in their posetal order and are therefore limited in their ability to represent more complex, higher-order relationships that may exist within a space. To address this limitation, various discrete geometric structures have been suggested as a means to encode higher-order relationships in data, assuming that the data is supported on topological domains. These, too, may be represented as posets.

\begin{ex}\label{exHypergraphToPoset}
A \emph{hypergraph} $H$ on a vertex set $V$ is a collection of subsets of $V$ (which is commonly assumed to contain all singletons $\{v\}$ for $v\in V$). Then we have a partially ordered set $S(H)=(H,\subset)$ with the order given by inclusion. See Figure~\ref{figPosets}.
\end{ex}

\begin{ex}\label{exBinRelToPoset}
Another way to transform a hypergraph $H$ into a poset is similar to Example~\ref{exGraphToPoset}: we can assume all hyperedges incomparable and set $\bar{S}(H)=V\sqcup H$, with relation $v<h$ if a vertex $v$ belongs to a hyperedge $h$. More generally, one can take any \emph{binary relation} between two sets $R\subseteq A\times B$. This relation defines a strict order on $A\sqcup B$, if we assume that all elements of $A$ are pairwise incomparable and all elements of $B$ are pairwise incomparable. 
\end{ex}

\begin{ex}\label{exSimpCpxToPoset}
A \emph{simplicial complex} $\ca{K}$ on a finite vertex set $V=\{1,\ldots,m\}$ is a collection of subsets of $V$, which is closed under inclusion, i.e. $\sigma\in\ca{K}$ and $\tau\subset \sigma$ implies $\tau\in\ca{K}$. The elements of $\ca{K}$ are called simplices. The collection $\Cells(\ca{K})$ of non-empty simplices of $\ca{K}$ is partially ordered by inclusion. Any simplicial complex $\ca{K}$ defines a poset $\Cells(\ca{K})$. See Figure~\ref{figPosets}.
\end{ex}

\begin{ex}\label{ex2pts2edges}
There exist natural examples of ``point-free'' posets, which do not originate from partial order in a Boolean cube. For example, consider a finite set $\{1,2,a,b\}$ and impose a partial order by setting $1<a$, $1<b$, $2<a$, $2<b$.
\end{ex}

\begin{ex}\label{exCellCpxToPoset}
Consider a \emph{regular CW-complex} $\ca{X}$, see Definition~\ref{difinCWhausdorff}. The collection $\Cells(\ca{X})$ of its cells is partially ordered by inclusion. See Figure~\ref{figPosets}.
\end{ex}

\begin{ex}\label{exGridToPoset}
A pixel grid $P$ can be considered as a 2-dimensional CW-complex with square cells (\emph{quadrillage}). This gives rise to the poset $\Cells(P)$. This poset (and the corresponding Alexandrov topology) is the main object of study in digital topology~\cite{HermanBook,KovalevskyBook}
\end{ex}

Different geometric structures are well suited to represent different types of relationships. In particular, hypergraphs are well-suited to represent social networks ~\cite{amato2017influence}, where hyperedges encode groups of people. Hypergraphs have also been used to model the structure of thematic corpora in document analysis applications ~\cite{lee2024hints}. Simplicial complexes are used to represent coauthorship graphs ~\cite{ebli2020simplicial}, molecular graphs in computational chemistry and biology ~\cite{lan2023simplicial}, and in encoding 3D shapes as meshes ~\cite{lee2019implementation}. 

\begin{figure}
  \centering
  \includegraphics[scale=0.24]{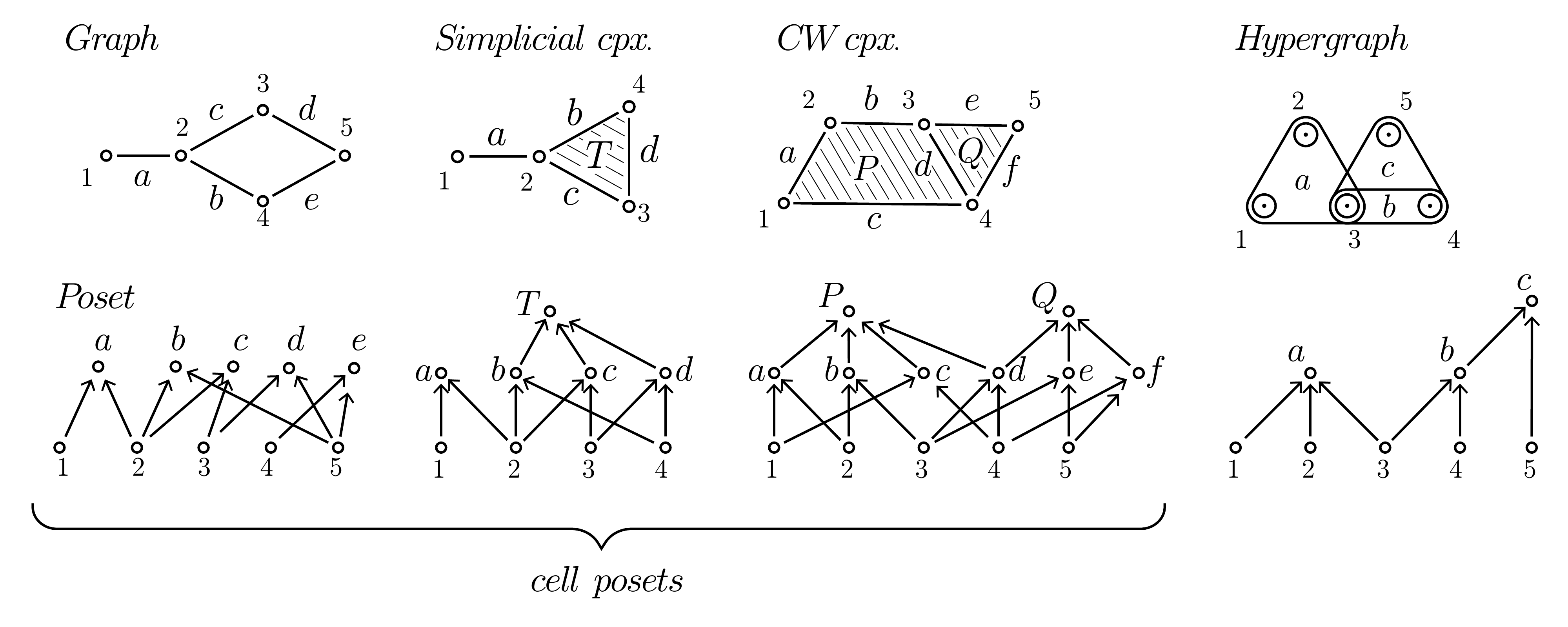}
  \caption{Common geometrical structures encoded as posets}\label{figPosets}
\end{figure}

\begin{con}\label{conHasseDiag}
It is convenient to represent finite posets graphically by means of Hasse diagrams. The \emph{Hasse diagram} $G(S)$ of a poset $S$ is a finite directed graph $(S,E)$ with $(s_1,s_2)\in E$ if and only if $s_1<s_2$ and there is no element $s'\in S$ such that $s_1<s'<s_2$. In other words, we represent a relation $s_1<s_2$ by an edge from $s_1$ to $s_2$ if there are no other intermediate elements between $s_1$ and $s_2$. The original order relation is recovered\footnote{In the finite case, of course.} as the transitive closure of $G(S)$: we have $s_1\leq s_2$ in $S$ if and only if there is a directed path from $s_1$ to $s_2$ in $G(S)$. 

The Hasse diagram is the particular case of transitive reduction of a graph known in computer science. Since transitive reductions are known to be the most memory-optimal data structures to represent a directed graph~\cite{TransitiveReduction}, Hasse diagrams are the most optimal structures to store posets. All figures in the paper show not the posets themselves (not all relations), but the corresponding Hasse diagrams.
\end{con}

\begin{rem}\label{remManyCellPosets}
There exists a subclass of particularly well-behaved posets, called \emph{cell posets}, see Definition~\ref{definCellPoset}. It serves as a finite data structure, which captures the topological intuitions about compact CW-complexes. The posets introduced in Examples~\ref{exGraphToPoset},\ref{exSimpCpxToPoset},\ref{exCellCpxToPoset},\ref{exGridToPoset} are cell posets, following Proposition~\ref{propCellPosetToCpx}. Examples~\ref{exHypergraphToPoset} and~\ref{exBinRelToPoset} are generally not cell posets. 
\end{rem}

\subsection{Exploration of geometrical structure}\label{subsecGeometryExploration}

\begin{ex}\label{exBat}
A bat navigates an environment by means of echolocation. How is it even possible to understand something about the global geometry of an environment by staying at a single point? There is an obvious answer. An environment is capable to conduct signals: there are PDE's telling how a signal at a single point propagates to infinitely close points. At a more fundamental level sheaves, as informally introduced in subsection~\ref{subsecFundNotions}, are involved. A sheaf\footnote{In the continuous mathematical model where a bat flies in a manifold $X\subseteq\Ro^3$, we deal with the sheaf of differential functions $\ca{O}^\infty(X)$.} is needed to mathematically model two claims: (1) every point of the environment is capable to store some information, or ``the state of matters'' in this point, (2) if points are close, there is a specified relation connecting their states.

A somewhat similar but a way more classical example is~\cite{HearDrum}, hearing the shape of a drum. Essentially, we speak about eigenvalues of Laplace--Beltrami operator on Riemannian manifold $X$. But on the more fundamental level we need a drum to store local information and conduct signals. This is why we are working not with the drum $X$ itself, but with a sheaf on $X$.  
\end{ex}

\begin{figure}
  \centering
  \includegraphics[scale=0.24]{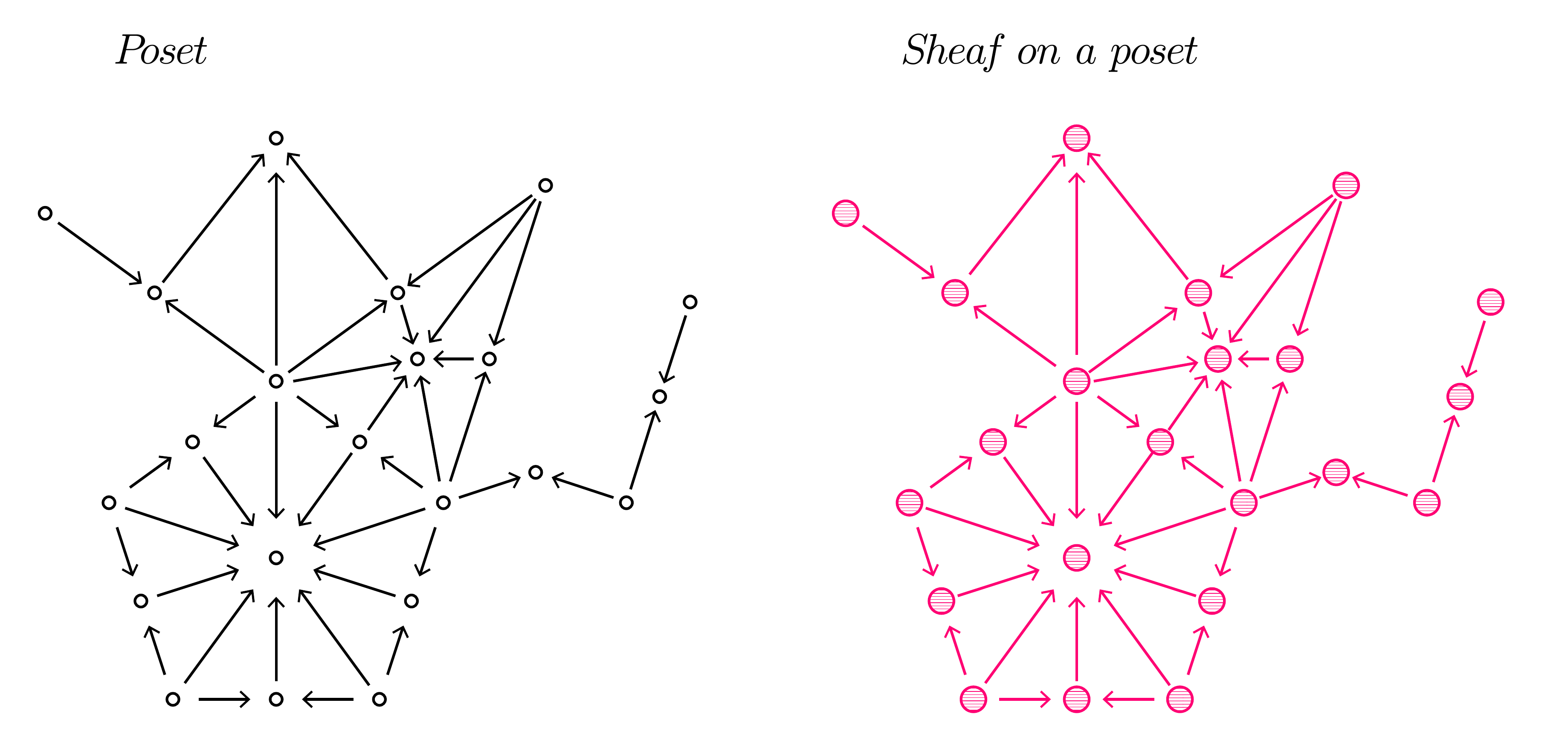}
  \caption{Poset is just a pure structure. Sheaf assigns some information container to each node and specifies how information flows between containers.}\label{figPosetSheaf}
\end{figure}

It is difficult to perform exact computations about continual spaces and differential sheaves so in practice we are interested in discretized problems. In the previous subsection we assumed that geometry is encoded in a (finite) poset. The point highlighted by the previous two examples stays valid for finite posets: in order to say something about a poset, we should first impose an additional structure on it, i.e. a sheaf, and then run some calculations over this structure.

\begin{con}\label{conInformalSheaf}
The information stored in a poset $S$ has purely syntactic nature: each claim $s_1<s_2$ means precisely ``there is a relation between $s_1$ and $s_2$''. In order to be able to say something more meaningful about the geometry of $S$, we impose some sort of semantics on $S$, and work with this semantics. This means that some mathematical entity $D(s)$ should be specified for every point $s\in S$, and, whenever $s_1<s_2$ in $S$, some specific relation $D(s_1<s_2)$ between the $D(s_1)$ and $D(s_2)$ is defined. Informally this means the following.
\begin{enumerate}
  \item The data type $D(s)$ prescribed to $s\in S$ serves as a local information container about possible ``states'' at the point $s$.
  \item The relation $D(s_1<s_2)\colon D(s_1)\to D(s_2)$ describes information flow or interaction between related points of $S$.
\end{enumerate}
See Figure~\ref{figPosetSheaf} as a general illustration. The compositionality is required from the information flow: whenever $s_1<s_2<s_3$, we should be able to compose the information flow $D(s_1<s_2)$ from $D(s_1)$ to $D(s_2)$ with the information flow $D(s_2<s_3)$ from $D(s_2)$ to $D(s_3)$ and get the information flow $D(s_1<s_3)$ from $D(s_1)$ to $D(s_3)$.

The most general framework formalizing these intuitions is \emph{category theory}~\cite{MacLane}. With a category $\Vv$ chosen, the data containers $D(s)$ are just objects of $\Vv$, and $D(s_1<s_2)\colon D(s_1)\to D(s_2)$ are morphisms of $\Vv$. Every poset $(S,\leq)$ may be treated as a syntactic category $\cat(S)$, whose objects are the elements of $S$, and morphisms are the relations $a\leq b$. The compositionality requirement states that $D$ should be a functor from $\cat(S)$ to $\Vv$, i.e. a diagram on $S$. In the case of posets, a sheaf is synonymous with a diagram: diagrams on $S$ correspond to sheaves on the related Alexandrov topology on $S$ (Proposition~\ref{propDiagSheaf}). Henceforth, the words `sheaf' and `diagram' are used interchangeably in the main part of the text.
\end{con}

\begin{defin}\label{definSheafDiagramMainText}
A $\Vv$-valued \emph{diagram}, or \emph{sheaf}, $D$ on a poset $S$ is a functor from $\cat(S)$ to $\Vv$. In more detail, a sheaf consists of
\begin{enumerate}
  \item an assignment, to each element $s\in S$ of an object $D(s)\in \Vv$, called the stalk of $D$ at $s$;
  \item an assignment, to each inequality $s<t$, of a morphism $D(s<t)\colon D(s)\to D(t)$, called the structure maps or restriction maps~\footnote{In applied papers these maps are called restriction maps, however this may create a confusion with the term used in the definition of a presheaf on a topology, see Remark~\ref{remPresheafInformally}.} of the sheaf $D$. Structure maps should satisfy the \emph{compositionality equations}: for any $s_0<s_1<s_2$, the equality\footnote{Here and below we use the notation for the composition borrowed from theoretical programming: $f;g=g\circ f$.} holds
      \begin{equation}\label{eqCompositionalityMainPart}
      D(s_0<s_1);D(s_1<s_2)=D(s_0<s_2).
      \end{equation}
\end{enumerate}
\end{defin}

For example, if $\Vv=\Sets$ is the category of finite sets, then each stalk is a set, and structure maps are mappings between sets. If $\Vv=\Ro\Vect$ is the category of real vector spaces, then each stalk $D(s)$ is just a vector space, and structure maps are linear maps of vector spaces. In finite-dimensional case each stalk corresponds to $\Ro^{d_s}$ and structure maps $D(s<t)$ are matrices of size $d_t\times d_s$.

\begin{ex}
If $G=(V,E)$ is a graph, and $S=\Cells(G)$ is the corresponding poset, defined in Example~\ref{exGraphToPoset}, then assigning a $\Ro\Vect$-valued sheaf $D$ on $S$ means that every vertex $v\in V$ is assigned a vector space $D(v)$, every edge $e\in E$ is assigned a vector space $D(e)$, and, whenever $v\in e$ (also denoted $v\face e$), there is a linear map $f_{v,e}\colon D(v)\to D(e)$.
\end{ex}


\begin{rem}
Imposing a $\Vv$-valued sheaf on $S$ makes it possible to borrow computational abilities of the target category $\Vv$ and use them to mine geometrical features of $S$. Sheaves algebraize geometry.
\end{rem}

The remark above is informal but motivates subsequent discussion providing evidence towards its veracity. 

\subsection{Global sections and coherent states}\label{subsecGlobalSectionsAndCoherence}

The choice of target category $\Vv$ has significant computational ramifications. The compositionality requirements imply a notion of equality in $\Vv$ and an ability to solve the system of equations defined by the underlying poset structure. This is covered in more detail in Construction~\ref{conOtherCategoriesInvLim}.

\begin{rem}\label{remConcrete}
In most applications, the target category is assumed \emph{concrete}~\cite[p.26]{MacLane}, meaning its objects are treated as sets and morphisms as structure-preserving maps between sets. In this case, the notion of choosing an element $x\in X$ of an object $X\in\ST{C}$ is a well-defined operation.
\end{rem}

The categories $\Sets$ and $\Ro\Vect$ are concrete. The category $\cat(S)$ of a poset is not concrete. To make it concrete would mean to define some good $\Sets$-valued sheaf on $S$. In the following, $\Vv$ denotes a concrete category.

\begin{defin}\label{definGlobalSectionsConcrete}
Consider a sheaf $D$ on a poset $S$. A \emph{global section} of $D$ is a choice, for each $s\in S$, of an element $x_s\in D(s)$, such that whenever $s_1<s_2$ in $S$, the following equality holds true
\begin{equation}\label{eqCoherentStates}
D(s_1<s_2)\colon x_{s_1}\mapsto x_{s_2}.
\end{equation}
The set of all global sections is denoted by $\Gamma(S;D)$. The set $\Gamma(S;D)$ is a subset of the Cartesian product $\prod_{s\in S}D(s)$. For each $t\in S$, there is a map $\Gamma(S;D)\to D(t)$ which picks the $t$-th component of a global section: $(x_s\mid s\in S)\mapsto x_t$. 
\end{defin}

\begin{figure}
  \centering
  \includegraphics[scale=0.24]{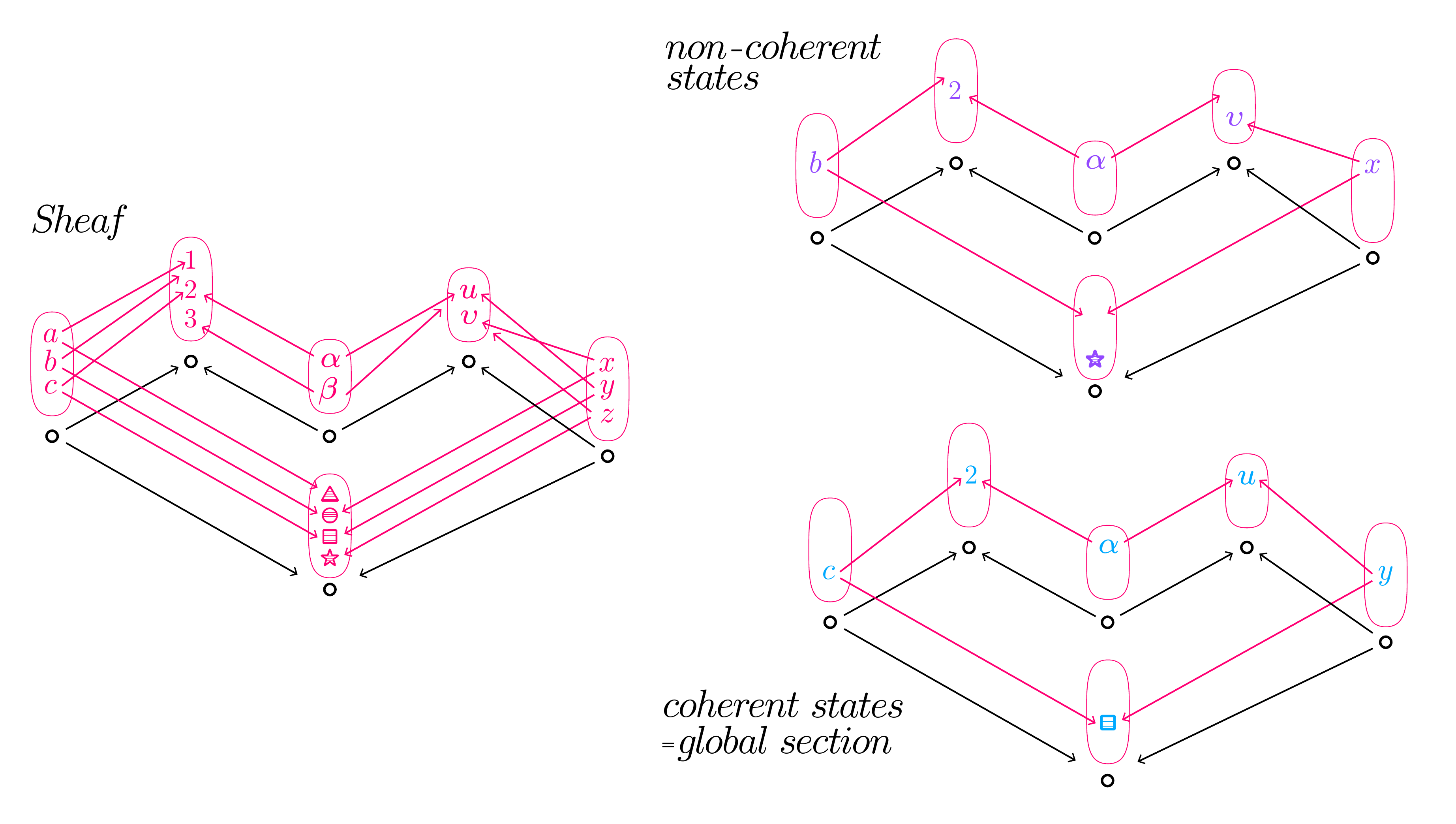}
  \caption{Example of non-coherent states of a given sheaf, and an example of coherent collection of states, i.e. global section.}\label{figGlobalSections}
\end{figure}

\begin{rem}\label{remCoherentStates}
The collection of elements $(x_s\mid s\in S)\in\prod_{s\in S}D(s)$ satisfying equations~\eqref{eqCoherentStates} may be interpreted as the choice of local states $x_s\in D(s)$ which are coherent with respect to information transfer rules provided by $D$, see Fig.~\ref{figGlobalSections}
\end{rem}

\begin{rem}\label{remReduceEquationsToGenerators}
In the system of equations defined by an arbitrary poset $S$~\eqref{eqCoherentStates}, we must only consider the equations corresponding to edges of the Hasse diagram $G(S)$ to determine global sections: the set of solutions will not change. Indeed, if $D(s_1<s_2)x_{s_1}=x_{s_2}$ and $D(s_2<s_3)x_{s_2}=x_{s_3}$ hold true, then the equality
\[
D(s_1<s_3)x_{s_1}=D(s_2<s_3)\circ D(s_1<s_2)x_{s_1}=x_{s_3},
\]
follows from the compositionality assumption~\eqref{eqCompositionalityEquations}. This remark is a straightforward way to reduce the number of equations defining a global section. A less straightforward way to simplify calculations even further is to utilize the local topological structure of $S$, e.g. if $S$ is a cell poset, see Remark~\ref{remCompatibilityReducesToCells}.
\end{rem}

\begin{rem}\label{remLocalSection}
The term ``global section'' suggests, there should be a notion of a ``local section''. It exists indeed, and it is quite natural: the coherent relations may be defined and satisfied only on a chosen subset $U$ of $S$. A collection $(x_s\mid s\in U)$ is called coherent on $U$ if it satisfies equations~\eqref{eqCoherentStates} restricted to $U$.

Notice that, whenever $x_s$ is chosen, equations~\eqref{eqCoherentStates} allow to compute $x_t=D(s<t)x_s$ for any $t>s$. For this reason, the subset $U$ in this construction is usually assumed to satisfy the assumption that whenever $s\in U$ and $t>s$, then $t\in U$. Such subsets are precisely the open subsets of the Alexandrov topology $X_S$ corresponding to $S$, see Construction~\ref{conDiagToSheaf}. In general, sheaves are defined on a topology --- with each open set $U\subset S$ we associate the set $\ca{F}(U)=\Gamma(U;S|_U)$ of sections coherent on $U$, see Construction~\ref{conDiagToSheaf}.
\end{rem}

\begin{con}\label{conOtherCategoriesInvLim}
Consider the category $\Ro\Vect$. All equations~\eqref{eqCoherentStates} are linear, so that $\Gamma(S;D)$ is a vector subspace in the direct sum $\bigoplus_{s\in S}D(s)$. This means $\Gamma(S;D)$ lives in the same category $\Ro\Vect$, it is not just a set, but a vector space itself. 

Instead of making such remark for any category $\Vv$ under consideration, a universal categorical construction of inverse limit is used for the formal definition of ``the global sections object'' in arbitrary category, 
see Constructions~\ref{conGlobalSections} and~\ref{conInverseLimit} and Remark~\ref{remConcreteLimSets}. The ability to take (finite) inverse limits in a category reduces to two basic operations: formation of finite direct products 
and formation of finite equalizers, i.e. the ability to solve equations. Both opportunities are present in the common categories, such as $\Sets$ or $\ko\Vect$, see Remark~\ref{remLimsToProductsEqualizers} for more details. From this perspective, all the subsequent constructions such as Laplacians and sheaf diffusion can be seen as numerical approaches to find the equalizer.
\end{con}


For the purpose of this brief exposition we restrict to the target category $\Vv=\Ro\Vect$ of real vector spaces.

\begin{con}\label{conConstantSheaf}
Consider a vector space $V\in\Ro\Vect$. The easiest and most natural way to define a sheaf on $S$ is to set all its stalks equal to the same object $V$, and all structure maps to be the identity maps $\id_V$. This sheaf is denoted $\bar{V}$ and called \emph{the constant sheaf}. Such sheaves are frequently used in algebraic topology.
\end{con}

\begin{ex}\label{exConsensusAlongApath}
Consider a path graph $G_p$ and its corresponding poset $S_p=\Cells(G_p)$ see Fig.~\ref{figConstSheaves}, left. The constant sheaf $\overline{\Ro^1}$ is shown schematically on the same figure. In this example the vector space of global sections is the diagonal subspace
\[
\Gamma(S_p;\overline{\Ro^1})=\left\{(x,x,\ldots)\in\prod\nolimits_{s\in S_p}\overline{\Ro^1}(s)\right\}\cong \Ro^1.
\]
Indeed, the states are coherent if and only if they coincide on the endpoints of each arrow of the Hasse diagram, since all structure maps of $\overline{\Ro^1}$ are just the identity maps.
\end{ex}

\begin{figure}
  \centering
  \includegraphics[scale=0.23]{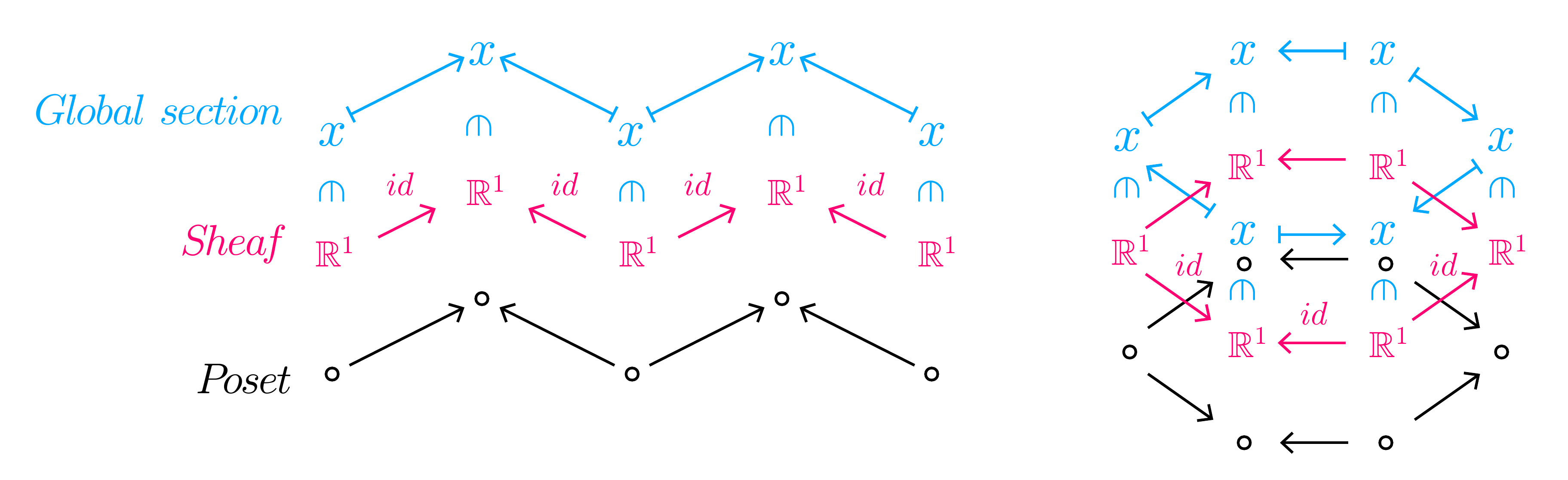}
  \caption{Constant sheaves over a path graph, and over a cycle graph, and their global sections.}\label{figConstSheaves}
\end{figure}

\begin{ex}\label{exConsensusAlongAcircle}
Consider a similar example with a cycle graph $\Cy_n$ on $n$ vertices, the corresponding poset $S_n=\Cells(\Cy_n)$ and, the constant sheaf $\overline{\Ro^1}$ as before, see Fig.~\ref{figConstSheaves}, right. Again, we have $\Gamma(S_n;\overline{\Ro^1})\cong \Ro^1$.
\end{ex}

\begin{ex}\label{exConsensusGraphCase}
In general, if $G$ is a graph and $\overline{\Ro^1}$ is a constant sheaf on $\Cells(G)$, then we have
\begin{equation}\label{eqNumberOfConComponents}
\Gamma(\Cells(G);\overline{\Ro^1})\cong\Ro^{c(G)},
\end{equation}
where $c(G)$ is the number of connected components of $G$. Indeed, the coherence relations~\eqref{eqCoherentStates} imply the equality of states over graph elements (vertices and edges) from the same connected component. Distinct connected components may have distinct coherent states.

More generally $\dim\Gamma(S;\overline{\Ro^1})$ equals the number of connected components of the geometric realization $|S|$ of $S$, see Definition~\ref{definGeomRealPoset} and Example~\ref{exConstantSheaf}.
\end{ex}

\begin{figure}
  \centering
  \includegraphics[scale=0.25]{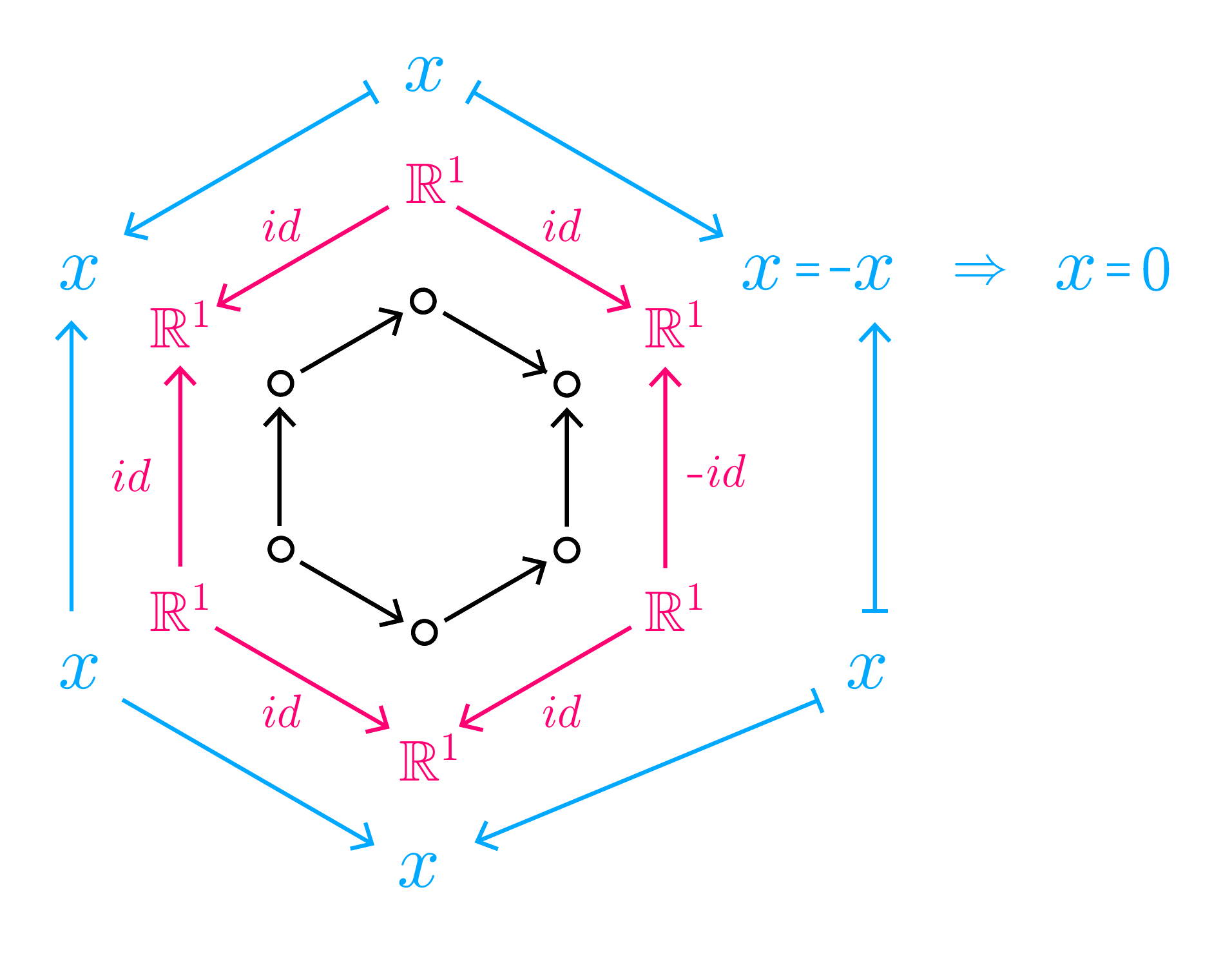}
  \caption{M\"{o}bius sheaf on a cycle graph.}\label{figMoebius}
\end{figure}

\begin{ex}\label{exConsensusAlongMoebius}
Consider the same cycle graph $\Cy_n$ as in Example~\ref{exConsensusAlongAcircle}, but this time with the different sheaf $D_\mu$. All stalks of $D_\mu$ are $\Ro^1$ and all maps but one are identity maps. The exceptional map is the multiplication by $-1$, see Fig.~\ref{figMoebius}. We can call it a Möbius sheaf, as the analogy with the Möbius strip is transparent. In this case, the system of coherence equations~\eqref{eqCoherentStates} only has a trivial zero solution: propagating the equality along the circle, we end up with the equation $c=-c$, which has a unique solution $c=0$ (since $\Ro$ is not a field of characteristic $2$). Therefore $\Gamma(S_n;D_\mu)=0$.
\end{ex}

\subsection{Consensus higher and wider}\label{subsecConsensusGoesWild}

We have seen that the notion of a global section allows one to reason about coherent (or consensual, or compatible) states of some system over a geometrical structure. This connection of possible system state with system structure is encoded by a sheaf. 

However, the collection $\Gamma(S;D)$ of global sections, as a stand-alone object, is not very informative. Examples~\ref{exConsensusAlongApath} and~\ref{exConsensusAlongAcircle} demonstrate that the difference between path-graph and cycle-graph can't be seen by only observing the consensus states of the respective constant sheaves. The two classical ways to squeeze more information from the consensus equations are the following.
\begin{enumerate}
  \item Estimate how ``proofs of consensus'' (equations~\eqref{eqCoherentStates}) relate to each other. We have not only proofs of consensus, but also proofs between proofs, proofs between proofs between proofs, etc. This naturally leads to the notion of higher-order cohomology $H^*(S;D)$ of a sheaf, allowing one to estimate qualitative (i.e. topological) properties of $S$. This cohomological framework could potentially be replaced with higher categories and homotopy theory, but this latter framework is syntactical and algorithmically undecidable problems quickly emerge, such as the word problem. Although complicated, homological algebra is algorithmically feasible.
  \item Use ``finding consensus'' as an algorithmic prior. We can estimate the speed and direction of convergence to consensus using the spectral theory of Laplacians. Or we can use the process of reaching the consensus (or the combination of several such processes) to model unknown dynamical systems on a geometrical structure, leading to the collection of sheaf neural network architectures which have been proposed within the deep learning literature.  
\end{enumerate}

These two items are covered in more detail in the next subsections.

\subsection{Sheaf cohomology}\label{subsecCohomologyMainPart}

The general idea of cohomology can be understood by analyzing the difference between Examples~\ref{exConsensusAlongApath} (consensus on a segment) and~\ref{exConsensusAlongAcircle} (consensus on a circle).

\begin{ex}\label{exPathVsCircle}
Informally, the difference between a segment and a circle lies in the number of substantially different ways the consensus is propagated. In the path graph, the consensus between any two nodes $s_0$ and $s_n$ propagates along a single path. A value $x\in\Ro^1=\overline{\Ro^1}(s_0)$ can be copied into to the stalk $D(s_n)$ by the consecutive solution of linear equations $x_s=x_t$ with $s<t$ along the unique path $s_0<t_1>s_1<t_2>s_2<\cdots>s_n$. In the cycle case, given a value $x\in\Ro^1=\overline{\Ro^1}(s_0)$, another point $s_n$ receives the signal from two sides. It is the same signal ``Set $x_{s_n}$ equal to $x$'', however, it arrives at $s_n$ by two distinct routes: clockwise and counterclockwise, see Fig.~\ref{figConstSheaves} for reference.

Algebraically, this intuition can be quantified by the fact that, in the cycle, there are redundant linear equations among~\eqref{eqCoherentStates}--the coherence equations have dependencies. These dependencies may be formalized with higher-order cohomology. We have
\[
H^1(S_p;\overline{\Ro^1})=0,\mbox{ for a path, and } H^1(S_c;\overline{\Ro^1})\cong\Ro^1\mbox{ for a cycle.}
\]
In other words, higher cohomology of the constant sheaf can distinguish a path graph from a cycle graph.
\end{ex}

Relations between relations are sometimes called \emph{syzygies}. To have a notion of a syzygy, one needs the target category $\Vv$ to be an abelian category. Abelian categories integrate nicely into the toolbox of linear and homological algebra. Sufficiently well-behaved abelian categories such as the category $\ko\Vect$ of vector spaces over a fixed ground field $\ko$, or algebras, or sheaves, support the machinery of derived functors, introduced in mathematics by Cartan and Eilenberg~\cite{careil1956homalg}.

\begin{defin}\label{definCohomologyMainPart}
Let $S$ be a poset, $D$ a sheaf on $S$ valued in a sufficiently nice abelian category $\Vv$, and $j\geq 0$ a nonnegative integer. Then \emph{the cohomology} $H^j(S;D)\in\Vv$ is the \emph{$j$-th right derived functor} $R^j\Gamma(S;D)$ of the functor of global sections.
\end{defin}

In most deep learning-related papers, a different definition of cohomology can be found: the one based on cellular cochain complex. It is only applicable to sheaves over CW-complexes, but in this case it coincides with the above definition (see Theorem~\ref{thmCWcohomologyIsAGcohomology}). The posetal perspective employed thus far is more general than the CW-complex case, so we provide the above fundamental definition of cohomology and move all associated details to Appendix~\ref{secMathCohomology}.  

\begin{con}\label{conCochainMainPart}
Despite its fundamental theoretical importance, Definition~\ref{definCohomologyMainPart} is inconvenient from computational perspective. In order to actually compute sheaf cohomology, one needs to associate, with a sheaf $D$, a cochain complex $(C^*(S;D);d)$ (see Definition~\ref{definCochainCpx}):
\begin{equation}\label{eqCochainSheafMainPart}
0\stackrel{d}{\rightarrow} C^0(S;D)\stackrel{d_0}{\rightarrow} C^1(S;D)\stackrel{d_1}{\rightarrow} C^2(S;D)\stackrel{d_2}{\rightarrow} \cdots,
\end{equation}
with a theoretical guarantee that its cohomology
\[
H^j(C^*(S;D);d)=\Ker d_j\colon C^j(S;D)\to C^{j+1}(S;D)/\im d_{j-1} C^{j-1}(S;D)\to C^j(S;D)
\]
is isomorphic to the sheaf cohomology $H^j(S;D)$ as given by Definition~\ref{definCohomologyMainPart}. If this is the case, we say that the cochain complex $C^*(S;\cdot)$ \emph{honestly computes cohomology} of sheaves on $S$, see Definition~\ref{definHonestlyComputes}. 

There are infinitely many cochain complexes which honestly compute cohomology, the examples can be found in subsections~\ref{subsecMathCohomologySimplicial} and~\ref{subsecMathCohomologyCellular}. Luckily, there is a convenient criterion which checks if a given cochain complex honestly computes sheaf cohomology, given by the famous Grothendieck's universal delta-functor theorem~\cite{grothendieck1957tohoku}. We reformulate it in a less technically loaded manner as Theorem~\ref{thmHonestCohomology}.
\end{con}

Cohomology $H^j(S;D)$ in degrees $j>0$ is typically not found in deep learning applications outside of the computation of persistent homology. The $0$-th cohomology $H^0(S;D)$ coincides with the global sections $\Gamma(S;D)$ of a sheaf, and this is typically the cohomological information employed by most applications. The cochain complex~\eqref{eqCochainSheafMainPart} is important for the actual computation of global sections, since any complex honestly computing cohomology satisfies
\[
\Gamma(S;D)=H^0(S;D)\cong H^0(C^*(S;D),d)=\Ker d_0\colon C^0(S;D)\to C^1(S;D).
\]
Although only the components of degrees $0$ and $1$ are involved in the above formula, the choice of the suitable cochain complex affects the construction of objects like the sheaf Laplacian and consequently influences signal processing applications based on this spectral information.

The construction of a cochain complex that honestly computes cohomology constitutes a nontrivial procedure and heavily depends on the type of geometrical structure $S$.

\begin{ex}\label{exCellPosetsSupportCanonicalCCpx}
Cell posets $\ca{X}$ (see Remark~\ref{remManyCellPosets}) are considered the most convenient class of discrete geometrical structures, because, among all cochain complexes honestly computing cohomology of sheaves on $\ca{X}$, there exists the distinguished one --- \emph{the cellular cochain complex} $C_{CW}^*(\ca{X};D)$ (Definition~\ref{definCWcochainComplex}). Cell complexes are used, among other applications, in the definitions of the cellular sheaf Laplacian, commonly used in deep learning applications. See Construction~\ref{conCellComplexDiffusion} for further discussion.

A similar construction, called a Morse cochain complex, is suggested in subsection~\ref{subsecMathOneShot}, which generalizes cellular cochain complexes to more general posets. The construction, as its name implies, is motivated by Morse theory.
\end{ex}

\begin{ex}\label{exHypergraphsRoosCpx}
Hypergraphs, viewed as two-layered posets like in Example~\ref{exBinRelToPoset}, are not cell posets. A cellular cochain complex does not exist for such structures because incidence numbers cannot be canonically defined for a hypergraph as is natural for graphs. Nevertheless, there is a classical construction of Roos complex, also known as standard simplicial resolution, which honestly computes cohomology (see Theorem~\ref{thmRoosIsAGcohomology}) on any poset, including those derived from hypergraphs. It can be shown that the definition of the hypergraph sheaf Laplacian, appearing in the recent papers on hypergraph sheaf networks, are basically the Laplacians defined over Roos complexes. See Construction~\ref{conHypergraphDiffusion} for further discussion.
\end{ex}

It should be mentioned that Roos complexes are sub-optimal for computing cohomology and for defining Laplacians due to their large size. Other cochain complexes can be used, leading to matrices of smaller size. This is an important issue in practice as well: both sheaf neural networks and more general cohomology computations suffer from large computational complexity. In subsection~\ref{subsecMathMinimalComputations} we prove the lower bound for the complexity of cochain complexes honestly computing cohomology, and propose a new construction of the minimal cochain complex $C^*_{\min}(S;D)$ which is well-defined for any finite poset $S$ and achieves the lower complexity bound. 

\subsection{Laplacians and heat diffusion}\label{subsecLaplaciansMainPart}

The general utility of Laplacians can be understood with a basic example.

\begin{ex}\label{exConsensusAlongVariousCycles}
Consider a poset $S_n$ corresponding to a cycle graph $\Cy_n$ on $n$ vertices, as in Example~\ref{exConsensusAlongAcircle}, and the constant $\Ro^1$-valued sheaf on $S_n$. Its cohomology is given by
\[
(H^0(S_n;\Ro^1)\cong \Ro^1,H^1(S_n;\Ro^1)\cong \Ro^1,0,0,\ldots)
\]
which is independent of the number of vertices composing the cycle. This is not surprising: all graphs $\Cy_n$ are homeomorphic, and have the same qualitative properties. 

One distinguishing feature between the graph $\Cy_3$ and the graph $\Cy_{100}$ is the time needed for a signal to spread across each graph. The process of signal propagation is modeled by heat diffusion (Definition~\ref{definHeatDiffusion}). The convergence rate of the process is determined by the smallest positive eigenvalue $\lambda_{\min}$ of the Laplacian matrix of the graph: the bigger $\lambda_{\min}$ the faster the process converges, as explained in Subsection~\ref{subsecMathDiffusion}. It is known~(see e.g.~\cite[ex.1.5]{Chung}) that the (unnormalized) Laplacian of $\Cy_n$ has eigenvalues $2-2\cos\frac{2\pi k}{n}$, for $k=0,1,\ldots,n-1$. Hence we have
\[
\lambda_{\min}(\Cy_n)=2-2\cos\frac{2\pi}{n} \approx \frac{C}{n^2}
\]
which is smaller for $n=100$ than for $n=3$. The heat spreads faster over a cycle of smaller length--an intuitive result.
\end{ex}

\begin{figure}
  \centering
  \includegraphics[scale=0.24]{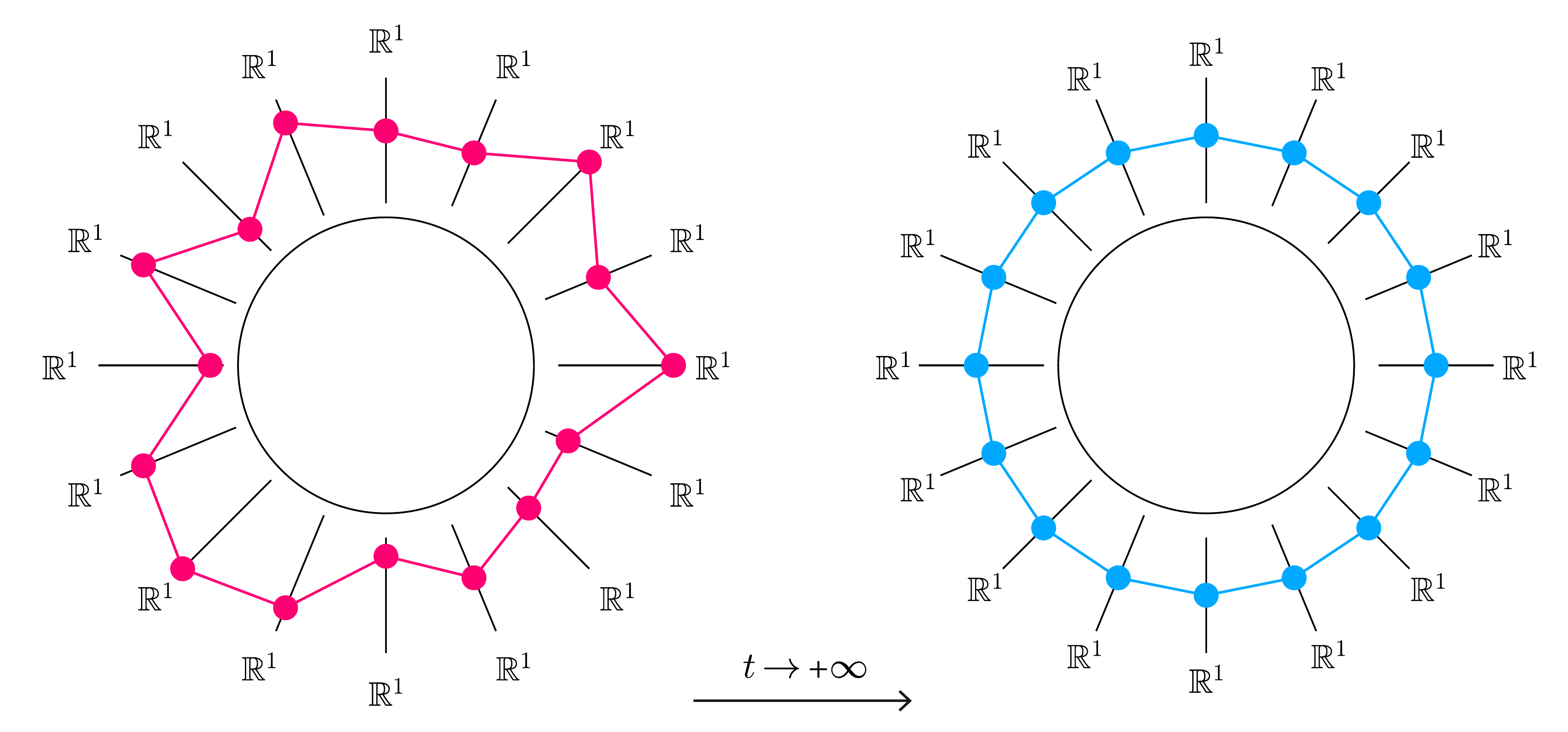}
  \caption{Heat diffuses in a constant sheaf, resulting in a system of coherent states}\label{figDiffusion}
\end{figure}

Let $D$ be $\Ro\Vect$-valued sheaf on $S$. Laplacians appear as a byproduct of solving the equations on global sections of $D$. Assume that some cochain complex $(C^*(S;D),d)$ (see~\eqref{eqCochainSheafMainPart}) honestly computes cohomology of $D$. Then the computation of $j$-th cohomology reduces to finding the subquotient $\ker d_j/\im d_{j-1}$ for the sequence
\begin{equation}\label{eqShortExactMainPart}
C^{j-1}(S;D)\stackrel{d_{j-1}}{\rightarrow} C^j(S;D)\stackrel{d_j}{\rightarrow} C^{j+1}(S;D).
\end{equation}
If inner products are chosen on the vector spaces $C^i(S;D)$, we can construct the sheaf Laplacian.

\begin{defin}\label{definSheafLaplacianMainPart}
\emph{The sheaf Laplacian} of $j$-th order of sheaf $D$ on a poset $S$ (defined on a suitable cochain complex $(C^*(S;D),d)$) is the operator
\[
\Delta_j=d_j^*d_j+d_{j-1}d_{j-1}^*\colon C^j(S;D)\to C^j(S;D).
\]
The subspace $\ca{H}^j(S;D)=\Ker\Delta_j\subseteq C^j(S;D)$ is called \emph{the subspace of harmonic cochains} of $D$.
\end{defin}

Here $A^*$ denotes the adjoint of $A$, or the transposed matrix $A^\top$ if an orthogonal basis is chosen. The basic linear algebra implies the isomorphism $\ca{H}^j(S;D)\cong H^j(S;D)$, see Corollary~\ref{corHarmonicIsCohomology}. 

\begin{rem}\label{remVanillaLaplacian}
When $j=0$, we have $\Delta_0=d_0^*d_0$, since there is no component of degree $-1$ in the cochain complex. In this case we have
\[
\ca{H}^0(S;D)=\Ker\Delta_0=\Ker d_0=H^0(S;D)\cong \Gamma(S;D).
\]
The subspace of global sections is isomorphic to, and coincides with, the subspace of harmonic $0$-cochains, which is the zero eigenspace of $\Delta_0$. In most sources the term ``sheaf Laplacian'' refers to $\Delta_0$. We call $\Delta_0$ \emph{the vanilla sheaf Laplacian}, to distinguish it from higher order sheaf Laplacians $\Delta_j$ with $j>0$ described above.
\end{rem}


\begin{con}\label{conSheafEnergyMainPart}
The vanilla sheaf Laplacian $\Delta_0$ is a nonnegative self-adjoint operator. It is diagonalizable with all eigenvalues nonnegative (Subsection~\ref{subsecMathHodge}). Henceforth the quadratic function $Q_{\Delta_0}(x)=\langle \Delta_0x,x\rangle$ defined on $C^0(S;D)$ is nonnegative and convex. The function $Q_{\Delta_0}(x)$ is called the (\emph{Dirichlet}) \emph{energy function} of the sheaf $D$. The zero-set of $Q_{\Delta_0}$ coincides with its set of points of local (or global) minima and also coincides with the harmonic subspace $\ca{H}^0(S;D)=\Ker \Delta_0$ isomorphic to the space $\Gamma(S;D)$ of global sections.
\end{con}

\begin{defin}\label{definHeatDiffusion}
The flow of gradient descent, either discrete or continuous, of the energy function $Q_{\Delta_0}$ on the euclidean space $C_0(S;D)$ is the (\emph{$0$-th order}) \emph{heat diffusion equation} on the sheaf $D$.
\end{defin}


In practice, the normalized version of the Laplacian (Construction~\ref{conNormalizationOfLaplacian}) is used in calculating diffusion. Independently of the initial point, the flow converges to a harmonic cochain in $C_0(S;D)$. The convergence rate of heat diffusion is determined by the eigenvalues of $\Delta_0$.

\begin{con}\label{conLambda1asymptotics}
Let $\lambda_{\min}$ denote the minimal positive eigenvalue of $\Delta_0$. For an initial point $x_0$ let $x_{+\infty}$ denote the limiting harmonic cochain of the heat diffusion process starting at $x_0$. Then, for generic initial value $x_0$, the distance $\|x(t)-x_{+\infty}\|$ is asymptotically equivalent to $Ce^{-2\lambda_{\min}\eta t}$ --- in the case of continuous heat diffusion and  $(1-2\eta\lambda_{\min})^k$ --- in the discrete case. See Construction~\ref{conSolutionDiagonalized} for the proof. The remaining positive eigenvalues of the Laplacian $\Delta_0$ measure the lower asymptotic terms of the heat diffusion.
\end{con}

Eigenvalues of sheaf Laplacians provide real numerical characteristics which can be used in applications or as auxiliary information for various machine learning algorithms. These characteristics are interpretable in the sense that they have clear geometrical and physical meaning. A warning about spectral determinability is warranted, however.

\begin{rem}
A Laplacian itself is not an invariant of a sheaf $D$. First of all, to make things well-defined one needs to impose inner product on all the stalks $D(s)$ --- the construction concerns \emph{Euclidean sheaves}, not sheaves in general. This step in the definition is often skipped because the vector spaces in most applications are just the coordinate vectors spaces $\Ro^k$; the basis comes for free and is assumed orthonormal. The second and probably even more important remark is that the Laplacian is defined on a particular cochain complex $(C^*(S;D),d)$. As discussed in Construction~\ref{conCochainMainPart}, there are many cochain complexes which honestly compute cohomology of $D$ (and many more cochain complexes which compute something else!). Which complex to choose to define Laplacians is a modeling problem. If $S=\ca{X}$ is a cell poset, there is a canonical choice of the cellular cochain complex $(C^*_{CW}(S;D);d_{CW})$, as in Example~\ref{exCellPosetsSupportCanonicalCCpx}. The term \emph{sheaf Laplacian} usually refers to the Laplacian defined on this particular cochain complex. The ambiguity in the definition is the choice of cells' orientations: luckily, it is easily proved that Laplacians, up to conjugation, do not depend on the orientations of cells. 
For hypergraphs, Roos complexes may be chosen to define a Laplacian, as described in Construction~\ref{conHypergraphDiffusion}. 
\end{rem}

\subsection{Sheaf learning}\label{subsecSheafLearning}

Using sheaf theory as a descriptive language to speak algebraically about geometrical objects is prominent in both theoretical mathematics and computer science. In this subsection we show how the language of sheaf theory has recently fused with deep learning. The following specific constructions mirror the specific implementations of sheaf theory that have, thus far, been applied in deep learning applications.

\begin{defin}\label{definSheafDiffusion}
Let $\Delta_0=d_0^*d_0\colon C^0(S;D)\to C^0(S;D)$ be the vanilla sheaf Laplacian of a sheaf $D$ on $S$, and $\eta>0$ a hyperparameter. Then the generator of the discrete heat diffusion, i.e. the map
\[
\sd_D\colon C^0(S;D)\to C^0(S;D),\qquad \sd_D(x)=x-2\eta \Delta_0x
\]
is called \emph{the sheaf diffusion} determined by $D$.
\end{defin}

Heat diffusion on a graph well known in spectral graph theory~\cite{Chung} is the instance of this general definition as the following example shows. 

\begin{ex}
Let $S=\Cells(G)$ be a graph (the poset associated with a graph $G=(V,E)$ as in Example~\ref{exGraphToPoset}), and $\overline{\Ro^1}$ be the constant sheaf on $S$, so that the vector space $C^0(S;\overline{\Ro^1})=\bigoplus_{v\in V}\Ro^1$ is freely spanned by the vectors $\varepsilon_v\in V$. This space can be understood as assigning a single real value, like the temperature, to each vertex of a graph. In this case, the map $\sd_D$, for the unnormalized Laplacian, is defined on the basis vectors by
\[
\sd_D(\varepsilon_v)=\varepsilon_v(1-2\eta \deg v)-\sum\nolimits_{u\sim v}\eta \varepsilon_u,
\]
where $\deg v$ is the degree of the vertex $v$ and the summation is taken over all vertices $u$ adjacent to $v$. The diffusion operates as averaging the temperature of $v$ and all its neighbors.
\end{ex}

\begin{defin}\label{definSheafNN}
A \emph{sheaf diffusion layer} determined by a sheaf $D$ is the linear map $\sd_D\colon C^0(S;D)\to C^0(S;D)$. A \emph{sheaf-type neural network} is an architecture that uses sheaf diffusion layers, probably in combination with other types of layers, in its construction.
\end{defin}

Sheaves $D_i$ (the collection of matrices $D_i(s<t)$) used in the layers are sometimes learned during the training of the neural network.

\begin{ex}\label{exShNNlinear}
The basic example of a sheaf-type neural network is given by a composition $F_{D_1,\ldots,D_l}=\sd_{D_1};\cdots;\sd_{D_l}$ of sheaf diffusion layers determined by a finite sequence of sheaves. This network, as written, is not particularly expressive since the map $F_{D_1,\ldots,D_l}$ is linear. In practice, some non-linear operators (e.g. coordinate-wise sigmoidal or ReLU operations) are inserted between sheaf diffusion layers to make the resulting function class of the network more expressive.
\end{ex}

\begin{rem}\label{remNNdynamical}
Sheaf-type neural networks are related to the mathematical theory of heat diffusion in a similar manner to how transformer networks are related to dynamical properties of Kuramoto models~\cite{TransformersKuramoto} and Vlasov equations~\cite{VlasovTransformers}. In both cases, if all layers of a network are the same, and there are infinitely many infinitesimal layers, the network becomes a mathematically reasonable object: a dynamical system with nice asymptotical properties. It should be noted, however that this perspective implies sheaf-type neural networks are much simpler than transformers because they correspond to linear dynamical systems, while the mathematical models of transformers are non-linear. The applied potential of sheaf networks seem to arise from the fact that the domain $C^0(S;D)$ of their definition has a rich semantical structure determined by the underlying geometry of $S$. 
\end{rem}

\begin{rem}\label{remMessagePassing}
Sheaf diffusion on a graph may be seen as a form of linear \emph{message passing}. In general, message passing updates the state $x_v$ of a vertex by the formula
\begin{equation}
x_v^{(t)} = \phi\left(x_v^{(t-1)}, \biguplus\limits_{u \in N_v}\psi(x_u^{(t-1)}, x_v^{(t-1)}, w_{uv}^{(t-1)}) \right)
\end{equation} 
where $x_v^{(t)} \in \Ro^d$ is the subsequent state of $v$, $N_v$ is the graph neighborhood i.e. the set of nodes adjacent to $v$, $w_{uv}^{(t-1)}$ is a learnable vector associated to an edge $\{u,v\}$, $\psi$ and $\phi$ are fixed piecewise differentiable functions called \emph{message} function, and \emph{update} function respectively, and $\biguplus$ is a permutation invariant aggregation function\footnote{It is common in machine learning to use symbol $\bigoplus$ for such aggregation functions, however this usage contradicts the more classical notation of a direct sum in abelian categories, which plays an important role in the mathematical part of a paper.} of arbitrary number of inputs, e.g. coordinate-wise sum or max or similar. General graph neural networks consist of several layers performing message passing.   

Message passing can be defined in a similar way on arbitrary posets $S$ instead of graphs as long as the neighborhood relation between elements of $S$ is somehow defined from the structure of the poset. For example, if $S$ is a cell poset, then two $k$-dimensional cells $\sigma,\tau$ may be called adjacent if they either cover a common $(k-1)$-cell, or are covered by a common $(k+1)$-cell. In this case, heat diffusion for a higher-order Laplacian given by Definition~\ref{definSheafLaplacianMainPart} becomes an example of linear message passing. Simply speaking, the neighborhood is defined by the 2-hop distance on the Hasse diagram of $S$, and the message passing between cells is performed by summation of matrix products over all such 2-hop routes. General constructions of message passing of this sort can be found in~\cite{TDLbeyond} or~\cite{DemystifyingHOGNN}.
\end{rem}

\subsection{Neural sheaf diffusion}

In practice, sheaf-type neural networks combine the mechanism of sheaf diffusion with other standard engineering practices, such as adding intermediate learnable weight matrices, adding auxiliary channels, and reducing the number of parameters by picking a suitable non-linear submanifold in the manifold of all sheaves.

\begin{con}\label{conNeuralSheafDiffusion}
The \emph{neural sheaf diffusion} (NSD) model~\cite{bodnar2022neural} learns a sheaf $D$ on a graph $G$ with $n$ vertices and all vertex-stalks $D(s)$  isomorphic to $\Ro^d$. Additionally there are $f_1$ input channels and $f_2$ output channels. A single layer of neural sheaf diffusion is given by the function $F\colon C^0(S;G)\otimes \Ro^{f_1}\to C^0(S;G)\otimes \Ro^{f_2}$ defined by
\begin{equation}\label{eqNeuralSheafDiff}
F=\sigma\circ ((\sd_D\circ W_1^{\oplus n})\otimes W_2),
\end{equation} 
where $W_2\colon \Ro^{f_1}\to \Ro^{f_2}$ is a learnable weight matrix mixing the channels, $W_1\colon \Ro^d\to \Ro^d$ is a learnable matrix applied independently in each vertex-stalk, $\sd_D$ is the diffusion operator as in Definition~\ref{definSheafDiffusion}, and $\sigma$ is the coordinate-wise sigmoid, providing nonlinearity. 
\end{con}

\begin{rem}\label{remGCN}
If $D$ is the constant sheaf $\overline{\Ro^1}$, $W_1=\id$, then $W_2$ is the only parameter collection to be learned. In this case formula~\ref{eqNeuralSheafDiff} defines the layer of the classical Graph Convolution Network (GCN) of Welling and Kipf~\cite{kipf2017semisupervised}.
\end{rem}

One can estimate space and time complexity of sheaf-type neural networks, in particular, the complexity of NSD.

\begin{con}
Consider a graph with $m$ edges and $n$ nodes and NSD layer over this graph with stalk dimension $d$, $f_1$ input channels and $f_2$ output channels. Let us introduce $c = f_1 \cdot d$. If we were to learn the matrices $W_1$, $W_2$ and all the restriction maps independently, the number of parameters (i.e. the space complexity) would be $d^2+f_1f_2+2md^2$ as follows easily from Construction~\ref{conNeuralSheafDiffusion}. This complexity depends linearly on the size of the graph. To overcome this issue, the paper~\cite{bodnar2022neural} proposes to learn not every restriction map separately, but a single function $F\colon (\Ro^c\oplus \Ro^c)\to \Ro^{d^2}$ of the form $F(x_v, x_u) = \sigma(V(x_v\oplus x_u))$ with the learnable matrix $V$ of size $dc\times d^2$, where $x_v\oplus x_u$ is the concatenation of the states of vertices computed at runtime. The space complexity of a single NSD layer becomes $d^2+f_1f_2+2cd^2$. 

Notice that GCN, the classical graph neural network described in Remark~\ref{remGCN}, has total number of parameters equal to $c^2$ which is in some cases even bigger than the space complexity of NSD.
\end{con}

One of the main limitations of NSD is computational complexity, by which we mean the time complexity of an inference of a single layer.

\begin{rem}\label{remComputationalComplexity}
For $d\times d$ general matrices serving as restriction maps we have computational complexity $O(n(c^2 + d^3) + m(cd^2 + d^3))$, see~\cite{bodnar2022neural}. Assuming $c\gg d$ this complexity becomes $O(nc^2 + mcd^2)$. It is claimed in~\cite{bodnar2022neural} that in practice it is better to use $1 \leq d \leq 5$, but experiments were conducted only for small graphs; the optimal $d$ for large graphs has to be found. For a GCN with $c$ many channels, the complexity equals $O(nc^2 + mc)$, so we see that NSD has $d^2$ times higher computational complexity than GCN. This could be a serious limiting factor for practical applications of NSD for really large graphs. Some of the modern NSD-based architectures surveyed in subsection~\ref{subsecFurtherArchitectures} were designed to overcome this limitation.
\end{rem}

\section{Review of applications}\label{secReview}

This section reviews the existing applications of sheaf theory in computer science and deep learning. To add order to the wide variety of applications, we introduce an organizing scale. The applications are roughly clustered into six categories starting with the most ``theoretical'' and ending with the most ``practical''. This scale roughly resembles the history of subject, with applications of sheaves within category theory and topology appearing at the beginning and applications of sheaves in the design of neural networks appearing towards the end --- they are moved to a separate section. 

The scale is as follows:
\begin{enumerate}
  \item Categorical level: sheaves as a categorical language to substitute geometry. Subsection~\ref{subsecReviewCatLevel}.
  \item Topological level: sheaves as a framework to study local-to-global properties and interactions. Subsection~\ref{subsecProblTopologyLevel}.
  \item Abelian level: sheaves valued in abelian categories reduce problems to homological algebra. Subsection~\ref{subsecReviewHomologyLevel}.
  \item Real level: sheaves valued in $\Ro$-vector spaces reduce problems to numerical linear algebra and combinatorial Hodge theory. Subsection~\ref{subsecReviewRealLevel}
  \item Manifold level: connection sheaves as a discrete model for vector bundles over smooth manifolds. Subsection~\ref{subsecReviewManifoldLevel}
  \item Data level: sheaves and traces of algorithms operating on sheaves as the source of learning signals, applications in deep learning, and design of neural network architectures. Section~\ref{secReviewShvsML}.
\end{enumerate}


\subsection{Categorical level}\label{subsecReviewCatLevel}

The attempts to use sheaves as a tool to describe behavior of complex systems, either natural or originating in computer science, and translate this behavior into the language of category theory seem to have been made from the very moment that sheaves were invented. The early applications of sheaf theory were based on its tight connection with mathematical logic. We do not give a review of all advances in this direction, and refer to classical sources: the chapter of Fourman and Scott in~\cite[Ch.``Sheaves and Logic'']{ApplOfSheaves} and the book of MacLane and Moerdijk~\cite{MacLaneMoerdijk}. The general utility of sheaves in logic can be roughly explained by the fact that the category $\Shvs(X;\Sets)$ of $\Sets$-valued sheaves on a given space $X$ forms a topos, (see Appendix~\ref{secMathSpaceRestoredFromShvs}), a categorical abstraction corresponding, in a certain sense, to a logical theory. Topos-theoretical foundations of mathematics are more flexible than the classical Zermelo--Frænkel set theory in the sense that they are capable to incorporate ideas of intuitionistic logic. Sheaf theory provides a semantics for intuitionistic logic~\cite[Ch.VI]{MacLaneMoerdijk} and~\cite[Ch.XIV-XV]{troelstra2014constructivism}. Since intuitionistic logic and constructive mathematics are closely related to type theories and programming language theory, sheaves play foundational role in the field.

In modal logic, the space $X$ parameterizes the collection of possible worlds. Sheaves on $X$ (valued in $\Sets$ or other conventional categories) are used as models of the predicate modal logic. Roughly speaking, the topos $\Shvs(X;\Sets)$ is treated as a logic parameterized by possible worlds. In more down-to-earth applications, the topos $\Shvs(X;\Sets)$ is treated as a logic parameterized by points or regions of some understandable topological space $X$. These ideas motivated a number of papers in computer science where sheaves were used to model interacting systems, concurrent processes, or generic ``objects'' which may be spread in physical space, time, or more abstract spaces of possibilities or observations. This category of applications was proposed by Goguen~\cite{GoguenEarly} and developed in subsequent works~\cite{SheafSemConcObjects, ObjAsProcesses, Airchinnigh1997TheGO, HUT-TCS-A23, SheafSemAgents, Malcolm2006, Sendroiu, Malcolm2009} to list just a few. All these papers may be thought of as a part of even more general program on categorical foundations of computer science. See also~\cite{SpivakTemporal} and~\cite[Ch.7]{Spivak} for the modern exposition and treatment of these ideas.

Sheaf semantics of programming languages proved useful in the construction of normalization by evaluation (NbE) algorithms for certain important classes of lambda-calculi~\cite{NormByEval1995,NormByEvalCoproducts2001}. Given a $\lambda$-expression in a specified language, the algorithm seeks to reduce the expression to its normal form, which is needed e.g. to compare two equivalent expression. Instead of doing the reduction syntactically (i.e. performing a long sequence of term rewritings), the normalization by evaluation technique translates the string into a semantical representation, from which the normal form is extracted more easily. Semantics of presheaves over small categories proved useful for a simply typed $\lambda$-calculus~\cite{NormByEval1995}, while sheaves over Grothendieck topologies solve the decision problem for a simply typed $\lambda$-calculus with binary coproducts~\cite{NormByEvalCoproducts2001}. 

\subsection{Topological level}\label{subsecReviewTopLevel}

Most modern papers on applied sheaf theory emphasize that sheaves are about interaction between local and global properties of geometric objects. A sheaf -- understood not as a diagram on a poset, but rather as a diagram on the lattice of open subsets of a topological space -- satisfies the axioms of gluing and locality, see Definition~\ref{definSheaf}. These properties make sheaves a suitable mathematical abstraction to describe situations when compatible local entities can be uniquely assembled into a global entity. Such situations often occur in practice. Goguen~\cite[p.165]{SheafSemConcObjects} even formulated a ``sheaf hypothesis''\footnote{This hypothesis is similar in its epistemic nature to Church thesis in computer science, or the assumption that data are sampled i.i.d. from an unknown distribution which constitutes the basis of statistical learning theory.}: all objects are understood through their observable computational behavior; the observations satisfy the gluing axiom, hence can be axiomatized with sheaves. 

An old but quite impressive example of using local-to-global principle in the design of algorithms belongs to Srinivas~\cite{Srinivas1993}. He conceptualized the classical Knuth--Morris--Pratt string matching algorithm in the language of sheaves over Grothendieck topologies, and generalized this algorithm from strings to general data structures satisfying certain category-theoretical assumptions. From the practical perspective, each generalized pattern-matching problem reduces to the problem of subgraph matching (SGM). These ideas were further developed in~\cite{SrinivasSpecWare} and fused into a number of subsequent papers on application of category-theoretical methods in high-level software engineering.

The ideas of algorithm design based on sheaves similar in spirit to those proposed by Srinivas~\cite{Srinivas1993} are reborn in the modern approaches to the pattern matching and structure isomorphism tests in the works of Conghaile~\cite{Conghaile2022CohomologyIC} and Abramsky~\cite{Abramsky2022presh}. These methods take origin in the theory of finite models in logic and its interactions with category theory~\cite{AmbramskyStrPower}. The pattern matching problems $\ca{P}$, such as SGM or constraint satisfaction (CSP), are known to be NP-hard. In practical applications, they are replaced with algorithmically feasible approximations $\ca{P}_k$; the general construction being motivated by Ehrenfeucht--Fraïssé games, also called spoiler-and-duplicator games. The well-known example is the famous Weisfeiler--Leman (WL) graph isomorphism test\footnote{There is no mathematical result that graph isomorphism problem is NP-hard, however it is believed that this problem do not belong to PTime.}, or its generalizations, the $k$-WL tests for $k>1$. We refer the reader to the general modern machine learning-related review~\cite{WLinMLsoFar} and~\cite{WLandLogic} for the result connecting $k$-WL to the first-order logic.

We can gain an informal understanding of the above approximations their relationship to sheaf theory through the following example. Let $\Hom(H,G)$ denote the set of embeddings of a graph $H$ into a graph $G$, while $\Hom_k(H,G)$ consists of collections of embeddings of ``subgraphs of $H$ of size $\leq k$'' in $G$ which are compatible on the intersections. To solve the problem of subgraph matching means to find a coherent collection in $\Hom_k(H,G)$; it can be naturally formulated as finding a global section of a certain presheaf $\ca{H}_k(H,G)$ associated with the problem, see~\cite[Obs.1]{Conghaile2022CohomologyIC}. While finding the global section is still computationally hard, the construction and size of $\ca{H}_k(H,G)$ is polynomial in $|H|\cdot|G|$. This opens a way to improvements of $k$-approximate pattern-matching algorithms, which are still polynomial in time, but are stronger than their standard versions.

\subsection{Abelian level}\label{subsecReviewHomologyLevel}

The modern era of machine learning and statistics is characterized by using arrays of real numbers or the real vector spaces $\Ro\Vect$ as the basic abstractions--as opposed to core computer science where booleans and $\Sets$ play a primary role. Unlike $\Sets$-valued sheaves used in logic, sheaves valued in abelian categories, such as the category $\Ro\Vect$, form an abelian category themselves and support the notion of higher-order cohomology. Sheaves valued in abelian categories can be analyzed by tools from homological and linear algebra. Most methods and intuitions currently used in deep learning applications of sheaves, such as the notions of cellular sheaves and incidence numbers, originate from algebraic topology.

\subsubsection{Applied topology and signal processing}

While applications of algebraic topology may be found throughout the physical sciences at least as far back to the 1930s, the explicit usage of topological concepts within the context of data processing saw a growth in popularity in the early 21st century. Following the popularity of algorithmic techniques like persistent homology~\cite{edelsbrunner2002topological}, the field of applied topology and its data-scientific sub-field topological data analysis (TDA) have emerged as active areas of research within the applied mathematics community.

Early data-driven applications of algebraic topology were motivated by signal processing problems, especially those for which a non-trivial interaction exists between a topological domain and the signals measured on that domain. Problems like signal localization~\cite{de2007coverage}, image reconstruction~\cite{singh2007topological}, shape estimation~\cite{carlsson2004persistence}, or clustering~\cite{carlsson2009topology} benefitted from the global and compositional perspective offered by early TDA methods. While sheaf theory has been used to study and extend many of the theoretical guarantees enjoyed by TDA techniques~\cite{Curry,curry2015topological,kashiwara2018persistent}, the explicit application of sheaf-theoretic methods to data-driven problems only began to appear within the applied topology literature in the 2010s.

This applied sheaf theory program, emerging largely from the work of Curry, Ghrist, and Robinson, further generalized applied topological methods in network and algebraic signal processing through the language of sheaf cohomology. This application of sheaf theory flows primarily from the observation that the assignment of data to points within a space can be generalized from the assignment of single numbers to an assignment of more complex mathematical objects like vectors, groups or lattices to each point. Sheaf cohomology then provides machinery for organizing and observing obstructions to consistent assignment of this data, regardless of its structural complexity as explained in Subsection~\ref{subsecCohomologyMainPart}. 

In~\cite{curry2012euler}, a sheaf-theoretic presentation of Euler calculus is given, along with a number of potential application domains, including target enumeration, sampling, and signal convolution, among others. Such signal processing perspectives were further developed in~\cite{robinson2013nyquist,robinson2013understanding,robinson2015sheaf}, and~\cite{robinson2014topological}, resulting in an encompassing theoretical reinterpretation of signal processing within the language of sheaf theory and algebraic topology. This sheaf-theoretic modeling approach to signal processing has continued to receive modest interest through time, with applications in stratification learning~\cite{brown2021sheaf} and uncertainty quantification~\cite{joslyn2020sheaf} emerging more recently in the literature.

Concurrent to early work in sheaf-theoretic signal processing, Ghrist and Hiraoka~\cite{ghrist2011applications} proposed a reinterpretation of network coding problems through the language of cellular sheaf theory, targeting applications in maximum flow calculations, network extendibility, and network robustness analyses. This network coding application of sheaf theory was later extended in~\cite{ghrist2013topological} and~\cite{krishnan2014flow}, generalizing the max-flow min-cut theorem on networks to sheaves of semimodules over a network and using directed homology to measure realized flow and cut values. This generalization provides a theory for working with multicommodity or multi-source and multi-target flows on network, in addition to expanding the flow data type from real-valued flows to more complex mathematical objects like group- or lattice-valued flows. This network coding application of sheaf theory has persisted into more recent literature on network flows and routing, most notably in the study of delay-tolerant networking~\cite{short2021towards,short2022sheaf}.

Despite early work emerging from the applied topology community showing potential for the application of sheaf theory to data, early work in the area was primarily theoretical, employing sheaves as a model for phrasing desired compositional measurements within an application domain. This early lack of data-driven work was due primarily to the fact that sheaves require the definition of a topological base space upon which data may be associated and measured, and such base spaces are generally difficult to define and measure in many existing signal processing domains. It took the development of \emph{cellular sheaves}, largely accomplished within Curry's Ph.D. thesis~\cite{Curry}, for these sheaf-theoretic concepts to find their way into more realistic data-driven domains. Curry showed that one may define a type of constructible sheaf on a cell complex by endowing that complex with the Alexandrov topology, see details in Appendices~\ref{secMathStructures} and~\ref{secMathSheaves} for a precise definition. Since many signal processing problems are encoded discretely and often in a relational form, the cell complex provides a convenient topological base space for defining sheaves derived from real data thanks to its discretized, combinatorial nature and its ability to encapsulate graph-structured data.

\subsubsection{Utility of cohomology}

While the notion of cohomology comes for free once one defines a sheaf valued in the category of vector spaces, there is a dearth of papers depicting actual utility of cohomology in positive degrees for general sheaves. There are many papers applying (co)homology of constant sheaves and persistence sheaves. In the case of a constant sheaf, the cohomology coincides with the usual (singular) cohomology of a space, see Example~\ref{exConstantSheaf}. Betti numbers, the dimensions of cohomology of such sheaves, measure meaningful topological characteristics of a space: the number of connected components, the number of holes, 2-dimensional cavities, etc. This motivates further constructions of persistence sheaves and Euler calculus, and is explained in any basic book on TDA. The review of all such applications would require writing a separate paper; instead we refer to~\cite{DONUT} for a searchable database of works in applied topology. None of the papers listed in the previous subsection, proposes an actual application of Betti numbers of non-constant (non-persistent) sheaves.


The main feature of sheaf theory is that individualized vector spaces are associated with various parts of a structure (whether the structure is understood as a purely mathematical entity or an abstraction for a physical mechanism). The parts of various mechanisms and engineering constructions may be modelled by 0-, 1- and 2-dimensional rigid cells of a polyhedral complex embedded in $\Ro^3$; each part has its own vector space representing the degrees of freedom allowed for this part; the way how parts are connected with each other provides linear maps between them. These vector spaces and linear relations naturally patch together into a cellular sheaf or cosheaf. The (co)homology of this sheaf then resemble degrees of freedom of the whole structure. 

Historically, the idea of using sheaves for analyzing degrees of freedom appeared in toric geometry~\cite{HowIsAGraph, GKMZara}, where the ``parallel redrawing problem'' (or deformation problem) was posed for a linear embedding of a graph in $\Ro^n$. The space of all parallel redrawings of a graph, --- those in which edges stay parallel to themselves, --- is isomorphic to the space of global sections of the sheaf which is now called the GKM\footnote{Named after Goresky--Kottwitz--MacPherson theorem known in toric geometry.} sheaf of a graph~\cite{JabeaBaird}. 

Cooperband et al.~\cite{CooperbandThesis,HomGraphStat,reciprocalFigures} have shown a great potential of using sheaves and related homological theory in the structural engineering, as an effective language to speak about degrees of freedom. Force cosheaves, linkage sheaves, and position sheaves introduced in~\cite{HomGraphStat} provide an elegant description for many familiar phenomena in graphic statics. The recent paper~\cite{origami2025} applies sheaves to flexions of origami constructions. In~\cite{physicsinformed}, the idea of sheaf-theoretical engineering was applied to study dynamics of molecules viewed as graphs with rigid edges: anisotropic sheaf on a graph is introduced, so that Normal Mode Analysis --- a technique in biophysics --- becomes fully interpreted in terms of the spectrum of the corresponding sheaf Laplacian.

Whenever the underlying geometrical structure with a sheaf carries the action of a monoid $G$, the cohomology becomes $G$-representation. This allows to squeeze more features from the data since $G$-representations are more informative than vector spaces~\footnote{If $G=\Zo_{\geq 0}$ or $\Ro_{\geq 0}$ we get persistent cohomology instead of the ordinary version.}. Utilizing representations of finite groups of symmetries in graphic static was recently proposed in~\cite{equivariantSheaf}. 

An interesting category of applications of higher-order sheaf cohomology exists, which is similar in spirit to the obstruction theory known in algebraic topology (see e.g.~\cite{Adhikari2016}). Here, we make use of the first four nontrivial terms in the long exact sequence~\ref{eqLongExactDerivedCohomology}, i.e.
\[
0\to H^0(S;F)\to H^0(S;D) \to H^0(S;D/F) \to H^1(S;D)
\]
for a subsheaf $F$ of $D$. This sequence lets us interpret the 1-st cohomology group $H^1(S;D)$ as an obstruction to extend a global section of the quotient sheaf $D/F$ to a global section of $D$. Such extension problems appear in seemingly unrelated areas. In the field of topological signal processing, Robinson~\cite[Thm.4.4]{robinson2014topological} used the vanishing of $H^1(S;D)$ as a theoretical guarantee that a signal can be restored from a collection of samples. In the field of applications of finite models to CSP, Conghaile~\cite{Conghaile2022CohomologyIC} applied the vanishing condition of $H^1(S;D)$ to strengthen the classical $k$-WL and $k$-CSP tests. This idea is based on the subsequent work of Abramsky et al.~\cite{ContextParadox}, who used the $1$-st \v{C}ech cohomology group of a sheaf to characterize and measure the contextuality in quantum systems, databases, and CSP-problems. 

GKM-sheaves on non-cellular posets have a number of applications in toric topology. Ayzenberg et al. in~\cite{AyzMasSol,AyzBuchCluster,AyzSor} used computations of cohomology of such sheaves in combination with Morse theory to prove non-existence of asymptotical diagonalization algorithms for matrices of specific sparsity classes.

\subsection{Real level}\label{subsecReviewRealLevel}

Global sections are defined for sheaves valued in all reasonable categories, while the computation of cohomology is specific to abelian categories. Sheaf cohomology is defined for any ground field. The field $\Ro$ of real numbers is specific since it supports the features needed to pose and solve optimization problems, model physical processes, analyze spectral characteristics of objects, and so on. This basic functionality of real numbers and euclidean vector spaces led to the notion of sheaf Laplacians and the application of combinatorial Hodge theory within a sheaf-theoretic context.

Hansen and Ghrist~\cite{hansen2019toward, hansen2020laplacians} developed a spectral theory for cellular sheaves valued in $\Ro\Vect$ and showing that the sheaf Laplacian, a generalization of the Hodge Laplacian for sheaf-structured data, has a number of desirable properties which make it useful as an object for e.g. computing or approximating global sections of cellular sheaves. The sheaf Laplacian also serves as a drop-in replacement for the graph Laplacian in settings where graph-structured data is decorated with more complex assignments of objects to points or inherits functional relationships between points which may be encoded by restriction maps. Towards this end, the sheaf Laplacian has been used to define a distributed optimization technique, as the kernel of this operator provides a globally consistent assignment of sections to a space from a collection of local differencing operations~\cite{hansen2019distributed}. Inversely, given smooth signals assumed to be organized into a cellular sheaf structure, Hansen also showed that one may infer a sheaf Laplacian from these signals through linear optimization over the convex cone of sheaf Laplacians on a graph, mimicking more traditional graphical lasso techniques~\cite{hansen2019learning}. This intersection of TDA and sheaf theory continues to develop, with more recent works investigating persistent sheaf Laplacians~\cite{wei2021persistent}, sheaf-theoretic shape discrimination~\cite{arya2024sheaf}, and community detection~\cite{wolf2023topological}.

Interestingly, the versions of sheaf Laplacians, in particular higher order sheaf Laplacians, exist for non-abelian categories, --- in the situations when higher order cohomology itself is not defined. Ghrist and Riess~\cite{TarskiLapl} introduced the notion of Tarski Laplacian for a cellular sheaf valued in the category of lattices: in this case the convergence of the ``diffusion'' process to the ``harmonic'' cochain is guaranteed by the Tarski's fixed point theorem. The recent work~\cite{CatDiff} introduces the notions of Lawvere Laplacian and categorical diffusion by generalizing from the category of lattices to an arbitrary category enriched over quantales, where a generalization of Tarski's theorem holds. Such categorical diffusion allows one to treat some classical algorithms, such as Dijkstra's algorithm, as a form of sheaf diffusion, --- despite the fact that historically such algorithms are considered irrelevant to sheaf theory due to their non-linear nature. The mixture of sheaf-theoretical (maybe even homological!) methods with tropical mathematics, as it originally appeared in the study of algorithms, seems a promising area of future research.

\subsection{Manifold level}\label{subsecReviewManifoldLevel}

While a theoretical foundation for applied sheaf theory was developing within the applied algebraic topology and TDA literature, another more application-oriented thread of applied sheaf theory began to concurrently emerge from the graph theory literature. This line of work generally exploits generalizations of the graph Laplacian, derived from its original differential-geometric form, to process and manipulate higher-dimensional signals over a graphical base space. Here, a graph is treated as discrete substitute for a smooth manifold; the stalks $D(s)\cong \Ro^d$ are assumed to have the same dimension, and are treated as the ``fibers'' of a vector bundle, e.g. the tangent bundle. The structure maps $D(s_1)\to D(s_2)$ are treated as connection maps between fibers; they model the notion of a geometrical connection on a vector bundle, hence should be invertible; or even lie in a fixed gauge group, e.g. the orthogonal group $O(n)$. In practice, this abstraction provides the ability to integrate more complex signal data types into traditional graph signal processing approaches while permitting the expression of functional constraints within the base space. See Appendix~\ref{secMathConnection} for mathematical details.


Singer and Wu~\cite{singer2012vector} employed the graph connection Laplacian in the construction of a more general vector diffusion map technique which accounts for parallel transport of tangent vectors in the distance calculation between points. Given a finite set of points $X=\{x_1,\ldots,x_m\}$ in $\Ro^d$ sampled from an unknown submanifold $M\subset \Ro^d$ of dimension $n$, one may want to mine some geometrical properties of $M$. The procedure proposed in~\cite{singer2012vector} runs as follows. First, a proximity graph $G$ on $\{1,\ldots,m\}$ is built: it models the abstract (i.e. non-embedded) structure of the manifold. The tangent spaces $T_{x_i}M$ are modelled by the linear span of $n$ principal components of the subset $\{x_j\}$, where $j$ lies in the $G$-neighborhood of $i$. As abstract vector spaces all $T_{x_iM}$ have dimension $d$ (we forget their embedding in the ambient $\Ro^d$). The orthogonal matrices $O_{ij}\colon T_{x_iM}\to T_{x_jM}$ can be defined by a natural optimization procedure and constitute a discretized version of parallel transport maps that store information about the intrinsic geometry of $M$. The collection $\{D(i),O_{ij}\}$ may be viewed as a cellular $O(n)-$connection sheaf $D$ on a graph $G$. This abstraction faithfully models the latent manifold $M$.

In particular, the vanilla sheaf Laplacian of an $O(n)$-connection sheaf on a graph may be seen an discretized version of the Laplace--Beltrami operator on a Riemannian manifold. Peach et al.~\cite{peachimplicit} used this idea to generalize Gaussian processes for learning vector-valued signals over latent manifolds. The authors consider pairs $\{(x_i, v_i) \mid i = 1, \dots, n\}$, of training data where $x_i$ are samples from a submanifold $M\subset \Ro^d$ and $v_i$ are sampled from its tangent bundle $TM$, where both $M$ and $TM$ are unknown. One then seeks to infer a Gaussian process which can provide an implicit description of the vector field over $M$ that agrees with the training set and provides a continuous interpolation at out-of-sample test points with controllable smoothness properties. To do this, the authors develop a discrete version of vector-valued Matérn Gaussian process, known for Laplace--Beltrami operators, to a discretized setting of Singer and Wu described above. Vector-valued Gaussian processes allow for the interpolation of vector fields which are consistent with respect to some underlying manifold structure. This has applications in a number of computer graphics and data reconstruction domains like, as is shown empirically by Peach et al., spatiotemporal brain activity reconstruction from EEG data.

The general modeling technique proposed by Singer and Wu has since been extended to improvements on iterative least squares in multi-object matching~\cite{shi2020robust}, matrix-weighted consensus~\cite{trinh2018matrix}, generalizations of random walks and centrality to connection graphs~\cite{kempton2015high, cloninger2024random}, among others~\cite{tuna2016synchronization, huroyan2020solving}. 

It should be mentioned that the original paper~\cite{singer2012vector} does not mention sheaves; however we formulated their results in the general terminology chosen for this review. A number of alternative terminologies can be found in the literature. A similar structure is called an $O(n)$-bundle on a graph by Gao et al.~\cite{gao2021geometry}. On a more theoretical level, similar structures on generic finite posets appear in the work of Barmak and Minian~\cite{Barmak2012GcoloringsOP} where they are called $O(n)$-colorings of posets.

\section{Sheaves and deep learning}\label{secReviewShvsML}

This section describes how sheaf theory has been integrated within deep learning theory and applied to develop new deep learning algorithms.

\subsection{From graphs to sheaves}

Interest in the graph deep learning techniques surged in the mid-2010s following developments in deep learning and the proliferation of graph-structured data from social media and computational biology, resulting in the area of graph deep learning or, more generally, geometric deep learning~\cite{bronstein2021geometric}. The foundational architecture for modern graph deep learning is the graph neural network (GNN), whose layers are determined by a mechanism of message passing on a graph, see Remark~\ref{remMessagePassing}. There exists a variety of GNN architectures implementing the above message passing mechanism, the general review can be found elsewhere~\cite{velivckovic2022message}. We concentrate on the part of this story adjacent to sheaves. 

One of the earliest examples of message passing architectures is the graph convolutional network (GCN)~\cite{kipf2017semisupervised}. It implements the message passing operation as an approximation of a localized spectral filter. In the terminology of this review, GCN is a network based on the sheaf diffusion for the classical graph Laplacian, i.e. the vanilla Laplacian of the constant 1-dimensional sheaf on a graph, see Remark~\ref{remGCN}. 


In~\cite{hansensheaf} Hansen and Gebhart proposed the idea of Sheaf Neural Network architecture as a generalization of GCN. The architecture implements linear message passing, based on sheaf diffusion on a graph, but this time a sheaf may be arbitrary, not just constant. Arbitrary sheaves allow more freedom than the constant sheaf: nontrivial semantical information on relations between data assigned to adjacent nodes may be encoded in the structure maps of a sheaf. By the manual choice of a problem-related sheaf, the authors showed, through a simple example of a node classification task, the improved performance of sheaf neural networks over GCNs, which fail to perform above random chance in a heterophilous network. This performance discrepancy between Sheaf neural networks and GCN is due to the fact that for properly chosen restriction maps, i.e. when working over non-constant sheaves, the kernel of the laplaican is not the assignment of a constant cochain as in the case with GCNs~\cite{gebhart2022graph}. The idea of learning a sheaf from the data was proposed in~\cite{hansensheaf} but not implemented.


Bodnar et al.~\cite{bodnar2022neural} substantially extended the idea of using sheaves in the design of neural architectures. It was known that GCN architecture (as well as more general GNN architectures proposed until 2022) suffers from oversmoothing and exhibits poor performance on heterophilic graphs. 
Following~\cite{hansensheaf}, the authors defined Sheaf Convolutional Network (SCN) which uses sheaf diffusion in a predefined sheaf, several signal channels, and auxiliary weight matrices (see Construction~\ref{conNeuralSheafDiffusion}) and overcomes the stated problem appearing in existing GNNs\footnote{The paper claims that all GNN architectures suffer from oversmoothing and poor performance on heterophilic data. However, since sheaf diffusion on a graph is a particular case of message passing, SCN becomes a subclass of GNNs, in their modern understanding. This makes the claim ``sheaf networks overcome the problems of GNN'' a logical contradiction, so we avoid stating it in this particular form.}. 

The paper~\cite{bodnar2022neural} justified that SCNs (and sheaf Laplacian) are more expressive than GCNs (and graph Laplacian) on the node classification task, by analyzing their linear separation power. For any connected graph $G$ with $C\geq 3$ classes of nodes GCN cannot linearly separate the classes of $G$ for any initial node features. At the same time SCN with $d \geq C$ and diagonal invertible restriction maps can linearly separate the classes of $G$ for almost any initial node features. They considered various constrains on structure maps in a sheaf, resulting in different types of sheaves: symmetric invertible, non-symmetric invertible, diagonal invertible, and orthogonal. These types of sheaves show different separation capabilities, that is, the node classification task on a graph can be solved by performing diffusion with the specific type of sheaves.

The novelty of the work~\cite{bodnar2022neural} was in the implementation of sheaf learning. The paper proposed the model called Neural Sheaf Diffusion (NSD). The most important difference from SCN is that restriction maps of sheaves depend on the layer, and are considered learnable parameters. The training of a neural sheaf diffusion model can be described as learning a finite set of sheaves from data. Experiments show that Neural Sheaf Diffusion overcomes many limitations of classical graph diffusion equations (and corresponding GNN models) and provides competitive results on node classification tasks, specifically in heterophilic settings.

The combination of cellular sheaf theory with ideas of message passing appeared in Bodnar's thesis~\cite{bodnar2023topological}. Bodnar takes a perspective on topological deep learning based on the two postulates. \textbf{Locality Postulate.} The data is attached to neighbourhoods of a topological space. \textbf{Structure Postulate.} The data has a relational structure, which is induced by the overlaps between the neighbourhoods of the topological space. The author proposes Message Passing Simplicial Networks (MPSNs) and CW Networks operating on simplicial and cell complexes respectively. 
Topological pooling algorithm Deep Graph Mapper is introduced in~\cite{bodnar2023topological}. The algorithm works as some sort of decrease the resolution on a graph in a way that higher-resolution version of a geometrical structure gets compressed into a lower resolution structure with a non-trivial sheaf on top of it. This procedure mimics the existing pooling approaches in classical convolutional neural networks, but may be applied to arbitrary graphs instead of grids. 


\subsection{Development of sheaf-based architectures}\label{subsecFurtherArchitectures}

A number of new sheaf-inspired architectures was proposed after the work~\cite{bodnar2022neural}: they either deal with computational complexity described in Remark~\ref{remComputationalComplexity}, or generalize to non-graphical data structures, or propose some architectural innovations to enhance expressivity of a network.

%

Barbero et al.~\cite{barbero2022sheaf2} noticed that there were two main approaches to construct a sheaf over a graph described in subsequent works: either by manually constructing a sheaf based on domain knowledge or learn the sheaf end-to-end using gradient-based methods~\cite{bodnar2022neural}. However domain knowledge is often insufficient, and learning a sheaf could be computationally costly. Thus~\cite{barbero2022sheaf2} proposed to construct a sheaf from data in a deterministic way by adapting the method of Singer and Wu~\cite{singer2012vector} described in subsection~\ref{subsecReviewManifoldLevel}. The experiments show that the use of the resulting sheaf before the model-training phase improves the results on node classification tasks. Such an approach seems most well-suited for graphs with a small number of nodes.


For two similar but different local neighbourhoods GNN may generate identical node embeddings, thus GNN could fail to distinguish different graphs. Node positional encodings (PE) can prevent such undesirable GNN behavior by informing every node of its global position in the graph, thus breaking the similarity between nodes with similar neighbourhoods. Positional encodings are also required by Graph Transformers. He et al. in~\cite{he2023sheaf} introduced a novel way to construct PE based on eigenvectors of the sheaf Laplacian having the smallest eigenvalues. The authors consider several ways to obtain sheaf and the corresponding PE. The most simple option is to use just a graph Laplacian, generated by the constant sheaf with $d=1$. The other alternatives are to use the precomputed sheaf~\cite{barbero2022sheaf2} or the learned sheaf~\cite{bodnar2022neural}. The authors compared performance of GNNs without PE, with graph Laplacian based PE, and with precomputed sheaf based and learned sheaf based PE for node-level and graph-level tasks. GNNs equipped with precomputed or learned sheaf-based PE showed better results than GNNs without PE or with graph Laplacian based PE.


It was noticed by Barbero et al.~\cite{barbero2022sheaf1} that Graph Attention Networks (GATs)~\cite{velivckovic2018graph} suffer from the same two main problems as many other GNNs: oversmoothing and poor performance in heterophilic graphs. To keep the benefits of attention mechanism and fight these two problems at the same time, the authors introduced the Sheaf Attention Network (SheafAN), a generalization of the GAT utilizing sheaves on graphs. The key idea of SheafAN is to take a learned sheaf~\cite{bodnar2022neural} or a precomputed sheaf~\cite{barbero2022sheaf2} and, instead of using the vanilla sheaf Laplacian for sheaf diffusion, the blocks of the Laplacian are multiplied by attention weights. The authors also propose Res-SheafAN layer which is slightly different from SheafAN. The proposed approach shows good performance on heterophilic datasets, in particular it outperforms GAT. It also seems to be effective against oversmoothing, that is, the performance of SheafAN does not decrease as the number of layers increases.

In his thesis~\cite{barberoattention} Barbero studied the ability of sheaves and especially attention-based sheaves to solve the problems of oversmoothing and poor performance on heterophilic graphs. The thesis is built around the results of~\cite{barbero2022sheaf1} described above. The important novel part is an empirical evaluation of the influence of dimensionality (width) of stalks on the performance of the models on node classification task. The higher stalk width provides higher expressive power of the model, while also increasing the risk of overfitting. The experiments show that stalk width $d=4$ is often optimal. The used datasets are relatively small which amplify overfitting, so for larger datasets higher $d$ could be optimal. Another problem with a high stalk width $d$ is the additional memory costs, as the restriction maps grow quadratically with $d$.


The construction of local homology sheaf on a simplicial complex is well known in topology, in particular it was applied by McCrory~\cite{McCrory} to generalize Poincare duality to arbitrary simplicial complexes. The stalks and restriction maps of the local homology sheaf store important topological information about local neighborhoods of points in the geometrical structure. This notion appeared to be useful in applied topology~\cite{Robinson2018LocalHO}. The paper of Cesa and Arash~\cite{cesa2023algebraic} proposes using local homology sheaf for a clique complex of a graph to define sheaf diffusion and sheaf neural networks. While the idea of this architecture appears promising in capturing important topological information of the graph, its practical utility haven't shown to be useful so far.


In~\cite{bamberger2024bundle} Bamberger et al. introduced Bundle Neural Networks (BuNNs), which are the versions of GNNs operating via ``message diffusion'' instead of message passing. On a poset $S$ corresponding to a simple graph, a specific class of bundles ($O(d)$-connection sheaves) is defined, the so called flat vector bundles. In general, the definition of $O(d)$-connection sheaf $D$ requires to specify an orthogonal matrix $O_{v,e}=D(v<e)\in O(d)$ for each vertex-edge pair, see Appendix~\ref{secMathConnection}. The bundle is called flat if the matrices $O_{v,e}$ are all equal to a single matrix $O_v$ independently on the edge $e$. The monodromy of such bundle (see Remark~\ref{remMonodromy}), is trivial which makes the term consistent with the usual mathematical notation. Learning a flat bundle is easier than a general bundle simply because there are fewer parameters to learn. Still BuNN is reported to keep the advantages of general sheaf neural architectures.


Duta et al.~\cite{duta2024sheaf} consider sheaves on hypergraphs~\footnote{They are called cellular sheaves in the paper which is a terminological mistake. Since a hypergraph is not a cellular complex in any sense, the phrase ``cellular sheaf on a hypergraph'' makes as much sense as ``strong non-alcoholic beer''.} i.e. diagrams on 1-dimensional posets defined in Example~\ref{exBinRelToPoset}. Two types of Laplacians, linear and non-linear, are defined for such sheaves. The linear Laplacian is just the vanilla Laplacian as defined in Remark~\ref{remVanillaLaplacian} for the Roos complex, see Construction~\ref{conHypergraphDiffusion} and Remark~\ref{remRelationGrouping}. The linear Laplacian gives rise to sheaf diffusion in the standard manner described in subsection~\ref{subsecLaplaciansMainPart}. The non-linear Laplacian doesn't have a mathematical counterpart in the abelian sheaf theory, and resembles more general message passing mechanism on sheaves. Using these message passing mechanisms, the authors define Sheaf Hypergraph Neural Network (for linear Laplacian) and Sheaf Hypergraph Convolutional Network (for non-linear Laplacian). Their experiments confirm the advantages of using sheaf Laplacians for node classification task on hypergraphs.

The idea of using $O(n)$-connection sheaves on graphs, a natural discrete abstraction for a tangent bundle of a manifold, see subsection~\ref{subsecReviewManifoldLevel}), to define an analogue of convolution on arbitrary manifolds was proposed by Battiloro et al.~\cite{battiloro2024tangent}. Collecting these convolutions into layers combined with pointwise non-linearities gives rise to a continuous Tangent Bundle Neural Networks (TNN) architecture. Following a point cloud discretization procedure similar to that in Singer and Wu~\cite{singer2012vector}, the authors prove that the space-discretized architecture over the cellular sheaf converges to the underlying TNN as the number of sample points increases. Thus, this work reinforces the fact that the sheaf neural network discretizations applied in Barbero et al.~\cite{barbero2022sheaf2} and Bamberger et al.~\cite{bamberger2024bundle} are asymptotically appropriate. The authors also apply their discretized TNN architecture to a number of unique datasets generated by interesting manifold structures like torus denoising, wind field reconstruction and forecasting, and manifold classification. The discretized TNN architecture generally outperforms alternative manifold processing architectures.

Gillespie et al.~\cite{gillespie2024bayesian} also note that learning an informative restriction map structure for a task can risk overparametrization and overspecification to a particular learned sheaf structure during training, especially in scenarios when training data is limited. The authors address this problem by treating the sheaf Laplacian as a latent variable within the sheaf neural network architecture. This distributional perspective leads to the definition of a Bayesian Sheaf Neural Network which models the conditional probability of a classification problem into classes $y$ as $p_{\theta}(y \mid X)$ where $X$ are the input node features and $\theta$ are the parameters of the neural network. In this architecture, the input graph and features are passed into a variational sheaf learning process which approximates the posterior $p_{\theta}(\Fc \mid X, y)$. This results in a distribution over restriction maps that are assumed to be structured as either general linear, special orthogonal, or invertible diagonal transformations. This distribution is then sampled to create the restriction maps for each message passing layer of the underlying sheaf neural network architecture. This sampling procedure forces the parameters of the sheaf neural network to tune to classes of restriction maps instead of particular restriction map observations. This Bayesian approach appears to outperform deterministic learned restriction map structures on small datasets.

A plethora of other alterations to the sheaf neural network architecture have also been proposed within the literature. For example, Zaghen~\cite{zaghen2024nonlinear} inserts an edge-wise nonlinear function $\Phi\colon C^1(S,D)\to C^1(S;D)$ in the definition of the sheaf Laplacian to produce a nonlinear sheaf Laplacian operator $\Delta' = d^*\circ\Phi\circ d$, as in Hansen and Ghrist~\cite{hansen2021opinion}, and introduces the corresponding sheaf neural network architecture. 
Similarly, Caralt et al.~\cite{caralt2024joint} consider a number of restriction map update operations for sheaf neural networks inspired by opinion dynamics models~\cite{hansen2021opinion}. He et al.~\cite{he2023sheaf} integrate positional encoding within the sheaf neural network architecture to improve distinctiveness during the diffusion operations between neighbors in the network with equivalent restriction maps. Work~\cite{di2025learning} presents a new approach to learning sheaf Laplacians. The authors develop a method to jointly infer both the graph topology and the restriction maps of a sheaf, optimizing them to minimize the total variation of observed data. The main idea of the paper is that distance between signals on different nodes depends only on their cross-correlation and subspace dimensions, and does not depend on the specific structure of local dictionaries. Their approach is more computationally efficient than previous methods based on semidefinite programming, since it uses closed-form solutions for the basic optimization steps. The experiments show that their sheaf-based method achieves better data smoothness compared to conventional approaches.

\subsection{Applications of sheaf-type neural networks}

Recently, the field of application of sheaf theory in combination with deep learning methods in real-world applications has been actively developing. Various features of the subject area and gaps of existing solutions motivate the use of the corresponding advantages and strengths of sheaves.

Atri et al.~\cite{atri2023promoting} solves the problem of multi-document summarization, i.e. aggregation of information from several documents into a short summary. The FABRIC model is proposed, which combines simplicial complex layers to capture non-binary relations, BART as a text encoder and Sheaf Graph Attention to model complex semantic relations. The authors emphasize that for summarization, many documents are encoded by heterophilic graphs; in such conditions, vanilla GNNs have difficulties, but Sheaf Graph Attention is very effective. This ensures that local document relationships will be effectively represented in the global summary, addressing pressing linguistic consistency issues common in existing document summarization systems. FABRIC achieves the best results on the Multinews and CQASumm datasets and shows the usefulness of sheaf theory in handling complex inter-document relationships.

The Sheaf4Rec framework proposed by Purificato et al.~\cite{purificato2023sheaf4rec} applies Sheaf Neural Networks and their corresponding Laplacians to recommended systems. It provides a more efficient representation of complex user-item interactions using Cellular Sheaf than classical GNNs. Sheaf4Rec shows notable performance improvements on the Yahoo and MovieLens datasets.

In networks where nodes model people and edges correspond to relations, the community detection problem reduces to identifying cohesive groups of similar nodes. Wolf et al.~\cite{wolf2023topological} proposes a sheaf-theoretic approach to this problem. Empirical evaluations on the Karate Club benchmark show the robustness of the approach and high modularity scores.

Several papers explore the application of sheaf theory to Personalized Federated Learning (PFL), the training paradigm that delivers unique parameters to clients in a federation based on their local data distribution, enabling parameter sharing across clients. Nguyen et al.~\cite{nguyen2024sheaf} presents the Sheaf Hyper Networks (SHN) approach, which solves the over-smoothing and heterophily problems by integrating Sheaf structure in classical Graph Hyper Networks (GHN) architecture. SHN outperforms baseline methods accuracy by 2.7\%  in non-IID settings. Meanwhile, Liang et al \cite{liang2024fedsheafhn} propose for the Personalized subgraph Federated Learning task the approach FedSheafHN, which constructs a server-side collaboration graph where client embeddings are enriched using Neural Sheaf Diffusion integrated with an attention-based hypernetwork. FedSheafHN achieves high performance, demonstrating fast convergence and effective new clients generalization. Work \cite{issaid2025tackling} propose to use sheaves on graphs for federated multi-task learning (FMTL). FMTL aims to train separate but related models simultaneously across multiple clients (nodes), each potentially focusing on different but related tasks. The clients do not share raw data and can have different model sizes. Sheaves allow to address feature heterogeneity between clients, which is one of the main obstacles for previously used FMTL algorithms.

Huntsman et al.~\cite{huntsman2024prospects} applies sheaf theory in combination with large language models (LLM) to develop an approach for solving the problem of logical consistency assessment in hypertexts in the domains of jurisprudence documents and social networks. Note that it does not use Sheaf Graph Network, but uses the classical formalism of sheaf theory to ``glue'' local consistency judgments assessed by LLMs into a global consistency structure. This application demonstrates the power of sheaves in formalizing local-to-global reasoning problems.

\subsection{Representation Learning and Interpretability}

Sheaf theory lies at the boundary of geometry, algebra, category theory, and topology. Because it is situated on the border of computational tractability and abstract mathematics, sheaf theory has also been used as a lens through which to elicit better understanding for the performance and behavior of many modern deep learning pipelines.

\emph{Representation learning} is a common motif within machine learning which seeks to translate data in a raw format into some information dense latent representation that is better suited for solving down-stream tasks. Representation learning over graphs or network-like structures often requires the translation of elements of the graph, the nodes and edges, into vectorized representations which preserve information about the original graph topology while admitting a more computationally tractable form. This graph representation learning procedure is typically facilitated by the introduction of a self-supervised learning task wherein node representations are used to predict other nodes within their surrounding neighborhoods or along paths passing through that node. While such an approach is sufficient for simple, homogeneous graphs, this homophily-biased approach is ill-suited for heterogeneous graph structures like knowledge graphs wherein disparate nodes may be connected across typed relations whose type conceptually bridges this disparity~\cite{gebhart2023sheaf}.

\emph{Knowledge graph embedding} is a graph-based self-supervised learning approach which seeks to account for this discrepancy between topology and representation by learning node representations which are transformed across learned edge representations. These functional representations which connect adjacent nodes permit the graph embedding to be locally non-constant, allowing heterogeneity to be captured by proper transformations in the embedding space as one moves along the underlying graph's topology. Gebhart et al.~\cite{gebhart2023knowledge} showed that cellular sheaf theory provides a framework within which many pre-existing knowledge graph embedding approaches may be viewed as sheaf learning problems for a particular specification of an underlying sheaf. In other words, the authors show that many popular transductive knowledge graph embedding techniques may be reinterpreted as learning the restriction maps of a sheaf (the edge, or relation, representations) in conjunction with a particular cochain (the node, or entity, representations). In addition, Gebhart et al. show that many of the loss functions employed by knowledge graph embedding techniques may be interpreted as minimizing the inconsistency of the sheaf and cochain observation via implicit minimization using the sheaf Laplacian. Gebhart and Cobb~\cite{gebhart2023extending} later extended this approach to inductive learning, showing that optimal representations for new nodes included in the knowledge graph after the original self-supervised embedding procedure may be inferred by phrasing a harmonic extension problem using the sheaf Laplacian. This harmonic extension treats pre-existing nodes as the fixed, boundary values and the newly-introduced nodes as the interior for which a cochain representation must be inferred which is as close to a global section as possible with respect to the previously-learned relation representations connecting the boundary and the interior nodes. The authors show that this process can be solved exactly via (pseudo)inversion of the sheaf Laplacian, or approximated through iterative methods, again involving the sheaf Laplacian.

Kvinge et al.~\cite{kvinge2021sheaves} propose a sheaf-theoretical approach to understanding machine learning model fit. To do this, they define sheaves on the data by constructing a topology with open sets corresponding to subsets of related points. They build a data sheaf $\ca{D}$ consisting of all possible data value observations and a model presheaf $\ca{M}$ consisting of a specified family of functions, --- those derivable from the model, --- associated with each open set. The authors define a map $\Phi^{\ca{M}}_U: \ca{D}(U) \to \ca{M}(U)$ for each open set $U$ corresponding to a model for a data observation associated with $U$. Measuring the extent to which $\Phi^{\ca{M}}$ fails to satisfy the axioms of a presheaf morphism, in other words, measuring the inconsistency of the data-model association, allows the authors to point to specific subpopulations of a dataset on which a given model’s ability to fit the data changes, providing a local measure of model behavior and performance which can further inform global performance measures like accuracy.

\subsection{Beyond geometry and ML}\label{subsecBeyondML}

A number of applications of sheaf theory can be found in the literature, which do not fall directly into the categories listed above. The study of contextuality in quantum systems by sheaf-theoretical methods was initiated by Abramsky~\cite{AbramskyContextuality, ContextParadox}. The methods of quantitative analysis of contextuality based on sheaf theory fused into the study of lexical ambiguities in a natural language, in particular in anaphora resolution~\cite{SemanticUnification, SheafLanguageAmbig}.

Robinson~\cite{Robinson2016SheafAC} proposed a homological approach to the analysis of graphical causal models based on sheaf cohomology (on general posets), as defined in subsection~\ref{subsecMathCohomologySimplicial}). The approach is based on the observation that compositionality equations~\eqref{eqCompositionalityMainPart} remind chain rule of probability. The solutions to a graphical model can be treated as global sections of certain sheaf associated with the model. Belief propagation, a particular instance of message passing known in graphical models, is described in the terminology of pushforward sheaves; it happens that reliability of belief propagation can be measured in homological terms. Rosiak~\cite[Ch.9]{Rosiak} describes a related idea; a bayesian network gives rise to a pair (sheaf, cosheaf) on a simplex, and the assignment of joint probabilities to the nodes is treated as a simultaneous global section of both a sheaf and a cosheaf.

A research on using sheaves in networks' science initiated by Robinson, Ghrist, and Hansen had raised interest in engineering applications~\cite{SheafNetworkingNASA}. Macfarlane~\cite{Macfarlane04072014,Macfarlane17022017} introduced the category of ``combinatorial systems'' as an elegant mathematical model for scheduling and resource management in supply and production networks; the notion is potentially applicable to describe more general complex systems, as long as their formal ontology contains quantifiable inputs, outputs and production time. Sheaves over combinatorial systems describe collections of schedules; while higher order \v{C}ech cohomology modules measure discrepancy between local and global schedules.


The idea of sheaf theory as a language describing how local entities patch into global had found application in modern cognitive science. Philips~\cite{Phillips2018GoingBT} uses sheaves to describe generalization ability: the ability to patch knowledge from a finite set of observations. These ideas could be of interest to specialists in ML and general AI. In deep learning, a different, more statistical, toolbox is applied to measure this notion. Basically, generalization ability of a model is its ability to guess the outputs for inputs which were not used during the training. Such vague definition of generalization ability is independent on a theory describing knowledge representation inside a model\footnote{This reminds of IQ-tests: they neither measure intelligence, nor explain anything about intelligence theoretically. They measure the ability to solve IQ-tests; no more and no less.}. Sheaf theory provides more intrinsic and logical view on generalization ability, which may potentially find its place in theoretical ML and AI alongside with Vapnik--Chervonenkis theory, singular learning theory, and other elegant formalisms.

\section{Proposals and problems}\label{secProblems}

In this section we gather a number of proposals and problems, which in our opinion, may stimulate further work in the area of applied sheaf theory. The solutions may result in both the improvement of practical algorithms, and the further development of theory. It should be noted, however, that none of below problems are problems (or conjectures) in a strict mathematical sense, and their formulations may be imprecise. However, these problems can be placed on the same scale introduced in Section~\ref{secReview}: we believe that these problems may be of certain interest to specialists at each of the corresponding levels.

\subsection{Categorical level}\label{subsecProblCategoryLevel}

By geometrical structures we mean either of the structures mentioned in Section~\ref{secIntuitions}: undirected graphs, simplicial complexes, CW-complexes, posets or topologies.

\begin{probl}\label{problDescribeGeometry}
For a given finite geometrical structure $X$ provide a general mathematical basis for sheaf learning.
\end{probl}

Informally, sheaf learning, as understood in sheaf neural networks, is a procedure of picking a finite number of sheaves from $\Shvs(S;\Ro\Vect)$ in a way that sufficiently solves some task (whatever it is). A straightforward formalization of this procedure means that $\Shvs(S;\Ro\Vect)$ is treated as a smooth manifold, and gradient optimization is performed on a single component of this manifold (the one with prescribed dimensions of stalks) to optimize a given loss function. Sheaves learned this way stay unrelated to each other; the procedure simply provides a sequence of sheaf diffusion operators stacked together in a network.

In other words, the existing learning procedure is totally agnostic of the fact that the collection $\Shvs(S;\Ro\Vect)$ has a rich structure: it is an abelian category. To our knowledge, the whole area of ``learning over a category'', --- when we want not only to learn some objects, but also learn morphisms between them, --- does not yet exist. One could extend this to infusing the specific structure of the abelian category directly into the learning paradigm.

\subsection{Topology level}\label{subsecProblTopologyLevel}

The dimension of a (geometric realization of a) poset $S$ equals the largest length $n$ of a chain $s_0<s_1<\cdots<s_n$ in $S$. Henceforth, higher-dimensional topological notions come into play when the underlying structure $S$ is not just a binary relation (as in Examples~\ref{exGraphToPoset} and~\ref{exBinRelToPoset}). The problem here is that the mechanics of sheaf learning described in subsection~\ref{subsecSheafLearning}, while being clear and natural, is only applied in practice to 1-dimensional structures. It may seem at a first glance that high-dimensional structures such as simplicial or cellular complexes are used in topological deep learning~\cite{TDLbeyond}. However, unwinding the constructions, we see that each such structure $S$ is first disassembled into a number of 1-dimensional structures $S_1,\ldots,S_p$, and sheaf learning is performed independently on these structures, not on the original $S$.

The formal definition of a sheaf poses a conceptual challenge if one wants to learn a sheaf. Assume that the structure maps $D(s<t)$ of an $\Ro\Vect$-valued sheaf $D$ are represented by real matrices $D_{st}$ whose entries are to be learned by the gradient descent. By Definition~\ref{definSheafDiagramMainText}, a sheaf should satisfy compositionality~\eqref{eqCompositionalityMainPart} which is a system of nonlinear equations
\begin{equation}\label{eqCompositionalityEquations}
D_{s_1s_2}D_{s_0s_1}=D_{s_0s_2}\mbox{ for each }s_0<s_1<s_2 \mbox{ in }S
\end{equation}
on the matrix entries. How do we ensure that the learned matrices satisfy these equations? In other words, how do we know that a learned sheaf is indeed a sheaf? There are two standard answers, we consider both quite unsatisfactory:

\begin{enumerate}
  \item In practice, sheaf learning is applied to graphs or hypergraphs, which are 1-dimensional, so system~\eqref{eqCompositionalityEquations} is empty. Even if the data were originally represented by higher-order structure, the practical implementation forgets the higher order relations and treats data as a finite collection of binary relations on which sheaf diffusion or message passing is performed.
  \item System~\eqref{eqCompositionalityEquations} is ignored. In this case the learned collection of matrices is not a sheaf but rather a quiver representation. Again, this approach loses the information of commutativity relations which was present in the original diagram, and made all topology essential. The language of quiver representations contains neither global sections nor higher-order cohomology, so one should be careful in applying the terms which are not applicable.
\end{enumerate}

Related to this observation, we highlight the following proposal to take commutativity relations into account.

\begin{probl}\label{problRelations}
Approach the design of ML algorithms in the way that commutativity relations inherent to data are explicitly reflected in the architecture of a neural network.
\end{probl}

\begin{rem}
As a naive solution, the degree of non-commutativity 
\[
\sum_{s_0<s_1<s_2}\|D_{s_1s_2}D_{s_0s_1}-D_{s_0s_2}\|^2
\] 
can be used as an additional regularization term of the loss function when training a sheaf. The precise form of such regularization term may vary. For example, in the case of regular CW-complexes it could be of the form $\sum_{s_0<s_2}\|D_{s_1's_2}D_{s_0s_1'}-D_{s_1''s_2}D_{s_0s_1''}\|^2$, where $\dim s_2=\dim s_0+2$ and $s_1',s_1''$ are the two cells in between $s_0$ and $s_2$; this particular form doesn't use auxiliary ``skip connections'' $D(s_0<s_2)$.

From the mathematical perspective, the construction of such regularization terms is the subject of quiver representations with relations: the set of basic relations may be given in the problem's formal ontology together with the quiver itself. This leads to the consideration of 2-categories instead of 1-categories, so if things are taken seriously, a 2-categorical version of message passing should be developed. There may be even more general relations between relations, and so on, encoded in a structure called $n$-computad~\cite{Makkai} or a polygraph~\cite{BURRONI199343}, which suggests a nontrivial fusion of deep learning with higher categories.
\end{rem}

\subsection{Homology level}\label{subsecProblHomologyLevel}

While most of the theory described in Appendices~\ref{secMathCohomology} and~\ref{secMathLaplacians} is more or less folklore and well known to mathematicians, there is no good working implementation so far. Existing symbolic algebra packages such as Sage~\cite{Sage} and Macaulay2~\cite{Macaulay} do not support sheaf cohomology computations on arbitrary finite posets. Most existing homology-related packages are developed in the framework of topological data analysis, and support neither general geometrical structures, nor general sheaves, nor Laplacian-related numerical linear algebra.

\begin{probl}\label{problPackage}
Develop a computational system with the following functionality.
\begin{enumerate}
  \item Construction of diagrams over arbitrary finite posets (and maybe even finite categories) valued in common computationally feasible categories $\Vv$ such as $\FinSets$, $\ko\FinVect$, $\ko[t]\Mod$. 
  \item The ability to check commutativity of a diagram.
  \item The ability to find some/all global sections of a diagram in an optimal way.
  \item Over abelian categories $\Vv$, the ability to compute cohomology of sheaves and related notions of homological algebra including (minimal) injective resolutions, derived functors, etc.
  \item Over $\Ro\Vect$ the ability to form Laplace operators, compute their eigendecomposition, simulate heat diffusion, and construct sheaf diffusion operators.
\end{enumerate}
\end{probl}

Similar problems were stated in~\cite{Macfarlane04072014,Macfarlane17022017} motivated by the analysis and scheduling of production networks. The problem is motivated not only by practical considerations reviewed in Section~\ref{secReview} but also by the needs originating in theoretical mathematics\footnote{The first author faced the need to compute cohomology of GKM-sheaves in~\cite{AyzSor})}.

\begin{rem}
Although Problem~\ref{problPackage} seems to be pure engineering, we believe that its solution will stimulate further study in a broad area of hardware optimization for scientific computing. If the system is designed to work with arbitrary categories by design, then the required optimization patterns will not be specific to $\Ro\FinVect$ (as in applied linear algebra), or $\FinSets$ (as in SAT- and CSP-solvers), but will have more universal categorical nature.
\end{rem}

\subsection{Real level}\label{subsecProblRealLevel}

In modern ML, there exists a linear representation hypothesis~\cite{park2023LinRepr} which claims that even in nonlinear language models, the high-level semantical concepts are represented by directions or linear functionals in the representation space. The number of meaningful concepts may be bigger than the dimension of the representation space, so that a space $\Ro^l$ is capable of storing more features than the dimensionality $l$. This is confirmed on toy examples~\cite{ToySuperposition} and by the general success of vector embeddings~\cite{Mikolov2013EfficientEO} where the number of words in the vocabulary is usually much bigger than dimensionality of the embedding space. From a topological perspective, the phenomenon is also not surprising: a simplicial complex $\ca{K}$ on $m$ vertices can be embedded in the space $\Ro^m$ (Construction~\ref{conStandardGeomRealization}), but it can also be embedded in the euclidean space of dimension $2\dim\ca{K}+1$ or even less~\cite{Horvatic}. The embedding dimensionality depends on the maximal size of relations between nodes, not their total number.

In Section~\ref{secIntuitions} we discussed that there is nothing magical in the process of heat diffusion on a sheaf $D$ on a cell poset $\ca{X}$: it is just the optimization of a nonnegative quadratic function on the real space $C^0_{CW}(\ca{X};D)$ of cellular cochains (Remark~\ref{remNNdynamical}). What makes this procedure efficient is that $C^0_{CW}(\ca{X};D)$ has a rich semantic structure by design; it comes from its decomposition into the direct sum $\bigoplus_{v\colon V}D(v)$ (see Remark~\ref{remNNdynamical}). Looking at a vector from $C^0_{CW}(\ca{X};D)$, we know which part of it stores the state of each particular node. If the stalks $D(s)$, for $s\in S$, have large dimension, the state vector may be more informative, but the cost of the model becoming computationally inefficient. However, if we don't need precise higher cohomological information about a structure, and only need approximate features of heat diffusion, what if we squeeze the semantic information into a lower-dimensional space instead of keeping stalks orthogonal in the direct sum $\bigoplus_{v\colon V}D(v)$? This question suggests the following framework.

\begin{defin}\label{definRepresentationEvolution}
Let $\ca{X}$ be a cell poset. A working representation $W$ for $\ca{X}$ is a euclidean space $W_{tot}\cong\Ro^d$, and a choice, for each cell $\sigma\in\ca{X}$, of a euclidean subspace $W(\sigma)\subset W_{tot}$ together with the orthogonal projection $p_\sigma\colon W_{tot}\to W(\sigma)$. An evolution process $\Theta$ is a dynamical system on the space $W_{tot}$. We call the pair $(W,\Theta)$ a \emph{condensed sheaf} on $\ca{X}$.
\end{defin}

A similar definition may be given not only for cell posets, but for any meaningful geometrical structure. If a condensed sheaf is chosen well-aligned with the structure of $\ca{X}$, then performing a flow $\Theta$, then projecting to the semantical subspace $W(\sigma)$, we would get the evolution of the state of $\sigma$ which may be a valuable source of learning signal.

\begin{ex}
For a cellular sheaf $D$ on $\ca{X}$, define a working representation $W^{CW}_{tot}=\bigoplus_jC^j(\ca{X};D)$, with $W(\sigma)$ being the corresponding direct summand, and the evolution $\Theta^{CW}$ being the heat diffusion in all degrees at once. This working representation and evolution had proven useful in practice, as confirmed by the papers listed in Section~\ref{secReviewShvsML}.
\end{ex}

In general, there exist subspace arrangements $\{W(\sigma)\subset W_{tot}\mid \sigma\in\ca{X}\}$ which exhibit superposition (non-transversality) similar to the described in~\cite{ToySuperposition}. However, so far there is no working example demonstrating practical viability of condensed sheaves.

\begin{probl}\label{problCondensedSheaves}
Construct a condensed sheaf which has better expressivity compared to cellular sheaves, and the dimension $d=\dim W_{tot}$ of the working representation smaller than the size of $\ca{X}$. Devise methods to construct/learn such condensed sheaves where both operations $\Theta$ and $p_\sigma$ are computationally simple compared to sheaf-type neural networks.
\end{probl}

This problem stands in line with the whole general idea of machine learning as it is understood nowadays: ``the space $\Ro^d$ is capacious enough to store data, gradient descent is enough to process the data''. It may be a good idea to bring this philosophy to geometry.

\subsection{Manifold level}\label{subsecProblManifoldLevel}

The manifold hypothesis is well-known in modern machine learning~\cite{ManHypo} and statistics, although it implicitly originated in 1970's as a part of applied catastrophe theory (and already that time criticized, see~\cite[Sec.``Spurious quantification'']{ApplCatastrophIsShit}).

A data set is a finite subset $X$ of $\Ro^n$. As is common in statistical learning, the data points are assumed to be sampled i.i.d. from some unknown distribution on $\Ro^n$. The manifold hypothesis asserts there exists a submanifold $M$ of $\Ro^n$, with $m=\dim M<n$ such that the distribution is continuously supported on $M$. Such manifold $M$ is called a data manifold, or a latent manifold. The task then becomes to describe $M$, or at least study its mathematical properties, given a finite set $X$ of samples. This research area is called \emph{manifold learning}. It should be noticed however that in machine learning, unlike mathematics, there is no generally accepted definition of a manifold. Some papers treat $M$ as a solution to a system of equations; others consider $M$ to be an immersion of a single chart $\Ro^m$ into $\Ro^n$; in topological data analysis manifolds are replaced by triangulations and filtrations; in statistics probability distributions are the first-class citizens. Usually the definition of a manifold is not even given and in practice one is interested in some vague characteristics of $M$ as there is no practical need to define the space rigorously.

We believe that the problem of manifold learning should be taken more seriously from geometrical perspective. Recall that the classical definition of a (smooth closed) manifold $M$ involves an atlas of open charts $\{U_\alpha\subset\Ro^m\}$, a collection of open subspaces $V_{\alpha\beta}\subseteq U_{\alpha}$, and a collection of smooth invertible gluing maps $g_{\alpha\beta}\colon V_{\alpha\beta}\to V_{\beta\alpha}$ satisfying the cocycle equations $g_{\alpha\beta}=g_{\beta\alpha}^{-1}$ and $g_{\alpha\beta};g_{\beta\gamma};g_{\gamma\alpha}=1$ whenever they make sense. A smooth map (in particular, an embedding) $f\colon M\to \Ro^n$ is encoded by a collection of smooth maps $f_\alpha\colon U_\alpha\to \Ro^n$ which are compatible with respect to gluing maps $f_\alpha=g_{\alpha\beta};f_\beta$.

\begin{probl}\label{problManifoldHypothesis}
Given a finite set of points $X\subset \Ro^n$ and $m<n$, find a manifold $M$, its embedding $f\colon M\to\Ro^n$, and a finite subset $L\subset M$, such that $f(L)=X$ and the data $(M,f)$ have the lowest possible complexity.
\end{probl}

By complexity we mean the complexity of the total description in terms of an atlas and the gluing maps, this is similar to the classical regularization in machine learning. The main difficulty in learning a manifold, as stated in Problem~\ref{problManifoldHypothesis}, is to enforce the cocycle equations, at least approximately, within the output structure. This is similar in nature to Problem~\ref{problRelations}: in order to actually learn a sheaf on a high-dimensional structure one needs to attend to compositionality relations. 

The main tool of smooth geometry is the tangent bundle of a smooth manifold. We expect solution to Problem~\ref{problManifoldHypothesis} should be related to the method of Singer and Wu~\cite{singer2012vector} outlined in subsection~\ref{subsecReviewManifoldLevel}. It will likely make use of connection sheaves, since the latter are an expressive mathematical abstraction for vector bundles.

\begin{rem}\label{remNeuroSymbPlayground}
Problem~\ref{problManifoldHypothesis} does not seem to be solvable in higher dimensions. First of all, manifold learning is usually considered a method of dimensionality reduction, whereas the internal dimension $m$ in Problem~\ref{problManifoldHypothesis} is already given. Second, the curse of dimensionality applies here as well. In high dimensions, even large datasets are too sparse for the problem of manifold reconstruction to make practical sense. Nevertheless, the potential solution of Problem~\ref{problManifoldHypothesis}, in the way it is stated, seems valuable even in low dimensions. The problem natively suggests the fusion of neural and discrete methods. While the cocycle relations can be solved by running gradient descent, the topological structure of an atlas requires discrete optimization methods such as genetical algorithms, and their topological manifestations such as Morse surgery. For this reason we consider Problem~\ref{problManifoldHypothesis} a perfect playground to test neuro-symbolic approaches in machine learning.
\end{rem}
%

\subsection{Data level}\label{subsecProblDataLevel} 
In this subsection we pose two problems, both of which having relatively clear specification.

\textbf{Improvement of algorithms.} As discussed in the introduction and explained in Section~\ref{secReview}, sheaf theory provides an extremely powerful toolbox for the representation of geometrical data in the way that captures geometrical, topological, and human intuitions. Sheaves offer an efficient geometrical data representation for feature extraction. However, most of such applications improve high-level characteristics of neural networks (such as accuracy, precision, generalization ability, etc.) which depend on a large number of factors, many of which have little obvious relationship to the topological properties of the underlying space. It is difficult to quantify the gain of using sheaves in each particular case and, given the computational overhead of many sheaf-related methods, as described in subsection~\ref{subsecFurtherArchitectures}, it is not always clear that the abstraction is worth these costs. 

In the literature, we couldn't find any example, when a computationally heavy precise algorithm achieves higher speed by utilizing concepts of sheaf theory. By heavy precise algorithms we mean empirical algorithms solving NP-hard problems, such as boolean satisfiability problem (SAT), constraint satisfaction problem (CSP), subgraph matching problem (SGM), etc. In all these problems there is no accuracy/precision, and the only characteristic needed to be optimized is the time of performance on a given class of examples. There exist extremely efficient solvers that can empirically solve such problems despite their complexity class. None explicitly utilize sheaves in their internal mechanisms.

\begin{probl}
Find an NP-hard problem for which an application of concepts from sheaf theory beats state-of-the-art solutions in terms of total computation time.
\end{probl}

We believe that solution of this problem will not only be of practical value, but also stimulate new research on interconnections of sheaf and category theory with computer science and complexity theory.

\textbf{Large graphs.} The existing research focuses primarily on the application of sheaf-type neural networks to relatively small graphs. At the same time, many real-world datasets come in the form of large graphs: knowledge graphs, social networks, Wikipedia links' structures, and other large-scale information networks. Applicability of sheaf-type neural networks to problems posed on such graphs remains an open challenge. Understanding the effectiveness and limitations of sheaf-type neural networks in these contexts is crucial for advancing its practical utility.

\begin{probl}
Extend the architecture of sheaf-type neural networks to large graphs, i.e. graphs with more than million nodes.
\end{probl}

%
%
%
%
%
%
\clearpage

\appendix

\section{Geometrical structures} \label{secMathStructures}

In the appendix sections we give rigorous definitions of the data structures and sheaf-related theory mentioned in the main part of the text, and provide a self-contained exposition of the related theory. Here is the quick guide to the appendices.
\begin{enumerate}
  \item The remaining part of Appendix~\ref{secMathStructures}. We start with the review of geometrical structures: partially ordered sets (posets), simplicial complexes, CW-complexes, introduce the notion of cell poset and its variations, among which the most important is homology Morse cell poset. We recall the classical connection of posets with Alexandrov topologies.
  \item Appendix~\ref{secMathSheaves}. We proceed to diagrams over posets and sheaves over topologies, and describe their connection. Global sections and functorial properties of sheaves are defined.
  \item Appendix~\ref{secMathCohomology}. We recall the basic machinery of homological algebra: cochain complexes, exact sequences. We construct enough injective objects in the category of diagrams over a poset, and define sheaf cohomology as right derived functors of the global sections. The classical constructions of cochain complexes are presented: Roos complex (also known as standard simplicial resolution or cobar construction) and cellular cochain complex. In subsection~\ref{subsecMathOneShot} and~\ref{subsecMathMinimalComputations} we prove new results. A notion of one-shot sheaf cohomology computation is introduced; we prove that a poset $S$ allows for one-shot computation if and only if $S$ is a homology Morse cell poset. For an arbitrary poset $S$ we prove a theoretical bound for the complexity of sheaf computations on $S$, which is similar in spirit to Morse inequalities. We prove the bound is exact by introducing a minimal cochain complex on $S$. In case of cell posets, this minimal cochain complex coincides with cellular cochain complex.
  \item Appendix~\ref{secMathLaplacians}. Here we recall some basic linear algebra and differential equations, and introduce the notions of euclidean sheaves, Laplacians, higher-order Laplacians, energy functions, and heat/sheaf diffusion. All notions make sense for sheaves on arbitrary posets, not just cellular sheaves, which is important, if one wants to consider hypergraphs. In subsection~\ref{subsecMathDiffuseExamples} we provide examples how the mathematical formalism introduced in the previous subsection describes the internal mechanisms of neural networks used in topological deep learning. Namely, we describe, in a unified manner, diffusion in sheaf neural networks, and linear diffusion in sheaf hypergraph neural networks, and propose the way to things can be generalized.
  \item Appendix~\ref{secMathConnection}. We recall the notion of a connection sheaf (also known as a local system or a vector bundle) and provide basic mathematical claims about them.
\end{enumerate}

\begin{rem}
A small disclaimer: the mathematical appendices do not contain a list of notions despite their relevance to the subject of the review. We make no mention of Grothendieck topologies and sheaves over them; this subject can be found in the book of Rosiak~\cite{Rosiak}. We intentionally avoid mentioning cosheaves' homology (because this story is about reverting the order in a poset) as well as sheaf homology and cosheaf cohomology (which are more interesting); this subject is covered by Curry's thesis~\cite{Curry}. We don't review any of the topics of topological data analysis and applied topology in general, as there are many other sources~\cite{Zom,EdelHarer,Oudot,CarlssonVejdemo,DeyWang}. We also assume that the reader is familiar with the basics of category theory: objects, morphisms, categories, functors. More specific terms will be either explained or provided with a reference as they appear in the text. 
\end{rem}

\subsection{Posets}

Let $S$ be a partially ordered set (poset), or a preordered set (preposet). The non-strict order relation on $S$ is denoted $\leq$, while its strict version is denoted by $<$ (the latter makes sense for posets). 

\begin{con}\label{conCatOfPoset}
For each preposet $S$ define a small category $\cat(S)$, whose objects are elements from $S$, and there is only one morphism from $s_1$ to $s_2$ if $s_1\leq s_2$, and no morphisms otherwise.
\end{con}

A map between (pre)posets $f\colon S\to T$ is called \emph{monotonic}, or a morphism of (pre)posets if the inequality $s_1\leq s_2$ in $S$ implies the inequality $f(s_1)\leq f(s_2)$ in $T$. Each such morphism induces a functor $\cat(f)\colon \cat(S)\to \cat(T)$ between the corresponding small categories.

\begin{con}\label{conLowerUpperSetNotation}
For any (pre)poset $S$ and any element $s\in S$ we define
\[
S_{\geq s}=\{t\in S\mid t\geq s\},\qquad S_{>s}=\{t\in S\mid t>s\}.
\]
Similarly, we define $S_{\leq s}$, $S_{<s}$, etc. Poset $S$ with the dual order is denoted by $S^{\op}$.
\end{con}

\begin{con}\label{conPreposetToPoset}
For a preordered set $(S,\leq)$, let $\sim$ denote the equivalence relation defined by $s_1\sim s_2$ if and only if $s_1\leq s_2$ and $s_2\leq s_1$. The quotient set $S/\sim$ inherits the partial preorder ($[s_1]\leq[s_2]\Leftrightarrow s_1\leq s_2$), which is a partial order. This poset $S/\sim$ with this order is denoted $\bar{S}$. We have a couple of morphisms $S\leftrightarrows \bar{S}$, where one is the natural projection, and another one is inclusion of a representative in each class (for infinite sets the existence of such map follows from the axiom of choice). On the level of categories, $\cat(\bar{S})$ is the category obtained from $\cat{S}$ by taking one representative in each isomorphism class of objects.
\end{con}

\subsection{Alexandrov topologies}\label{subsecMathAlexandrov}

\begin{defin}\label{defAlexandrovTopology}
A topology $\Omega$ on a set $X$ is called an \emph{Alexandrov topology} if the intersection of any (possibly infinite) number of open subsets is open.
\end{defin}

\begin{defin}\label{defT0}
A topology $\Omega$ on a set $X$ \emph{satisfies the Kolmogorov axiom} (separation axiom) $T_0$ if, for each pair of points $x_1\neq x_2\in X$, there exists an open neighborhood $U$ of $x_1$ which does not contain $x_2$, or an open neighborhood $U$ of $x_2$ which does not contain $x_1$.
\end{defin}

The following correspondence is well known in general topology and mathematical logic.

\begin{prop}\label{propPosTop}
The category $\PrePos$ of preordered sets and their morphisms is equivalent to the category $\AlexTop$ of Alexandrov topological spaces with continuous maps. Under this equivalence, posets correspond to Alexandrov topologies satisfying separation axiom $T_0$.
\end{prop}

The proof can be found in the literature, for example,~\cite{Arenas}. We present the main constructions from the proof since they are used in what follows.

\begin{con}\label{conPrePosToSpace}
Let $S$ be a set with a preorder $\leq$. Consider the topology $\Omega_S$ (the family of subsets, which are called open) on the set $S$, which consists of upper order ideals: $U\in \Omega_S$ if and only if $x\in U$ and $y\geq x$ implies $y\in U$. The unions and intersections of any number of upper order ideals is an upper order ideal as well. Hence $X_S=(S,\Omega_S)$ is an Alexandrov topological space.
\end{con}

\begin{con}\label{conSpaceToPrepos}
Let us now review the reverse construction: given an Alexandrov topological space, we build a preorder. Let $X=(M,\Omega_X)$ be an Alexandrov topological space. The key property of Alexandrov spaces is the existence of an open neighbourhood of any point, minimal by inclusion. Indeed, let $x\in M$. Consider the family $\mathcal{U}=\{U\in \Omega_X\mid x\in U\}$ of all open neighbourhoods of a chosen point. Their intersection
\begin{equation}\label{eqDefUx}
U_x=\bigcap_{U\in \mathcal{U}}U
\end{equation}
is an open set by the definition of the Alexandrov topology. It is easy to see that this intersection belongs to any open neighbourhood of $x$, therefore we can call this intersection the minimal open neighbourhood of $x$. Now we can define the partial preorder on the set $M$: we put $x_1\leq x_2$ if $U_{x_1}\subseteq U_{x_2}$.
\end{con}

Checking that the Constructions ~\ref{conPrePosToSpace} and \ref{conSpaceToPrepos} are mutually inverses is a simple exercise. Both constructions are functorial (i.e. monotonic functions are continuous and vice versa), which is another exercise. The last exercise left to the reader is to check that partial orders correspond to $T_0$-spaces under this correspondence.

\begin{rem}\label{remMinNbhdIsACone}
In terms of the (pre)order on $S$, \emph{the minimal open neighbourhood} $U_x$ defined in Construction~\ref{conSpaceToPrepos} is an upper cone over $x$:
\[
U_x=\{s\in S\mid s\geq x\}.
\]
This is the minimal upper order ideal, which contains a point $x$.
\end{rem}

\subsection{Simplicial and cellular complexes}

In this section we mainly deal with finite structures: finite posets, finite simplicial and cell complexes. However, many definitions make sense for infinite sets.

\begin{defin}\label{definSimpComp}
A(n abstract) \emph{simplicial complex} on a (finite) vertex set $M$ is a collection $\ca{K}$ of subsets of $M$ such that
\begin{enumerate}
  \item any singleton\footnote{For our exposition we prefer to forbid ghost vertices.} $\{i\}$ for $i\in M$ lies in $\ca{K}$;
  \item if $I\in \ca{K}$ and $J\subset I$ then $J\in \ca{K}$.
\end{enumerate}
\end{defin}

The element $I\in \ca{K}$ is called \emph{a simplex}, the number $\dim I=\#I-1$ its dimension. The set of nonempty simplices of $\ca{K}$ is partially ordered by inclusion, we denote this poset by $\Cells(\ca{K})$. It is usually assumed that the vertex set $M$ is well ordered, hence may be identified with $[m]=\{1,\ldots,m\}$.

\begin{con}\label{conStandardGeomRealization}
The \emph{standard geometrical realization} of a simplicial complex $\ca{K}$ is the compact space
\begin{equation}\label{eqStdGeomRealiz}
|\ca{K}|=\bigcup_{I\in \Cells(\ca{K})}\triangle_I\subset\Ro^m,
\end{equation}
where $\triangle_I=\ConvHull\{e_i\mid i\in I\}$ and $e_1,\ldots,e_m$ is the standard basis of $\Ro^m$. Each $\Delta_I$ is a (geometrical) simplex, hence $|\ca{K}|$ is a geometrical simplicial complex, that is a space decomposed into cells isomorphic to simplices.
\end{con}

\begin{rem}\label{remAbuseSimpComp}
There exists a common abuse of notation: we call two simplicial complexes $\ca{K}_1,\ca{K}_2$ homeomorphic (homotopically equivalent, etc.) if their geometrical realizations are homeomorphic (resp. homotopically equivalent, etc.).
\end{rem}

\subsection{Order complex}

In combinatorial topology, there is a classical way of turning a (finite) poset into a simplicial complex.

\begin{defin}\label{definGeomRealPoset}
Let $S$ be a poset. \emph{The order complex of} $S$ is a simplicial complex $\ord S$ on the vertex set $S$ whose simplices are chains (i.e. well ordered subsets) of $S$. \emph{The geometric realization of} $S$ denoted by $|S|$ is the geometric realization $|\ord S|$ of the order complex $\ord S$.
\end{defin}

\begin{rem}\label{remSimplexOfOrd}
It follows from the definition, that a $j$-dimensional simplex $\sigma$ of $\ord S$ is encoded by a sequence $s_0<s_1<\ldots<s_j$ in $S$. Equivalently, it is a sequence of $j$ non-identical morphisms in $\cat(S)$.
\end{rem}

\begin{rem}\label{remBarycentric}
If $\ca{K}$ is a simplicial complex, then the simplicial complex
\[
\ca{K}'=\Cells\ord(\ca{K})
\]
is \emph{the barycentric subdivision} of $\ca{K}$. Simplicial complexes $\ca{K}$ and $\ca{K}'$ are homeomorphic. See details in~\cite[\S2.1.5]{KozlovBook}.
\end{rem}

There is a fundamental McCord's theorem, which establishes a connection between Alexandrov spaces and geometric realizations of finite posets.

\begin{rem}\label{remMcCord}
Let $S$ be an arbitrary finite poset. Then there is a canonical map:
\[
h_S\colon |S|\to X_S.
\]
If $x\in |S|$ is a point in the relative interior of a geometrical simplex $\triangle_{\{s_0,\ldots,s_k\}}$ corresponding to a chain $s_0<s_1<\cdots<s_k$, then $h_S(x)=s_k$, the maximal element of the chain in $S$. It is not difficult to check that $h_S$ is continuous (with respect to the Alexandrov topology on $X_S$ and Hausdorff topology on $|S|$). McCord's theorem~\cite{McCord} states that $h_S$ is a weak homotopy equivalence, i.e. induces isomorphism of all standard homotopy invariants. Among other applications, this theorem allows one to think of a finite topology as a nice substitute for a finite cell complexes with a comparable homotopical expressiveness. However, it should be noted that McCord theorem doesn't mention sheaves or their cohomology, so certain care is needed in this topic.
\end{rem}


\subsection{Cell complexes} 

Now we switch to cellular complexes. In our review this notion is essentially needed to define a cellular sheaf introduced by Zeeman and Shepard, see~\cite{Curry}. However, there is a certain discrepancy between formal definition of a (regular finite) CW-complex known in topology (see e.g.~\cite{bjorner1984poset}) and how this notion is actually used in applied papers. The intuition that matches both formal definition and the way how it is commonly understood is expressed by the following informal definition.

\begin{difin}\label{difinCWhausdorff}
A \emph{CW-complex} (or a \emph{cell complex} or a \emph{cellular complex}) is a filtered Hausdorff topological space
\[
X=\bigcup_nX_n, \qquad X_0\subseteq X_1\subseteq\cdots
\]
where $X_n$, $n\geq 0$, is called \emph{the $n$-skeleton of} $X$. The space is obtained inductively, that is $X=\lim_{t,\to} X^{(t)}$. The base of induction $X^{(0)}$ is the empty topological space with the trivial filtration; and the space $X^{(t)}$ is the result of attaching a topological disc $D^k$ to the previously constructed space $X^{(t-1)}$ along some attaching map $a_s\colon \dd D^k\to X^{(t-1)}_{k-1}$ in a way that the attached disc contributes to skeleta $X^{(t)}_n$ with $n\geq k$. A cell complex is called \emph{regular} if the attaching maps are homeomorphisms to the image.

In other words, a cell complex is a Hausdorff space constructed inductively from relatively simple pieces $(D^k,\dd D^k)$ called the \emph{cells} of $X$.
\end{difin}

\begin{rem}\label{remBuildingBlocks}
In the sense of the above definition, cell complexes generalize simplicial complexes: instead of simplices we are allowed to use e.g. cubes, pentagons, and almost anything, as long as each building block is homeomorphic to a topological disc $D^k$, and the attachment of this disc is made along its boundary $(k-1)$-sphere $\dd D^k$.
\end{rem}

\begin{rem}\label{remNovikovForman}
Definition~\ref{difinCWhausdorff} is bad from algorithmical perspective since it operates with (usually infinite) Hausdorff spaces. If one chooses to encode all spaces $X_n^{(s)}$ as finite simplicial complexes and all the attachment maps $a_s$ as simplicial maps (or piecewise linear maps), the resulting structure may be too heavy to store and process. Besides, it is algorithmically impossible to say if a given collection of simplicial complexes $\{X_n^{(s)}\}$ is a cell complex or not due to Novikov's theorem~\cite{volodin1974sphere} which states that homeomorphism to a sphere is algorithmically undecidable. 
\end{rem}

\subsection{Cell posets} 

In the practice of machine learning, attachment maps are usually not considered at all. A cell complex is thought of as an abstract finite poset of cells satisfying certain restrictions. We refer to~\cite{bjorner1984poset} for the extensive review of the related notions and constructions~\footnote{It should be noted that Bj\"{o}rner's formal definition of a cell poset differs from the one given below.}. 

\begin{defin}\label{definGrading}\cite{StanleyComb}
A \emph{grading} on a (finite) poset $S$ is a function $\rk\colon S\to\Zo$ such that
\begin{enumerate}
  \item $\rk$ is strictly monotonic: if $s_1<s_2$ then $\rk s_1<\rk s_2$.
  \item $\rk$ agrees with covering relation: if $s_1<s_2$ is a covering relation i.e. there is no $t$ such that $s_1<t<s_2$, then $\rk s_2=\rk s_1+1$.
  \item All minimal elements have rank $0$.
\end{enumerate}
A pair $(S,\rk)$, where $\rk$ is a grading on $S$, is called \emph{a graded poset}.
\end{defin}

\begin{rem}\label{remRankUnique}
It can be easily proven that whenever a poset $S$ has a grading, this grading is unique. So far, grading is a property of a poset rather than a complementary structure. A poset is \emph{graded} if, for any $s\in S$, all maximal (saturated) chains descending from $s$:
\[
s_1<s_2<\cdots<s_k<s
\]
have the same length $k$. In this case we can set $\rk s=k$. However, in practice it is more convenient to store the rank function as a complementary data structure.
\end{rem}

\begin{ex}\label{exGradingExistsNot}
The poset $\Cells(\ca{K})$ of nonempty simplices of a simplicial complex $\ca{K}$ is naturally graded by $\rk I=\dim I$. Similarly, the poset of faces of any convex polytope is graded by their dimensions. Geometric lattices~\cite{CrapoComb} constitute another important class of graded posets. On the other hand, there exist posets which do not admit any grading at all, see e.g. the rightmost poset on Fig.~\ref{figPosets}.
\end{ex}

Cells of a cell complex will be encoded by elements of a poset. However, the condition that each cell is a disc attached along the boundary sphere should be somehow described in the language of posets. This is done in the standard way.

\begin{con}\label{conCones}
The \emph{cone} of an element $s\in S$ is the subposet $C(s)=S_{\leq s}=\{t\in S\mid t\leq s\}$. The \emph{boundary} of an element $s\in S$ is the subposet $\dd C(s)=S_{<s}=\{t\in S\mid t<s\}$. Notice that $S_{\leq s}$ is obtained from $S_{<s}$ by adding the greatest element, the point $s$ itself. Therefore, for the geometrical realizations we have $|C(s)|=\Cone|\dd C(s)|$, where $|\dd C(s)|$ is the base and $s$ --- the apex of a cone.
\end{con}

\begin{ex}\label{exDownConeSimplex}
If $S=\Cells(\ca{K})$, and $I=\{i_0,\ldots,i_k\}\in S$ is a simplex of dimension $k$, then $C(I)$ is the poset of faces of the simplex $I$, equiv. the poset of subsets of $I$, equiv. the boolean lattice of rank $k+1$ with the bottom element removed. Its boundary $\dd C(I)$ is the poset of proper faces of $I$, which coincides with $\Cells(\dd\triangle_I)$. This observation explains the name of $\dd C(s)$ in general.
\end{ex}

The above mentioned intuition leads to the following definition of a cell complex, given in the language of posets. We reserve the symbols $\ca{X}$, $\ca{Y}$, etc. for cell posets, and $\sigma$,$\tau$ etc. for their elements, the cells.

\begin{defin}\label{definCellPoset}
A graded poset $(\ca{X},\rk)$ is called a \emph{cell poset} if, for any $\sigma\in \ca{X}$, the boundary $|\dd C(\sigma)|$ is homeomorphic\footnote{We recall the standard formalism known in topology. The sphere $S^{-1}$ is the empty space. Any set homeomorphic to the empty space is the empty space.} to the sphere $S^{\rk\sigma-1}$. If this is the case, an element $\sigma$ of rank $k$ is called a $k$-\emph{dimensional cell}.
\end{defin}

\begin{rem}
As a consequence of Construction~\ref{conCones}, the cone $|C(\sigma)|$ of each cell $\sigma$ of a cell poset is homeomorphic to a disc $|C(\sigma)|=\Cone|\dd C(\sigma)|\cong \Cone S^{\rk \sigma-1}\cong D^{\rk \sigma}$. The pair $(|C(\sigma)|,|\dd C(\sigma)|)$ resembles the building block $(D^k,S^{k-1})$ of a cell complex in terms of Definition~\ref{difinCWhausdorff}.
\end{rem}

\begin{con}\label{conCellPosetSkeleton}
The subposet $\ca{X}_n=\{\sigma\in \ca{X}\mid \rk \sigma\leq n\}$ is called \emph{the $k$-skeleton} of a cell poset $\ca{X}$. It resembles the topological skeleton of a cell complex as made transparent in the proof of the following proposition.
\end{con}

The following claim provides the connection between CW-complexes and cell posets.

\begin{prop}\label{propCellPosetToCpx}
If $X$ is a regular cell complex, then the poset $\Cells(X)$ is a cell poset. If $\ca{X}$ is a cell poset, then the Hausdorff space $|\ca{X}|$ filtered by $|\ca{X}_n|$ is a regular cell complex.
\end{prop}

See~\cite[Prop.3.1]{bjorner1984poset} for the proof.
%
%

\subsection{Variations of cell posets}\label{subsecMathCellVariations}

When dealing with specific aspects of cell complexes, for example homotopical or homological, Definition~\ref{definCellPoset} can be relaxed in one way or another. For example, we may forget the requirement for all spaces to be filtered. This situation appears in Morse theory~\cite{MilnorMorse} (Morse decomposition of a manifold does not usually give the structure a CW-complex), and in its algebro-geometrical version Bialynicki-Birula theory \cite{BialynizkiDecomp} (the analogues of cell complexes are called affine pavings). The requirement that $|\dd C(\sigma)|$ is homeomorphic to a sphere may be weakened in several obvious ways.

\begin{defin}\label{definHomotopyCellPoset}
A graded poset $(\ca{X},\rk)$ is called a \emph{homotopy cell poset} if, for any $\sigma\in \ca{X}$, the boundary $|\dd C(\sigma)|$ is homotopy equivalent to the sphere $S^{\rk\sigma-1}$.
\end{defin}

\begin{defin}\label{definHomologyCellPoset}
A graded poset $(\ca{X},\rk)$ is called a \emph{homology cell poset} if, for any $\sigma\in \ca{X}$, the boundary $|\dd C(\sigma)|$ has the same $\Zo$-homology as the sphere $S^{\rk\sigma-1}$:
\[
H_i(|\dd C(\sigma)|;\Zo)\begin{cases}
                          \cong\Zo, & \mbox{if } i=\rk\sigma-1 \\
                          =0, & \mbox{otherwise}.
                        \end{cases}
\]
\end{defin}

Definition~\ref{definHomologyCellPoset} has an advantage over Definitions~\ref{definHomotopyCellPoset} and~\ref{definCellPoset}: the property of a graded poset $(\ca{X},\rk)$ to be a homology cell poset is algorithmically decidable, while homotopy cell posets and cell posets are not, see Remark~\ref{remNovikovForman}. Yet another way to weaken the definition of a cell poset is to drop the assumption of being graded. We will see in Subsection~\ref{subsecMathOneShot} that the basic properties of cellular sheaves and their cohomology remain the same if these more general notions are used instead of cell posets.

\begin{defin}\label{definMorseHomologyCellPoset}
A poset $\ca{X}$ is called a \emph{(homology) Morse cell poset} if, for any $\sigma\in \ca{X}$, the boundary $|\dd C(\sigma)|$ has the same singular homology as a sphere $S^{d_s-1}$. In this case, the elements $s\in S$ are called \emph{cells}, and the corresponding number $d_s$ is called \emph{the dimension of a cell} $s$.
\end{defin}

\begin{rem}\label{remHomologyInContext}
Depending on the context, singular homology is taken with specific coefficients. If the target abelian category is $\Abel=\Zo\Mod$, then the coefficients are taken in $\Zo$, so that $|\dd C(\sigma)|$ should have the same $\Zo$-homology as a sphere. If the target abelian category is the category $\ko\Vect$ of vector spaces over a field $\ko$, then $|\dd C(\sigma)|$ should have the same $\ko$-homology as a sphere. The same remark relates to Definition~\ref{definHomologyCellPoset} as well.
\end{rem}

The definition of Morse homotopy cell poset may be given in a straightforward way.

\begin{figure}
  \centering
  \includegraphics[scale=0.25]{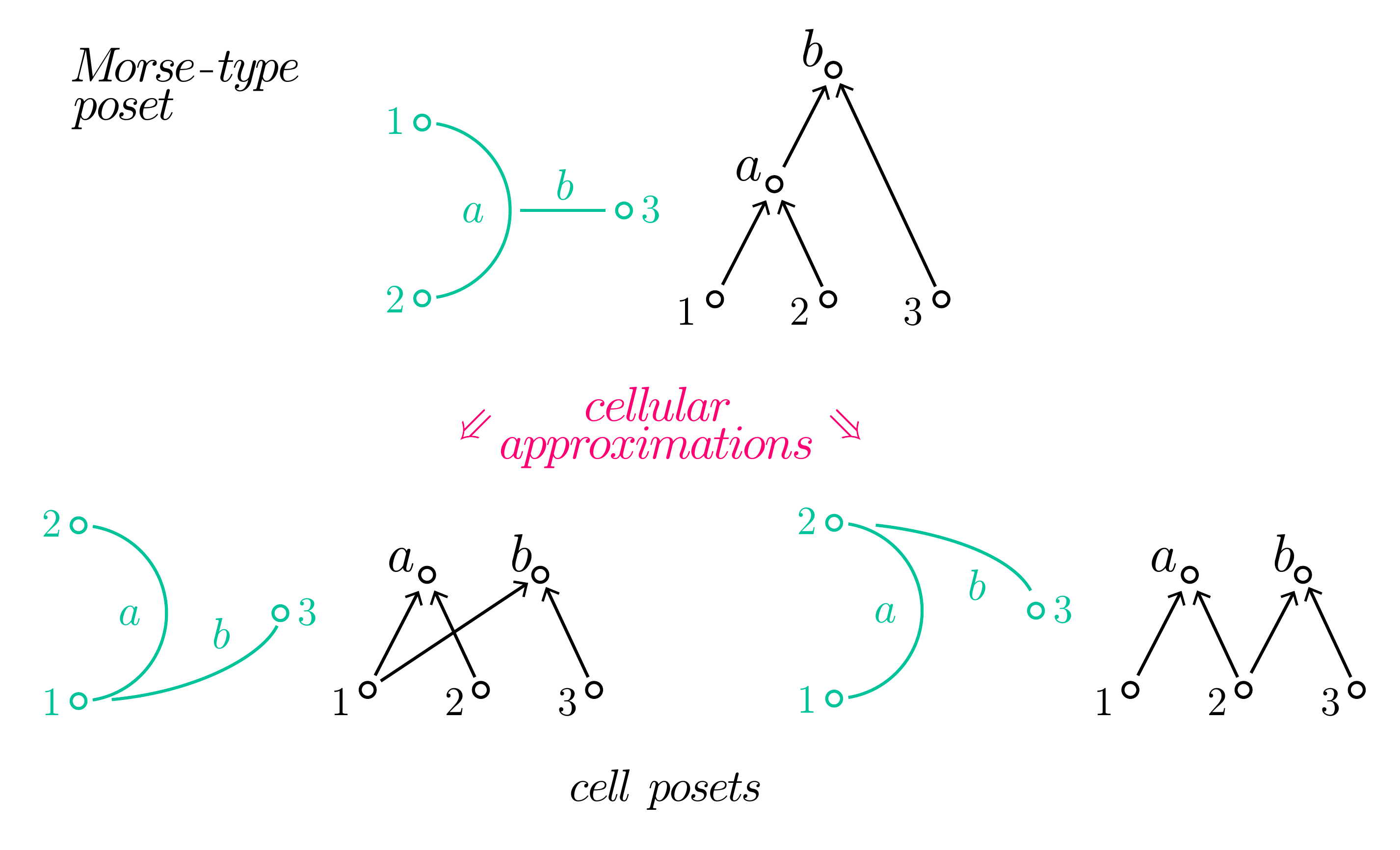}
  \caption{Example of Morse sell poset and the corresponding ``topological intuition''. There are two cellular approximations available for this poset.}\label{figElementaryMorse}
\end{figure}

\begin{ex}\label{exMorseExample}
Fig.~\ref{figElementaryMorse} shows an example of a Morse cell poset which is not graded. There is a geometrical intuition behind such a poset: assume we have 3 vertices $1,2,3$ and two edges $a,b$, and $a$ is attached to $1$ and $2$ in a natural way, while endpoints of $b$ are attached to some middle-point of $a$ and $3$ respectively. This geometrical figure is not a CW-complex, because it contradicts to the assumption that boundaries of cells are attached to lower-dimensional skeleta. However, such situations often (if not always) occur in the classical Morse theory: the (un)stable manifolds of the gradient flow of a Morse function do not form a CW-decomposition.
\end{ex}

\begin{rem}\label{remSubposet}
The posets defined by either of the Definitions~\eqref{definCellPoset},\ref{definHomotopyCellPoset},\ref{definHomologyCellPoset}, or~\ref{definMorseHomologyCellPoset} will be called cellular-like posets. If $\ca{X}$ is cellular-like, then lower order ideals $\ca{Y}\subseteq\ca{X}$ will be called cell subposets. The condition of lower order ideal means that whenever $\sigma\in\ca{Y}$ and $\tau<\sigma$, then $\tau\in\ca{Y}$. It is easily seen that cell subposets inherit the defining property of $\ca{X}$: for example, if $\ca{X}$ is a homology Morse cell poset, then so are its subposets $\ca{Y}$.
\end{rem}

A convention used in the following: notation $S_{\leq s}$, $S_{<s}$ is used for generic posets, while $C(s)$, $\dd C(s)$ is used for cell posets to be defined below. The letter $C$ stands here for ``cell'', and cell is usually assumed to be bounded by a sphere, or something like a sphere.


\section{Sheaves and diagrams}\label{secMathSheaves}

\subsection{Diagrams} 

In this section we recall the notion of a diagram on a poset, sheaf on a topological space, and describe their basic relations. We start with a more common and more intuitive notion of a diagram on a poset.

We recall some basic terminology from category theory to fix the notation. If $\Vv$ is a category, $\Ob(\Vv)$ denotes the class of objects of $\Vv$ and $\Hom_\Vv(a,b)$ the set of morphisms from an object $a$ to an object $b$. We usually abuse the notation and write $a\in\Vv$ instead of $a\in\Ob(\Vv)$. The symbol $\Cc$ will denotes generic small categories, while $\Vv$ is reserved for big categories such as
\begin{enumerate}
  \item the category $\Sets$ of sets and their maps;
  \item the category $\Abel$ of abelian groups and group homomorphisms;
  \item the category $\ko\Vect$ of vector spaces over a ground field $\ko$ and linear maps;
  \item the category $R\Mod$ of $R$-modules and module homomorphisms.
\end{enumerate}
Usually it is assumed that $\Vv$ is \emph{complete and cocomplete}, i.e. has all direct and inverse limits.

\begin{defin}\label{definDiagramOnPoset}
A \emph{diagram on a preposet} $S$ valued in a category $\Vv$ is a functor $D\colon \cat(S)\to \Vv$. 
\end{defin}

\begin{rem}\label{remDiagramInformally}
In simple words, a diagram on a preposet consists of two assignments. (1) To each element $s\in S$ we assign some object $D(s)=V_s$ of a category $\Vv$ (e.g., if the category $\Vv$ is $\ko\Vect$, $V_s$ is a vector space). (2) To each pair $s_1,s_2\in S$ such that $s_1\leq s_2$, we assign a map
\[
f_{s_1s_2}=D(s_1\leq s_2)\colon V_{s_1}\to V_{s_2}.
\]
These maps satisfy compositionality: $f_{ss}=\id_{V_s}$ for any $s\in S$, and $f_{s_1s_2};f_{s_2s_3}=f_{s_1s_3}$ for any $s_1\leq s_2\leq s_3$.
\end{rem}

All diagrams on $S$ form the category $\Diag(S,\Vv)$, in which the morphisms are natural transformations of functors.

\begin{rem}\label{remPreposetsNotNeeded}
Let $S$ be a preordered set, and $\bar{S}$ the corresponding quotient poset, as described in Construction~\ref{conPreposetToPoset}. Assume $D$ is a diagram on $S$. Whenever $s_1\sim s_2$ in $S$, we have two maps
\[
D(s_1\leq s_2)\colon D(s_1)\leftrightarrows D(s_2)\colon D(s_2\leq s_1)
\]
which are inverses of each other, hence provide the isomorphism between $D(s_1)$ and $D(s_2)$. Therefore we can meaningfully define a diagram $\bar{D}$ on $\bar{S}=S/\sim$ by setting $\bar{D}([s])=D(s)$, $\bar{D}([s_1]\leq [s_2])=D(s_1\leq s_2)$. Two maps $S\leftrightarrows \bar{S}$ described in Construction~\ref{conPreposetToPoset} prove equivalence of categories $\Diag(S;\Vv)=\Diag(\bar{S};\Vv)$. For this reason, when we speak of diagrams, sheaves or whatever, we can restrict consideration to posets. Every preordered set can be naturally transformed into a poset so that the category of diagrams do not change.
\end{rem}

\subsection{Presheaves and sheaves} 
Now we define the notion of a sheaf and presheaf. For a general topological space $X$ let $\OpSets(X)$ denote the set of all open subsets of $X$. The set $\OpSets(X)$ is partially ordered by inclusion.

\begin{defin}\label{definPresheaf}
A \emph{presheaf} on a topological space $X$ valued in $\Vv$ is a contravariant functor $\ca{F}\colon \cat(\OpSets(X)\setminus\{\varnothing\})^{\op}\to \Vv$.
\end{defin}

\begin{rem}\label{remPresheafInformally}
$\cat(\OpSets(X)\setminus\{\varnothing\})^{\op}$ is a category, opposite to the category of the poset of open subsets of $X$. In other words, presheaf can be considered as a rule which associates, to any non-empty open subset $U\in\Omega_X$, an object $V_U=\ca{F}(U)$ of the category $\Vv$, and, to each pair of open subsets $U_1\subseteq U_2$, a morphism
\[
r_{U_2\supseteq U_1}=\ca{F}(U_2\supseteq U_1)\colon V_{U_2}\to V_{U_1}.
\]
These maps are called \emph{restriction morphisms} of a presheaf. All restriction morphisms should satisfy the compositionality conditions: $r_{U\supseteq U}=\id_{V_U}$ and $r_{U_3\supseteq U_2};r_{U_2\supseteq U_1}=r_{U_3\supseteq U_1}$.
\end{rem}

\begin{rem}\label{remSectionOfPresheaf}
In the common categories $\Vv$ like $\Sets$ or $\ko\Vect$, the notion of an element of an object $V\in\Ob(\Vv)$ make sense. The elements $x\in\ca{F}(U)$ are called \emph{sections} of a presheaf $\ca{F}$ on a set $U$. If $U_1\subseteq U_2$ and $x\in\ca{F}(U_2)$, then $\ca{F}(U_2\supseteq U_1)(x)\in \ca{F}(U_1)$ is called the restriction of the section $x$ to the subset $U_1$ and is denoted by $x|_{U_1}$.
\end{rem}

\begin{defin}\label{definSheaf}
A presheaf $\ca{F}$ is called a \emph{sheaf} if the following two properties are satisfied:
\begin{enumerate}
  \item (Locality) Let $U$ be an open set, which is covered by a family of open subsets, $\{U_{i}\}_{i\in I}$, $U_{i}\subseteq U$ for each $i\in I$, and suppose we are given two elements $s,t\in \ca{F}(U)$. If $s|_{U_{i}}=t|_{U_{i}}$ for each $i\in I$, then $s=t$.
  \item (Gluing) Let $U$ be an open set and $\{U_{i}\}_{i\in I}$ be its open covering by subsets $U_{i}\subseteq U$ and let $\{s_{i}\in \ca{F}(U_{i})\}_{i\in I}$ be a collection of sections. Suppose the sections agree on the overlap of their domains, i.e. $s_{i}|_{U_{i}\cap U_{j}}=s_{j}|_{U_{i}\cap U_{j}}$ for each $i,j\in I$. Then there exists a section $s\in \ca{F}(U)$ such that $s|_{U_{i}}=s_{i}$ for each $i\in I$.
\end{enumerate}
\end{defin}

All $\Vv$-valued presheaves on a topology $X$ form the category $\PreShvs(X,\Vv)$, whose morphisms are natural transformations of functors from $\cat(\OpSets(X)\setminus\{\varnothing\})^{\op}$ to $\Vv$. Sheaves form a complete subcategory $\Shvs(X,\Vv)$ in the category $\PreShvs(X,\Vv)$.

\begin{defin}\label{defStalk}
If $\ca{F}$ is a presheaf on a topological space $X$ and $x\in X$, then
\[
\ca{F}_x=\lim\limits_{\substack{\longrightarrow\\U\in\Omega_X, U\ni x}}\ca{F}(U)\hspace{6pt}\in \Ob(\Vv)
\]
(the direct limit of values of the presheaf over all open subsets which contain $x$), is called a \emph{stalk} of the presheaf $\ca{F}$ at a point $x$.
\end{defin}

The definition makes sense only if the category $\Vv$ has direct limits. From the general categorical definition of a direct limit~\cite{MacLane} it follows that for each $U\ni x$, $U\in \Omega_X$ there exists a canonical morphism $\ca{F}(U\ni x)\colon \ca{F}(U)\to \ca{F}_x$.


\subsection{Relation between sheaves and diagrams} 

Apparently, diagrams on posets and sheaves on topologies are connected via the notion Alexandrov topology and Kan's extension. According to Proposition~\ref{propPosTop}, each poset $S$ corresponds to a $T_0$-Alexandrov topology $X_S$. There is a natural one-to-one correspondence between diagrams on $S$ valued in $\Vv$ and sheaves on $X_S$ valued in $\Vv$.

\begin{prop}\label{propDiagSheaf}
The category $\Diag(S,\Vv)$ of diagrams on a poset $S$ is naturally equivalent to the category $\Shvs(X_S,\Vv)$ of sheaves on the corresponding Alexandrov topology.
\end{prop}

For completenes, we review two standard constructions providing transformation of a diagram into a sheaf and vice versa.

\begin{con}\label{conSheafToDiag}
Let $\ca{F}$ be a presheaf on $X_S$. Recall that for each point $x\in S$ there is a minimal open neighbourhood $U_x$ in the topology on $X_S$ (see Construction~\ref{conSpaceToPrepos} and Remark~\ref{remMinNbhdIsACone}). If $x\leq y$, then we have $U_x\supseteq U_y$. Define a diagram $F$ on $S$ by
\[
F(x)=\ca{F}(U_x),\qquad F(x\leq y)=\ca{F}(U_x\supseteq U_y).
\]
It is easy to check that this diagram on $S$ is well-defined.
\end{con}

\begin{rem}
For a presheaf $\ca{F}$ on the Alexandrov topology we have $\ca{F}_x=\lim\limits_{\substack{\longrightarrow\\x\in U}}\ca{F}(U)=\ca{F}(U_x)$ since $U_x$ is the minimal open set containing a point $x$, thus is the terminal object of the diagram, on which the limit is taken. Therefore, Construction~\ref{conSheafToDiag} means that $F$ is actually the diagram of stalks of a presheaf $\ca{F}$.
\end{rem}

\begin{con}\label{conDiagToSheaf}
Let $D$ be a diagram on $S$. We construct the corresponding sheaf $\ca{D}$ on the Alexandrov space $X_S$ with $\Omega_S=\OpSets(X_S)$ the collection of upper order ideals. For each open set $U\in\Omega_S$ define
\begin{equation}\label{eqSheafFromDiagram}
\ca{D}(U)=\lim\limits_{\substack{\leftarrow \\ s\in U}}D(s),
\end{equation}
--- the inverse limit of the restriction of the diagram $D$ to a subcategory $\cat(U)$. For each pair $U_1\subset U_2$ we define the map $\ca{D}(U_2\supset U_1)\colon \ca{D}(U_2)\to \ca{D}(U_1)$ in a natural way, by the universal property of the inverse limit. The uniqueness of the universal object and the corresponding morphisms guarantee that $\ca{D}$ is a presheaf on $X_S$ (i.e. a commutative diagram on $\cat(\Omega_S\setminus\{\varnothing\})^{\op}$). Moreover, simple arguments involving the uniqueness and naturality of the inverse limit establish the properties of locality and gluing, which ensure that $\ca{D}$ is a sheaf on $X_S$. We call it \emph{the sheaf of sections} of the diagram $D$.
\end{con}

\begin{rem}
In more abstract terms, Construction~\ref{conDiagToSheaf} can be defined by the right Kan extension, see~\cite[Theorem 4.2.10]{Curry}.
\end{rem}

Checking that Constructions~\ref{conSheafToDiag} and~\ref{conDiagToSheaf} define the bijective correspondence between sheaves on $X_S$ and diagrams on $S$ is yet another simple exercise for the reader. Functoriality of both constructions is easily checked easily as well proving Proposition~\ref{propDiagSheaf}.

\begin{rem}\label{remSheafification}
We note that Construction~\ref{conSheafToDiag} is formulated for an arbitrary presheaf $\ca{F}$. Let us take a presheaf $\ca{F}$ on an Alexandrov space $X$ and build the diagram $F$ of its stalks by Construction~\ref{conSheafToDiag}, and then build its sheaf of sections $\overline{\ca{F}}$ by Construction~\ref{conDiagToSheaf}. Then the sheaf $\overline{\ca{F}}$ has the same stalks as the presheaf $\ca{F}$. Generally speaking, a sheaf with such properties is called the \emph{sheafification} of a presheaf $\ca{F}$. It is the left adjoint functor of the inclusion functor from $\Shvs(X;\Vv)$ to $\PreShvs(X;\Vv)$.
\end{rem}

\subsection{Global sections}

\begin{con}\label{conGlobalSections}
Let $\ca{F}\in\Shvs(X_S;\Vv)$ be a sheaf. The value
\[
\Gamma(X_S;\ca{F})=\ca{F}(X_S)\in\Vv
\]
of this sheaf on the whole space $X_S$ is called the set\footnote{Formally it is not the set, but rather an object of a category $\Vv$. However, since most common examples of categories $\Vv$ are concrete, it is not a big mistake to call it a set.} of global sections of the sheaf $\ca{F}$. From the definition of a sheaf it follows that global sections are functorial: if we have a morphism $f\colon \ca{D}\to\ca{F}$ of sheaves, we naturally have a map $\Gamma(X_S;f)\colon \ca{D}(X_S)\to\ca{F}(X_S)$ of their values. Henceforth, we have a functor
\begin{equation}\label{eqFunctorOfGlobalSects}
\Gamma(X_S,\cdot)\colon \Shvs(X_S;\Vv)\to \Vv,
\end{equation}
which is called \emph{the functor of global sections}.
\end{con}

\begin{con}\label{conInverseLimit}
Proposition~\ref{propDiagSheaf} states $\Shvs(X_S;\Vv)=\Diag(S;\Vv)$. In the language of diagrams over $S$, we have the following description of the global sections. Let $D\in \Diag(S;\Vv)$ be a diagram. Then $\Gamma(S,D)$ is the value of the corresponding sheaf $\ca{D}$ on the space $X_S$. This value is defined as the inverse limit
\[
\Gamma(S,D)=\lim\limits_{\substack{\leftarrow \\ s\in S}}D(s)
\]
according to the general Construction~\ref{conDiagToSheaf}. This is the reason why the functor of global sections is called \emph{the inverse limit functor} in topological and algebraical sources, see~\cite{careil1956homalg, grothendieck1957tohoku}.
\end{con}

\begin{rem}\label{remConcreteLimSets}
We provide an explicit construction of the inverse limit in the category $\ko\Vect$ in order to make formula~\eqref{eqSheafFromDiagram} more meaningful for readers not familiar with abstract nonsense. Essentially the procedure of taking inverse limits is the most essential ingredient in the applications of sheaves: this part of the theory deals with solving equations. We have
\begin{equation}\label{eqInvLimit}
\ca{D}(U)=\lim\limits_{\substack{\leftarrow \\ s\in U}}D(s)=\left\{(x_s\mid s\in U)\in\prod\nolimits_{s\in U}D(s)\mid D(s_1\leq s_2)x_{s_1}=x_{s_2} \text{ for any }s_1\leq s_2 \right\}.
\end{equation}
In other words, a section $x\in \ca{D}(U)$ on an open set $U$ coincides with a collection of vectors $x_s\in D(s)$, one for each $s\in U$, which are compatible with respect to the maps of the diagram. We also call such collection of vectors coherent. In the form~\eqref{eqInvLimit}, the restriction map $\ca{D}(U_2\supset U_1)$ is written in an straightforward manner, because the family of compatible stalks on the set $U_2$ is also compatible on its subset $U_1$.
\end{rem}

We give an important methodological remark, to which we refer in the analysis of the sheaves' applications.

\begin{rem}\label{remLimsToProductsEqualizers}
A construction similar to~\eqref{eqInvLimit} makes sense in other common categories such us $\Sets$, $R\Mod$, $\Top$, etc. In more general categorical terms, there is a standard statement (see e.g.~\cite[\S V.2]{MacLane}) that procedure of taking inverse limits (over small or finite, respectively, categories) reduces to taking small or finite, respectively, products and equalizers. In practice, this means that dealing with general data-structures $\Vv$ as the values of a (finitely supported) sheaf $\ca{F}$, the generation of a global section $\ca{F}(U)$ reduces to two procedures as follows.
\begin{enumerate}
  \item Generation of finite products. Usually not a problem, since Cartesian products are supported by all common data structures. E.g. concatenation of vectors is ubiquitous in data science, and is rarely explicitly mentioned at all.
  \item Generation of equalizers. In practice this means solving (systems of) equations of the form ``Find $x\in V$, such that $F(x)=G(x)$'' where $V,W\in \Ob(\Vv)$ and $F,G\in \Hom(V, W)$ are given. This step is usually hard, but it makes the theory substantial. The solution of equations (as a process) is the substantial part of sheaves' inspired architectures of neural networks, as defined in subsection~\ref{subsecSheafLearning} and reviewed in Section~\ref{secReviewShvsML}.  
\end{enumerate}
\end{rem}


\subsection{Functorial properties}\label{subsecMathFunctorial}

Here we discuss how the maps of spaces and posets interact with sheaves and diagrams. We won't retell the whole six-functor formalism of Grothendieck here and restrict to the necessary minimum. In general, a continuous map of topological spaces $f\colon X\to Y$ induces direct and inverse image functors on categories of sheaves
\[
f_*\colon \Shvs(X;\Vv)\to \Shvs(Y;\Vv), \qquad f^*\colon \Shvs(Y;\Vv)\to \Shvs(X;\Vv).
\]
For the details of this construction in general, we refer to~\cite{Iversen}. In case of posets the definition of the functor $f^*$ (inverse image of a sheaf) is significantly simplified at the conceptual level and, as a result, finds numerous applications in applied topology and contemporary approaches to neural networks.

Let $f\colon S\to T$ be a morphism of posets (equiv, a continuous map of the corresponding Alexandrov topological spaces), see Proposition~\ref{propPosTop}.

\begin{defin}\label{defDirectImage}
Let $\ca{F}\in\Shvs(X_S;\Vv)$ be a sheaf on the Alexandrov space $X_S$. Then its \emph{direct image} is a presheaf $f_*\ca{F}$ on the topological space $X_T$ whose values on open subsets $U\in\OpSets(X_S)$ and inclusions $U_1\supseteq U_2$ are defined by the following formulas:
\[
f_*\ca{F}(U)=\ca{F}(f^{-1}(U)),\qquad f_*\ca{F}(U_1\supseteq U_2)=\ca{F}(f^{-1}(U_1)\supseteq f^{-1}(U_2)).
\]
\end{defin}

A quite standard proposition states that the direct image $f_*\ca{F}$ is a sheaf. The map $f_*$ is functorial, i.e. it is a well-defined functor from the category $\Shvs(X_S;\Vv)$ to the category $\Shvs(X_T;\Vv)$. The inverse image functor is more intricate, since an image of an open set under a continuous map may fail to be open. However, for posets and Alexandrov topologies, the construction of the inverse image functor is simplified due to the equivalence between sheaves and diagrams over posets. Assume again that $f\colon S\to T$ is a morphism of posets and $\ca{D}\in\Shvs(X_T;\Vv)$ is a sheaf on the topological space $X_T$. According to Proposition~\ref{propDiagSheaf}, sheaf $\ca{D}$ corresponds to a diagram $D\colon \cat(T)\to\Vv$.

\begin{defin}\label{defInverseImage}
The \emph{inverse image} of a diagram $D$ on $T$ (or the corresponding sheaf $\ca{D}$ on $X_T$) under a morphism $f\colon S\to T$ is
a diagram $f^*D$ on the poset $S$ (or the corresponding sheaf $f^*\ca{D}$ on a space $X_S$), defined by the formulas\footnote{In algebraic geometry there is a difference between $f^{-1}$ and $f^*$, and here we are actually defining $f^{-1}$. Since we do not work with ringed spaces in this review, we prefer to keep notation simple and use symbol $f^*$ for inverse image.}:
\[
f^*D(s)=D(f(s)),\qquad f^*D(s\leq s')=D(f(s)\leq f(s')).
\]
\end{defin}

It can be seen that $f^*$ is functorial, so that we have a functor $f^*\colon\Shvs(X_T;\Vv)\to \Shvs(X_S;\Vv)$. The proof of the next statement can be found in~\cite[Theorem 5.3.1]{Curry} for sheaves on finite spaces, or in~\cite[Theorem 4.8]{Iversen} for general topological spaces.

\begin{prop}\label{propAdjointPair}
The functors of direct and inverse image form an adjoint pair. This means there is a natural bijection
\[
\Hom_{\Shvs(X_T,\Vv)}(\ca{D},f_*\ca{F})=\Hom_{\Shvs(X_S,\Vv)}(f^*\ca{D},\ca{F})
\]
for any $\ca{F}\in\Shvs(X_S,\Vv)$ and $\ca{D}\in\Shvs(X_T,\Vv)$.
\end{prop}

Recalling Proposition~\ref{propDiagSheaf}, we have a sequence of natural bijections
\begin{multline}\label{eqAllBijections}
\Hom_{\Diag(T,\Vv)}(D,f_*F)=\Hom_{\Shvs(X_T,\Vv)}(\ca{D},f_*\ca{F})=\\
\Hom_{\Shvs(X_S,\Vv)}(f^*\ca{D},\ca{F})=\Hom_{\Diag(S,\Vv)}(f^*D,F).
\end{multline}

\section{Sheaf cohomology}\label{secMathCohomology}

This section describes the definition of sheaf cohomology, and introduces various computational approaches to cohomology calculations. An important role here is played by additive, and more generally abelian categories, and additive functors between them. The basic examples of abelian categories are the categories $\Abel$ of abelian groups, the category $\ko\Vect$ of vector spaces over a field $\ko$, or the category $R\Mod$ of modules (e.g. left modules) over a ring $R$. For a general introduction to abelian categories we refer to e.g.~\cite{GelMan}. We are mainly interested in Grothendieck categories, the particularly well-behaved subclass of abelian categories, to which the listed examples belong~\cite[Ch.5,\SS8-9]{BucurDeleanu}.

\subsection{Recap on (co)chain complexes}\label{subsecMathCochainRecap}

In this subsection we review the basic toolbox of homological algebra: cochain complexes and exact sequences.

\begin{defin}\label{definCochainCpx}
Assume $\Vv$ is an abelian category. The sequence
\begin{equation}\label{eqCochainCpx}
0\to C^0\stackrel{d_0}{\to} C^1 \stackrel{d_1}{\to} C^2 \stackrel{d_2}{\to}\cdots
\end{equation}
of objects of $\Vv$ and morphisms between them is called a \emph{(connective) cochain complex}\footnote{We consider only connective cochain complexes, i.e. indexed with nonnegative integers, because only these complexes appear in the review.} if $d_j;d_{j+1}=0$ for any $j=0,1,\ldots$ (sometimes the numeration starts with $-1$). For short, a cochain complex~\eqref{eqCochainCpx} is denoted $(C^*,d)$.
\end{defin}

The definition implies the natural inclusion $p_j\colon \im d_j\hookrightarrow \Ker d_{j+1}$.

\begin{defin}\label{definCohomologyCochain}
The \emph{cohomology of the cochain complex} $(C^*,d)$ are defined as the objects
\[
H^j(C^*,d)=\Ker d_{j+1}/\im d_j = \Coker p_j \mbox{ for }j=0,1,\ldots
\]
of the category $\Vv$, indexed by nonnegative numbers.
\end{defin}

A cochain complex $(C^*,d)$ --- or any sequence of the form~\eqref{eqCochainCpx} --- is called \emph{exact} (resp. exact at position $i$), if $H^j(C^*,d)=0$ for any $j$ (resp. for $j=i$). Equivalently, $\im d_j=\Ker d_{j+1}$.

\begin{con}\label{conCategoryCochain}
Connective cochain complexes over $\Vv$ form a category, in which the maps (called \emph{cochain maps}) are defined componentwise. A cochain map $f\colon (C^*,d_C)\to (B^*,d_B)$ is defined if the maps $f_j\colon C^j\to B^j$ are defined for any $j$, commuting with the differentials $f_j;d_b=d_C;f_{j+1}$. Each such map induces the map of cohomology $f_*\colon H^j(C^*,d_C)\to H^j(B^*,d_B)$. Let us denote the category of cochain complexes over $\Vv$ by $\Cochain(\Vv)$. The category $\Cochain(\Vv)$ is an abelian category itself with kernels and cokernels defined componentwise in a straightforward manner.
\end{con}

The following lemma is classical, see e.g.~\cite[Lemma 1.3.2]{weibel1994homalg}.

\begin{lem}[Zig-zag lemma]\label{lemZigZag}
A short exact sequence
\[
0\to (C^*,d_C)\to (B^*,d_B)\to (A^*,d_A)\to 0
\]
in $\Cochain(\Vv)$ induces the long exact sequence of cohomology
\begin{equation}\label{eqLongExactCochain}
\begin{tikzpicture}[descr/.style={fill=white,inner sep=1.5pt}]
        \matrix (m) [
            matrix of math nodes,
            row sep=1em,
            column sep=2.5em,
            text height=1.5ex, text depth=0.25ex
        ]
        { 0 & H^0(C^*,d_C) & H^0(B^*,d_B) & H^0(A^*,d_A) & \\
            & H^1(C^*,d_C) & H^1(B^*,d_B) & H^1(A^*,d_A) & \\
            & H^2(C^*,d_C) & H^2(B^*,d_B) & H^2(A^*,d_A) & \cdots\\
        };

        \path[overlay,->, font=\scriptsize,>=latex]
        (m-1-1) edge (m-1-2)
        (m-1-2) edge (m-1-3)
        (m-1-3) edge (m-1-4)
        (m-1-4) edge[out=355,in=175] node[descr,yshift=0.3ex] {} (m-2-2)
        (m-2-2) edge (m-2-3)
        (m-2-3) edge (m-2-4)
        (m-2-4) edge[out=355,in=175] node[descr,yshift=0.3ex] {} (m-3-2)
        (m-3-2) edge (m-3-3)
        (m-3-3) edge (m-3-4)
        (m-3-4) edge (m-3-5);
\end{tikzpicture}
\end{equation}
\end{lem}

\subsection{Sheaf cohomology as derived functors}\label{subsecMathCohomologyDerived}

Sheaf cohomology, as they are understood in algebraic geometry and homological algebra, are the derived functor of the global sections' functor. In this subsection we give the necessary definitions, constructions, and explain the specifics of finite topological spaces in terms of this definition. This subsection is entirely covered by Curry's thesis~\cite{Curry}: it is provided here for convenience and also to fix the notation for some important objects to be defined later. 

Let $\Vv$ be an abelian category. Then, for any topological space $X$, the category $\Shvs(X,\Vv)$ is an abelian category as well. If $X=X_S$ is the Alexandrov space corresponding to a poset $S$, then we have an abelian category of diagrams $\Diag(S,\Vv)=\Shvs(X_S,\Vv)$ (see Propositions~\ref{propPosTop} and~\ref{propDiagSheaf}).

\begin{con}\label{conDiagIsAbelian}
In terms of diagrams over posets, the characteristic properties of the abelian category $\Diag(S,\Vv)$ are the following. As the illustrative example one can imagine diagrams of vector spaces $\Diag(S,\ko\Vect)$.
\begin{enumerate}
  \item The \emph{zero diagram} $0\in \Diag(S,\Vv)$ is the diagram made of zero vector spaces and zero maps.
  \item There exist finite \emph{biproducts}, also called \emph{direct sums}. If $D_1,D_2\in \Diag(S,\Vv)$ are two diagrams, then their direct sum $D_1\oplus D_2$ is defined component-wise:
      \[
      (D_1\oplus D_2)(s)=D_1(s)\oplus D_2(s)
      \]
      and
      \[
      (D_1\oplus D_2)(s\leq t)=D_1(s\leq t)\oplus D_2(s\leq t).
      \]
  \item \emph{Kernel} and \emph{cokernel} of a morphism of diagrams are defined component-wise. If $G\colon D_1\to D_2$ is a morphism of diagrams, then the kernel $K\colon \Ker G\hookrightarrow D_1$ is defined by
      \[
      (\Ker G)(s)=\Ker G(s),\mbox{ the maps } (\Ker G)(s\leq t)\colon \Ker G(s)\to \Ker G(t)
      \]
      are naturally induced by $D_1(s\leq t)$, and $K(s)\colon (\Ker G)(s)\hookrightarrow D_1(s)$ is a natural inclusion. Cokernels are completely similar.
  \item Every monomorphism is a kernel of its cokernel, and every epimorphism is a cokernel of its kernel. Again, easily verified component-wise.
\end{enumerate}
\end{con}


The next proposition follows from the 3-rd item of Construction~\ref{conDiagIsAbelian}. It constitutes an important property of diagrams over posets: exactness of sequences of sheaves can be verified stalk-wise. 

\begin{prop}\label{propExactSequenceDiagrams}
A sequence $0\to D_0\to D_1\to D_2 \to\cdots$ of diagrams over a poset $S$ is exact if and only if the corresponding sequence of stalks
\[
0\to D_0(s)\to D_1(s)\to D_2(s)\to\cdots
\]
is exact for each $s\in S$.
\end{prop}

\textbf{Sheaves have enough injectives.} In order to deal with derived functors, one needs to assure that the abelian category of sheaves has enough injectives. For a general definition of injective object in an abelian category see~\cite[\S2.3]{weibel1994homalg}. An abelian category $\ST{A}$ is said to \emph{have enough injectives} if, for any object $A\in\ST{A}$ there exists a monomorphism from $A$ to an injective object of $\ST{A}$. The categories $\ko\Vect$, $\Abel$, $R\Mod$ are known to have enough injectives~\cite[\S 2.3]{weibel1994homalg}. For example, in the ``most applied'' category $\ko\Vect$ every object is injective. This simplifies many arguments and constructions in this case.

\begin{prop}[{\cite[Thm.1.10.1 and Prop.3.1.1]{grothendieck1957tohoku}}]\label{propEnoughInjectives}
If $\Vv$ has enough injectives, then $\Shvs(X,\Vv)$ also has enough injectives.
\end{prop}

Below, we provide the proof for finite topologies, following Curry~\cite[Def.7.1.3]{Curry}.

\begin{con}\label{conInjSheaves}
Let $s\in S$ and $V\in\Vv$. Consider the diagram $\low{s}{V}$ on $S$, defined by
\begin{equation}\label{eqDefinElementaryInj}
\low{s}{V}(t)=\begin{cases}
               V, & \mbox{if } t\leqslant s \\
               0, & \mbox{otherwise}.
             \end{cases}
\end{equation}
Mappings between the non-zero elements of the diagram are set to be the identity isomorphisms. The diagrams of this form, and their corresponding sheaves, will be called \emph{cone-shaped sheaves} on $S$ (with apex $s$ and value $V$). Notice the analogy with the definition of the cone of an element, Construction~\ref{conCones}. It can be seen that $\low{s}{V}$ coincides with \emph{the skyscraper sheaf} on $X_S$, concentrated in $s$, known in the classical homological algebra~\cite[Ex.2.3.12]{weibel1994homalg}.
\end{con}

\begin{lem}\label{lemElementaryInjective}
If $W$ is an injective object in $\Vv$, then $\low{s}{W}$ is an injective object in $\Diag(S,\Vv)$ (equiv. an injective sheaf on the space $X_S$).
\end{lem}

\begin{proof}
Let $\pt$ be the one-point poset. The category of sheaves on $\pt$ can be naturally identified with the category $\Vv$. Thus, $W$ can be considered as injective sheaf on $\pt$. Consider the inclusion map $i_s\colon \pt\to S$, $i_s(\pt)=s$. It is easily verified that $\low{s}{W}=(i_s)_*W$. Hence
\begin{equation}\label{eqIdentificationHoms}
\Hom_{\Diag(S,\Vv)}(F,\low{s}{W})= \Hom_{\Diag(S,\Vv)}(F,(i_s)_*W)=\Hom_{\Diag(\pt,\Vv)}(i_s^*F,W)=\Hom_{\Vv}(F(s),W).
\end{equation}
It is known~\cite[Lm.2.3.4]{weibel1994homalg} that an object $I$ of an abelian category $\ST{A}$ is injective if and only if the functor $\Hom_{\ST{A}}(\cdot, I)$ is exact. Taking stalks at the point $s$ is an exact functor from $\Diag(S,\Vv)$ to $\Vv$, see Proposition~\ref{propExactSequenceDiagrams}. Therefore, since $W$ is an injective object of $\Vv$, the functor
\[
(\ca{F}\in\Diag(S,\Vv))\mapsto \Hom_{\Vv}(\ca{F}(s),W)=\Hom_{\Diag(S;\Vv)}(\ca{F},\low{s}{W})
\]
is exact, hence the diagram (the sheaf) $\low{s}{W}$ is injective.
\end{proof}

Proposition~\ref{propEnoughInjectives} is proved by the following construction, which is a particular case of the Godement's construction~\cite{Godement1973} applied to a finite topology~\footnote{It should be noticed that general Godement's construction works with direct products which, for infinite indexing sets, do not coincide with direct sums, or coproducts. Therefore the assumption of finiteness of $S$ is important in this and subsequent constructions.}.

\begin{con}\label{conGodementConstruction}
Let $D\in\Diag(S,\Vv)$ be an arbitrary diagram on a poset. For each element $s\in S$ consider a monomorphism $j_s\colon D(s)\to J_s$ to some injective object of $\Vv$ (which exists since we assumed $\Vv$ has enough injectives). Consider the following map
\begin{equation}\label{eqGodementPoset}
j\colon D\to \bigoplus_{s\in S}\low{s}{J_s}, \quad j=\bigoplus_{s\in S}\tilde{j}_s,
\end{equation}
where $\tilde{j}_s\colon D\to\low{s}{J_s}$ is the map which corresponds to $j_s\colon D(s)\to J_s$ via the natural correspondence described in~\eqref{eqIdentificationHoms}. The map $j$ is monomorphic stalk-wise: it is monomorphic on an element $s$ since already its summand $\tilde{j}_s$ (which is basically $j_s$) is monomorphic.

In short, we mapped a diagram $D$ into the direct sum of the cone-shaped diagrams valued by objects in which the values of $D$ inject.
\end{con}

\begin{rem}\label{remInjIsSumOfCones}
Construction~\ref{conGodementConstruction} asserts, that every sheaf embeds monomorphically into a direct sum of cone-shaped sheaves, which is injective. In the following, when we speak about injective sheaves on Alexandrov topologies, we usually mean ``sums of cone-shaped sheaves'' since no other concrete constructions of injective sheaves appear in our work.
\end{rem}

\textbf{Derived functors.} The most important functor in sheaf theory is the functor of global sections described in Construction~\ref{conGlobalSections}. If $\Vv$ is abelian, then the functor of global sections $\Gamma(X_S;\cdot)$ (or $\Gamma(S;\cdot)=\lim\limits_{\substack{\leftarrow, S}}(\cdot)$ in terms of posets) can be seen to be additive.

In homological algebra, given two sufficiently nice abelian categories $\Ww$, $\Vv$ and an additive functor $\Fu$ between them, it is possible to define an infinite sequence of functors, called the derived functors of $\Fu$. 
Some steps of this definition appear useful in the following, so we recall the abstract definitions with a moderate level of mathematical details. The main purpose of this rather classical description is to check its algorithmical feasibility for finite diagrams of finite-dimensional vector spaces.

\begin{defin}\label{definExactFunctor}
Let $\Fu\colon \Ww\to \Vv$ be an additive functor between two abelian categories. The functor $\Fu$ is called \emph{left exact} if every short exact sequence in $\Ww$
\[
0\to W_1 \to W_2\to W_3\to 0
\]
induces the sequence
\begin{equation}\label{eqLeftExact}
0\to \Fu(W_1) \to \Fu(W_2) \to \Fu(W_1)
\end{equation}
which is exact at all positions except the rightmost.
\end{defin}

Assume that $\Ww$ has enough injectives. Let us define the right derived functors $R^j\Fu\colon\Ww\to\Vv$. Actually, we will concentrate on constructing $R^j\Fu(W)$ for a given object $W\in\Ww$, not a morphism. The tedious check of functoriality is based on the horseshoe lemma and may be found elsewhere~\cite[Theorem 2.3.7]{weibel1994homalg}.

\begin{con}\label{conDerivedFunctorGeneral}
Let $W\in\Ww$ be an object. Since $\Ww$ has enough injectives, there exists a monomorphic map $\imath_0\colon M\to I_0$ into some injective object. Now let us take an injective hull of $\Coker \imath_0$, which gives a map $\imath_1\colon I_0\to I_1$. Proceeding the same way inductively, we get a potentially infinite exact sequence
\begin{equation}\label{eqInjectiveResolGen}
0\to W\stackrel{i_0}{\to}I_0\stackrel{i_1}{\to}I_1\stackrel{i_2}{\to}I_2\stackrel{i_3}{\to}\cdots
\end{equation}
in which all objects $I_j$ are injective. Such a sequence is called an \emph{injective resolution} of $W$; it may be non-unique. Now apply the functor $\Fu$ to the sequence~\eqref{eqInjectiveResolGen} cut at the leftmost position:
\begin{equation}\label{eqFuncOfInjective}
0\to\Fu(I_0)\stackrel{\Fu(i_1)}{\longrightarrow}\Fu(I_1)\stackrel{\Fu(i_2)}{\longrightarrow}\Fu(I_2)\stackrel{\Fu(i_3)}{\to}\cdots
\end{equation}
This sequence may fail to be exact, but it is a cochain complex in $\Vv$. The cohomology of this cochain complex (see Definition~\ref{definCohomologyCochain}) are called the \emph{values of right derived functors} of $\Fu$ on $W$:
\[
R^j\Fu(W)=H^j(\Fu(I^*,\imath)).
\]
\end{con}

For convenience we provide the list of basic claims about derived functors in general, the proofs can be found e.g. in~\cite[\S 2.5]{weibel1994homalg}.

\begin{prop}\label{propDerivedFunctorProps}
The following properties hold for the derived functors $R^j\Fu$.
\begin{enumerate}
  \item $R^j\Fu(W)$ are well-defined up to natural isomorphism. This basically means that these objects of $\Vv$ do not depend on the choice of injective resolution of $W$.
  \item $R^0\Fu=\Fu$.
  \item Each short exact sequence $0\to W_1\to W_2\to W_3\to 0$ in $\Ww$ induces the long exact sequence
\begin{equation}\label{eqLongExactDerived}
\begin{tikzpicture}[descr/.style={fill=white,inner sep=1.5pt}]
        \matrix (m) [
            matrix of math nodes,
            row sep=1em,
            column sep=2.5em,
            text height=1.5ex, text depth=0.25ex
        ]
        { 0 & R^0\Fu(W_1) & R^0\Fu(W_2) & R^0\Fu(W_3) & \\
            & R^1\Fu(W_1) & R^2\Fu(W_2) & R^1\Fu(W_3) & \\
            & R^2\Fu(W_1) & R^2\Fu(W_2) & R^3\Fu(W_1) & \cdots\\
        };

        \path[overlay,->, font=\scriptsize,>=latex]
        (m-1-1) edge (m-1-2)
        (m-1-2) edge (m-1-3)
        (m-1-3) edge (m-1-4)
        (m-1-4) edge[out=355,in=175] node[descr,yshift=0.3ex] {} (m-2-2)
        (m-2-2) edge (m-2-3)
        (m-2-3) edge (m-2-4)
        (m-2-4) edge[out=355,in=175] node[descr,yshift=0.3ex] {} (m-3-2)
        (m-3-2) edge (m-3-3)
        (m-3-3) edge (m-3-4)
        (m-3-4) edge (m-3-5);
\end{tikzpicture}
\end{equation}
  which extends the sequence~\eqref{eqLeftExact}.
\end{enumerate}
\end{prop}

\textbf{Sheaf cohomology.} In order for the definition~\ref{definCohomologyMainPart} of sheaf cohomology to make sense, a bit more preparatory work should be done.

\begin{prop}\label{propGlobSectLeftExact}
The functor $\Gamma(S,\cdot)$ is left exact.
\end{prop}

This is derived from the fact that $\Gamma(S;\cdot)$ is right adjoint to the constant sheaf functor, see~\cite[\S 2.6]{weibel1994homalg}. The abelian category $\Shvs(X_S,\Vv)$ has enough injectives by Proposition~\ref{propEnoughInjectives}, therefore the derived functors of the functor $\Gamma(X_S;\cdot)$ are well defined according to Construction~\ref{conDerivedFunctorGeneral}. The next definition duplicates Definition~\ref{definCohomologyMainPart} given in the main part of the text.

\begin{defin}\label{definSheafCohomologyDerived}
The $j$-th right derived functor $R^j\Gamma(X_S;\cdot)$ is called \emph{the functor of $j$-th cohomology of a sheaf}. We use the notation\footnote{The notation AG stands for algebraic geometry. It may rightfully refer to Alexander Grothendieck.}
\[
H^j_{AG}(X_S;\ca{F})=R^j\Gamma(X_S;\ca{F}).
\]
\end{defin}

Applying Proposition~\ref{propDerivedFunctorProps} to the functor $\Gamma(X_S;\cdot)\colon\Shvs(X_S;\Vv)\to\Vv$ of global sections, we obtain the main properties of the sheaf cohomology.

\begin{prop}\label{propCohomologyProperties}
Consider sheaves on the Alexandrov topology $X_S$ corresponding to a poset $S$. The following properties hold true for the cohomology.
\begin{enumerate}
  \item Cohomology modules $H_{AG}^j(X_S;\ca{F})$ are well-defined and functorial with respect to a sheaf-component.
  \item The 0-degree cohomology $H_{AG}^0(X_S;\ca{F})$ coincide with the global sections $\Gamma(X_S;\ca{F})$.
  \item Each short exact sequence $0\to \ca{F}_1\to \ca{F}_2\to \ca{F}_3\to 0$ of sheaves over $X_S$ induces long exact sequence in cohomology:
	\begin{equation}\label{eqLongExactDerivedCohomology}
\begin{tikzpicture}[descr/.style={fill=white,inner sep=1.5pt}]
        \matrix (m) [
            matrix of math nodes,
            row sep=1em,
            column sep=2.5em,
            text height=1.5ex, text depth=0.25ex
        ]
        { 0 & \Gamma(X_S;\ca{F}_1) & \Gamma(X_S;\ca{F}_2) & \Gamma(X_S;\ca{F}_3) & \\
            & H_{AG}^1(X_S;\ca{F}_1) & H_{AG}^1(X_S;\ca{F}_2) & H_{AG}^1(X_S;\ca{F}_3) & \\
            & H_{AG}^2(X_S;\ca{F}_1) & H_{AG}^2(X_S;\ca{F}_2) & H_{AG}^2(X_S;\ca{F}_3) & \cdots\\
        };

        \path[overlay,->, font=\scriptsize,>=latex]
        (m-1-1) edge (m-1-2)
        (m-1-2) edge (m-1-3)
        (m-1-3) edge (m-1-4)
        (m-1-4) edge[out=355,in=175] node[descr,yshift=0.3ex] {} (m-2-2)
        (m-2-2) edge (m-2-3)
        (m-2-3) edge (m-2-4)
        (m-2-4) edge[out=355,in=175] node[descr,yshift=0.3ex] {} (m-3-2)
        (m-3-2) edge (m-3-3)
        (m-3-3) edge (m-3-4)
        (m-3-4) edge (m-3-5);
\end{tikzpicture}
\end{equation}
\end{enumerate}
\end{prop}

Next, we provide more constructive details on injective resolutions of sheaves over Alexandrov spaces and their cohomology.

\begin{con}\label{conSheafCohomDerivedStyle}
Assume $D$ is a diagram over a poset $S$ (equiv. sheaf $\ca{D}$ over the Alexandrov space $X_S$), and as before, $\Vv$ is an abelian category with enough injectives. In the abelian category $\Diag(S,\Vv)=\Shvs(X_S,\Vv)$ there exists an injective resolution
\begin{equation}\label{eqInjSheaf}
0\to \ca{D}\to\ca{I}_0\to\ca{I}_1\to\ca{I}_3\to\cdots
\end{equation}
that is an exact sequence of sheaves, in which all $\ca{I}_j$ are injective. Moreover, according to Construction~\ref{conGodementConstruction} we may assume that each $\ca{I}_j$ is a direct sum of the sheaves of the form $\low{s}{W}$ with $s\in S$ and $W$ an injective object of $\Vv$. Now we apply the functor of global sections to the non-augmented resolution:
\begin{equation}\label{eqGlobSectionsOfInjRes}
0\to\ca{D}(X_S)\to\ca{I}_0(X_S)\to\ca{I}_1(X_S)\to\ca{I}_2(X_S)\to\cdots
\end{equation}
which, in the language of posets, the same as the following sequence
\begin{equation}\label{eqGlobSectionsOfInjResLimits}
0\to\lim\limits_{\leftarrow S} D\to \lim\limits_{\leftarrow S} I_0\to\lim\limits_{\leftarrow S} I_1\to\lim\limits_{\leftarrow S} I_2\to\cdots
\end{equation}
according to Construction~\ref{conGlobalSections}. The sequence~\eqref{eqGlobSectionsOfInjRes} is obviously a cochain complex in $\Vv$. Its cohomology are called the cohomology of the sheaf $\ca{D}$, or cohomology of category $\cat(S)$ with coefficients in a functor $D$ (or simply cohomology of a diagram $D$):
\begin{equation}\label{eqCohomOfSheafAG}
H^j_{AG}(X_S;\ca{D})=H^j_{AG}(S;D)=\Ker(d\colon \ca{I}_j(X_S)\to \ca{I}_{j+1}(X_S))/\im(d\colon \ca{I}_{j-1}(X_S)\to \ca{I}_j(X_S)).
\end{equation}
\end{con}

In view of this construction, it is useful to compute the global sections of the basic ``injective block'' $\low{s}{W}$ used in the construction of injective resolution.

\begin{lem}\label{lemGlobSecOfInjectives}
For a sheaf $\low{s}{W}$ defined in Construction~\ref{conInjSheaves}, the global sections $\Gamma(X_S;\low{s}{W})=\low{s}{W}(X_S)$ are naturally isomorphic to $W$. Moreover, if $s\leq t$ in $S$ and the map $g\colon \low{s}{W}\to \low{t}{V}$ is induced by $f\colon W\to V$, that is defined on stalks by
\[
g(s')=\begin{cases}
        f\colon W=\low{s}{W}(s')\to V=\low{t}{V}, & \mbox{if } s'\leq s \\
        0, & \mbox{otherwise}.
      \end{cases}
\]
then the induced map $g_*\colon \Gamma(X_S;\low{s}{W})\to \Gamma(X_S;\low{t}{V})$ coincides with $f\colon W\to V$.
\end{lem}

The proof is a simple exercise: on the universality of an inverse limit if it is proved in an arbitrary category, or on the application of Remark~\ref{remConcreteLimSets} if one prefers to stay concrete.

\begin{rem}\label{remOpinionLeader}
It follows from Construction~\ref{conSheafCohomDerivedStyle}, that higher cohomology $H^{>0}(S;\low{s}{W})$ of the basic injective sheaf vanish. Indeed, an injective resolution of $\low{s}{W}$ has the form
\[
0\to \low{s}{W}\to \low{s}{W}\to 0\to 0\to\cdots
\]
so the global sections of its reduced version give the cochain complex $0\to W\to 0\to\cdots$, with no higher order cohomology.

In subsection~\ref{subsecCohomologyMainPart} it was explained that sheaf cohomology provide a language to speak about ways of information distribution and copying in a network. In terms of these allusion, the skyscraper sheaf $\low{s}{W}$ encodes the situation, when we have a single node $s\in S$, the opinion leader, which possesses information $W$, and all lower nodes of a poset copy this information identically (follow the leader). A consensus state of such sheaf is determined entirely by a state $x_s\in \low{s}{W}(s)=W$ of the opinion leader hence we have $H^0(S;\low{s}{W})=\Gamma(S;\low{s}{W})\cong W$ as stated in Lemma~\ref{lemGlobSecOfInjectives}. Each node $t$ receives its state in $\low{s}{W}(t)$ in a single way: by copying it from $\low{s}{W}(s)$ if $t\leq s$, or by setting it to $0$ otherwise. Therefore, higher cohomology of $\low{s}{W}$ vanish. Although the implication ``injective $\Rightarrow$ acyclic'' is mathematically tautological, it is fairly consistent with the common-sense interpretation demonstrated in Example~\ref{exPathVsCircle}.

The procedure of constructing injective resolution can be informally understood as a decomposition of an arbitrary medium (sheaf) $D$ on $S$ into a combination of simpler media $\low{s}{W}$: these media have opinion leaders hence are easier to deal with\footnote{This certainly looks like a good reason for political engineers to start using injective resolutions in their work.}.
\end{rem}

\begin{rem}\label{remComputFeasibilityOfAGcohomology}
Assume that $S$ is finite, and $\Vv=\Abel$ is the category of abelian groups. Is it possible to algorithmically compute cohomology of a sheaf $\ca{F}$ on $X_S$? Potentially, yes, Construction~\ref{conSheafCohomDerivedStyle} describes the algorithm. However, there is one formal detail, which complicates the deal. Namely, we should specify a computationally feasible class of abelian groups. In topological practice, consideration is restricted to finitely generated abelian groups. However, unwinding Constructions~\ref{conSheafCohomDerivedStyle} and~\ref{conGodementConstruction}, we see the necessity to generate injective hulls in $\Vv$, which kicks us out of the class of finitely generated groups. For example, the injective hull of the simplest group $\Zo$ is $\Zo\hookrightarrow \Qo$ so we need to take the groups like $\Qo$, $\Qo/\Zo$, their direct sums, and homomorphisms between them into account. This certainly complicates the computational engine.

The situation is a bit easier in the category $\ko\FinVect$ of finite-dimensional vector spaces over a field, where every object is injective. The computational approach to cohomology, in which injective resolution of a given sheaf is built step by step, looks a bit scary. However, we notice here that our homology-slaying Algorithm~\ref{algMainAlg} to be described in the latter sections, can be reformulated as the construction of a minimal injective resolution which is, in a sense, universal for all sheaves $D$ on a given poset $S$. 
\end{rem}

\subsection{Roos complex}\label{subsecMathCohomologySimplicial}

The definition of sheaf cohomology given in Subsection~\ref{subsecMathCohomologyDerived} will be used mainly for mathematical reference: it is universal, but seems impractical when it comes to actual calculations as shown in Remark~\ref{remComputFeasibilityOfAGcohomology}. There seems to be a consensus, both in theoretical and applied community, that when it comes to (co)homology computations, it is better to deal directly with (co)chain complexes, see Subsection~\ref{subsecMathCochainRecap}.


Given a diagram $D$ on a poset $S$ (equiv. a sheaf $\ca{D}$ on a finite topological space $X_S$) we want to define a cochain complex $(C^*(S;D),d)$ which is more-or-less easily constructed from $D$, and computes cohomology of $D$, i.e.
\[
H^j(C^*(S;D),d)\cong H^j_{AG}(X_S;\ca{D}).
\]
For any finite poset $S$ there is a classical construction introduced in the paper of Roos~\cite{roos1961derlim}.

\begin{con}\label{conLargestEltOfChain}
Let $\tau$ be a chain $s_0<s_1<\cdots<s_j$ in the poset $S$, which means that $\tau$ is a $j$-dimensional simplex of the order complex $\ord(S)$. Denote by $\tau_{\max}$ the largest element $s_k$ of this chain. If $\tau\subseteq\sigma$ is a set-theoretic inclusion of chains, then clearly $\tau_{\max}\leq \sigma_{\max}$.
\end{con}

\begin{con}\label{conStandardSimplicialIncNumbers}
Assume that $\sigma$ is a $(j+1)$-dimensional simplex of $\ord(S)$, i.e. a chain $s_0<s_1<\cdots<s_{j+1}$, and $\tau$ is a $j$-dimensional simplex. Then we can define the incidence number $\inc{\sigma}{\tau}\in\Zo$ in the standard alternating fashion. If $\sigma$ is a chain $s_0<s_1<\cdots<s_{j+1}$ and $\tau=\sigma\setminus\{s_i\}$, then $\inc{\sigma}{\tau}$ is set equal to $(-1)^i$; in all other cases $\inc{\sigma}{\tau}$ is set equal to $0$. This alternating assumption easily implies that, for any $\sigma_1<\sigma_2$, $\dim\sigma_1=j-1$, $\dim\sigma_2=j+1$, we have
\begin{equation}\label{eqDiamondSimplicial}
\sum_{\tau\colon \dim\tau=j}\inc{\sigma_2}{\tau}\cdot\inc{\tau}{\sigma_1}=0.
\end{equation}
There are only two nonzero summands in this sum.
\end{con}

\begin{defin}\label{definRoosComplex}
Let $D$ be a diagram on a poset $S$. Consider the graded vector space $C^*_{Roos}(S;D)$, and a map $d_{Roos})$ where
\[
C^j_{Roos}(S;D)=\bigoplus_{\tau\in\ord(S),\dim\tau=j}D(\tau_{\max}),\mbox{ and }d_{Roos}\colon C^j_{Roos}(S;D)\to C^{j+1}_{Roos}(S;D)
\]
defined by
\[
d=\bigoplus_{\tau<\sigma, \dim\tau=j,\dim\sigma=j+1}\inc{\sigma}{\tau}F(\tau_{\max}\leq \sigma_{\max}).
\]
It follows from~\eqref{eqDiamondSimplicial} that $d^2=0$. The cochain complex $(C^j_{Roos}(S;D),d_{Roos})$ is called \emph{the Roos complex} of the diagram $D$ on $S$, and its cohomology are called \emph{the Roos cohomology}:
\[
H^j_{Roos}(S;D)=H^j(C^j_{Roos}(S;D),d_{Roos}).
\]
\end{defin}

\begin{rem}\label{remRefinedDiagram}
Given a diagram $D$ on $S$, we can naturally define a refined commutative diagram $\hat{D}\colon\cat(\ord(S))\to \Vv$, given by $\hat{D}(\sigma)=D(\sigma_{\max})$ and $\hat{D}(\sigma\subset\tau)=D(\sigma_{\max}\leq\tau_{\max})$. Roos complex is basically the cellular cochain complex associated with the cellular sheaf (local coefficient system) $\hat{D}$ on $\ord(S)$. However, since cellular sheaves and cellular cochain complex haven't been defined yet, we avoid putting this remark as the definition of Roos complex to avoid any accusation in looping the definitions.
\end{rem}

\begin{con}\label{conConstantSheafOnLowerIdeal}
Let $T\subseteq S$ be a lower order ideal of $S$, i.e. a set with the property that $s\in T$ and $s'<s$ implies $s'\in T$. Equivalently, $T$ is a closed subset of the Alexandrov topological space $X_S$. Let $V\in\Vv$ be an object of the abelian category. Consider the diagram $V_T\in \Diag(S;\Vv)$ defined by
\[
V_T(s)=\begin{cases}
         V, & \mbox{if } s\in T \\
         0, & \mbox{otherwise}
       \end{cases}\mbox{ and }
       V_T(s'\leq s) =\begin{cases}
         \id_V, & \mbox{if } s\in T \\
         0, & \mbox{otherwise}.
       \end{cases}
\]
We denote the corresponding sheaf on $X_S$ with the same letter $V_T$.
\end{con}

The following statement shows the topological nature of the Roos construction. 

\begin{prop}\label{propRoosIsSingular}
If $T\subseteq S$ is a lower order ideal, then Roos cohomology $H^*_{Roos}(X_S;V_T)$ is naturally isomorphic to the singular cohomology $H^*(|T|;V)$ of the geometrical realization.
\end{prop}

\begin{proof}
The Roos complex $(C^j_{Roos}(S;V_T),d_{Roos})$ by construction coincides with the simplicial cochain complex of the simplicial complex $\ord(T)$ with coefficients in $V$, hence $H^*_{Roos}(S;V_T)\cong H^*_{simp}(|T|;V)$. The latter modules are isomorphic to simplicial, and hence singular cohomology of the geometrical realization, which is a standard fact in topology~\cite[Thm.2.27]{Hatcher}. 
\end{proof}

\begin{ex}\label{exInjectiveIsCone}
The basic injective sheaf $\low{s}{W}$ defined in Construction~\ref{conInjSheaves} coincides with $W_{S_{\leq s}}$. We have $H^*_{Roos}(X_S;\low{s}{W})\cong H^*(|S_{\leq s}|;W)$. The latter cohomology is trivial in positive degrees and isomorphic to $W$ in degree $0$, as follows from contractibility of $|S_{\leq s}|$.
\end{ex}

A classical theorem asserts that Roos cohomology of a sheaf are isomorphic to cohomology defined from the injective resolution.

\begin{thm}\label{thmRoosIsAGcohomology}
For any sheaf $D$ on a poset $S$, the Roos cohomology modules are naturally isomorphic to its cohomology:
\[
H^j_{Roos}(S;D)=H^j_{AG}(X_S;\ca{D}).
\]
\end{thm}

Basically, this statement is the instance of the universal delta-functor theorem. We postpone the proof to Subsection~\ref{subsecMathHonestComputations}, where it is formulated in bigger generality. The next claim follows from Theorem~\ref{thmRoosIsAGcohomology} and Proposition~\ref{propRoosIsSingular}.

\begin{rem}\label{remRoos0globalSec}
Theorem~\ref{thmRoosIsAGcohomology} implies $H^0_{Roos}(S;D)\cong \Gamma(S;D)$. This particular statement can be seen directly: it is a simple exercise to prove that, for the Roos differential $d_{Roos}\colon C^0_{Roos}(S;D)\to C^1_{Roos}(S;D)$, the equation $d_{Roos}x=0$ coincides with coherence equations~\ref{eqCoherentStates}.
\end{rem}

A number of important corollaries follow from Theorem~\ref{thmRoosIsAGcohomology}.

\begin{cor}\label{corLowerSheafIsSingularCohomology}
If $T\subseteq S$ is a lower order ideal, then the sheaf cohomology $H^*_{AG}(X_S;V_T)$ is naturally isomorphic to the singular cohomology $H^*(|T|;V)$ of the geometrical realization.
\end{cor}

\begin{ex}\label{exConstantSheaf}
When $T=S$, the sheaf $V_S$ is just the constant sheaf on $S$ (see Construction~\ref{conConstantSheaf}), we denoted it by $\bar{V}$. From the Corollary~\ref{corLowerSheafIsSingularCohomology}, it follows that $H^*_{AG}(X_S;\bar{V})\cong H^*(|X|;V)$; it is nice to have all notions consistent throughout mathematics. In particular, it follows that $\Gamma(X_S;\overline{\ko^1})\cong H^0(|S|;\ko)\cong \ko^{c(|S|)}$, where $c(|S|)$ is the number of connected components of the geometrical realization $|S|$, the latter property is well known for singular (co)homology.
\end{ex}

Another important Corollary of Theorem~\ref{thmRoosIsAGcohomology} states that, for a finite poset, cohomology in sufficiently large degrees vanish.

\begin{cor}\label{corVanishing}
Let $h=\dim\ord S$ be the length of the longest chain in $S$. Then, for any sheaf $\ca{D}$ on $X_S$ we have $H^j_{AG}(X_S;\ca{D})=0$ for $j>h$.
\end{cor}

\begin{proof}
By definition, the graded components $C^j_{Roos}(S;D)$ of the Roos complex vanish for $j>h$: there are no chains of length $>h$ in $S$. Therefore cohomology $H^j_{Roos}(S;D)$ vanish as well, and the rest follows from Theorem~\ref{thmRoosIsAGcohomology}.
\end{proof}

We finish this section with a general remark.

\begin{rem}
The alternative way to reduce cohomology computation to a simplicial structure is provided by the notion of \emph{\v{C}ech cohomology}. In this approach, an arbitrary topological space is replaced with the nerve of its covering, and the computation is performed on this simplicial complex. We do not give a review of the related theory here. However, it should be mentioned that \v{C}ech cohomology of sheaves on finite posets are not necessarily isomorphic to the cohomology $H^*_{AG}$ as observed by Husainov~\cite[Sect.5.4]{Husainov}. The analogue of Theorem~\ref{thmRoosIsAGcohomology} does not hold, therefore one should be careful about which type of cohomology is used.
\end{rem}

\subsection{Algorithmic issues}\label{subsecMathFeasibility}

\begin{defin}\label{definHomologicallyFeasible}
Assume that an abelian category $\Vv$ has finite direct sums. We say that $\Vv$ is \emph{homologically feasible} if, given a triple $V_{-1}\stackrel{d_0}{\rightarrow}V_0\stackrel{d_1}{\rightarrow} V_1$ in $\Vv$ satisfying $d_0;d_1=0$, the problem of describing subquotients $H=\Ker d_1/\im d_{0}$ is algorithmically solvable.
\end{defin}

\begin{ex}\label{exVectFeasible}
The category $\ko\FinVect$ of finite dimensional vector spaces is homologically feasible --- as long as arithmetical properties of the ground field $\ko$ are described somehow. This follows from the Gauss algorithm.
\end{ex}

\begin{ex}\label{exAbelFeasible}
The category $\FinAbel$ of finitely generated abelian groups is homologically feasible as well since the classification of finitely generated abelian groups is constructive.
\end{ex}

\begin{ex}\label{exPolyModulesFeasible}
The category $\ko[t]\FinMod$ of finitely generated modules over a polynomial ring in a single variable is homologically feasible by the same reason. The ring $\ko[t]$ is a principal ideal domain, and the classification of finitely generated modules over a PiD is constructive. The latter example is of particular importance in topological data analysis, since persistence modules, in particular persistent homology, are the objects of the category of (graded) $\ko[t]$-modules.
\end{ex}

The fact that cohomology of Roos complex coincides with cohomology defined in derived manner, has an important corollary.

\begin{cor}\label{corComputabilityPosets}
Assume $X$ is a finite topology, and $\Vv$ is a homologically feasible category. Then, for any sheaf $\ca{D}$ on $X$ valued in $\Vv$, all the cohomology $H^*_{AG}(X;\ca{D})$ are computable.
\end{cor}

\begin{proof}
Proposition~\ref{propPosTop} tells that $X=X_S$ for some preposet $S$, and according to Remark~\ref{remPreposetsNotNeeded}, without loss of generality, $S$ can be assumed a poset. For a poset $S$, Theorem~\ref{thmRoosIsAGcohomology} claims that $H^*_{AG}(X_S;\ca{D})\cong H^*_{Roos}(S;D)$. Roos cohomology are computed as subquotients of the Roos complex $C^*_{Roos}(S;D)$. Since there are only finitely many chains in a finite poset $S$, the whole complex $C^*_{Roos}(S;D)$ is a finite direct sum of stalks. For sufficiently large $j$ (for example for $j>\#S$), the component $C^j_{Roos}(S;D)$ simply vanishes. The assumption of algorithmic feasibility implies that cohomology are computable in all degrees $j$ below $\#S$.
\end{proof}

\subsection{Cohomology of cellular sheaves}\label{subsecMathCohomologyCellular}

Cohomology computation with the use of Roos complex is universal --- it works for all sheaves on Alexandrov spaces. In the case of a finite poset and a diagram valued in finite-dimensional vector spaces, the computation of cohomology can be done algorithmically, as Corollary~\ref{corComputabilityPosets} shows. However, such computation is heavy in practice: the order complex $\ord S$ of a general poset $S$ has many simplices, so the matrices involved in the calculation of Roos cohomology are exponential in the cardinality of $S$, in the worst case.

Another approach to calculate cohomology of a sheaf is more optimal but it narrows the class of posets on which it applies. In this section we review the construction and basic theory of cellular cohomology which are well defined on cell posets, or their analogues, given by Definitions~\ref{definCellPoset}, \ref{definHomotopyCellPoset}, and \ref{definHomologyCellPoset}.
The mathematical results of this subsection were proved in~\cite{EVERITT2015134}, and the definition of a cellular cochain complex appears in all applied papers dealing with cohomology of cellular sheaves. We provide formulations in terms introduced previously in the paper.

\begin{defin}\label{definCellularSheaf}
A \emph{cellular sheaf} $D$ is a diagram on a cell poset $\ca{X}$. Equivalently, a cellular sheaf is a sheaf on the Alexandrov space corresponding to a cell poset.
\end{defin}

The property of ``being cellular'' is actually not the property of a sheaf itself, but rather the property of its domain of definition.

\begin{rem}
The previous sections, especially Proposition~\ref{propDiagSheaf} should have convinced the reader that there is not much operational difference between posets $S$ and Alexandrov topologies $X_S$ when we speak of diagrams/sheaves. To avoid the notational monster $X_{\ca{X}}$, we abuse the notation and denote by $\ca{X}$ both cell posets (or the variations of this notion) and the corresponding Alexandrov topologies.
\end{rem}

In order to define a cochain complex, which computes sheaf cohomology of cellular sheaves in a more optimal way than Roos complex, one needs an important ingredient: incidence numbers~\footnote{At this place the perfect symphony of abstract nonsense is first broken by a down-to-earth notion of a number. The harmony will be restored later on in the text, where we show that incidence numbers are homology as well.} of cells. This notion is classical in the study of CW-complexes in algebraic topology. 

Let $\sigma$ be a cell of rank $k$ of a cell poset $\ca{X}$. By definition, it means that the space $|\dd C(\sigma)|=|\ca{X}_{<\sigma}|$ has integral homology of a $(k-1)$-dimensional sphere. From the contractibility of the cone $|C(\sigma)|$ and the long exact sequence of the pair $(|C(\sigma)|,|\dd C(\sigma)|)$, it follows that
\[
H^k(|C(\sigma)|,|\dd C(\sigma)|;\Zo)\cong \Zo
\]
while all other relative homology groups vanish: $H^j(|C(\sigma)|,|\dd C(\sigma)|;\Zo)=0$ for $j\neq \rk\sigma$.

\begin{defin}\label{definCellOrientation}
A choice of a generator $\kappa_\sigma$ of the group $H^k(|C(\sigma)|,|\dd C(\sigma)|;\Zo)\cong \Zo$ is called \emph{an orientation} of a cell $\sigma\in\ca{X}$.
\end{defin}

\begin{con}\label{conIncidenceNumbers}
Assume that $\sigma,\tau\in \ca{X}$ be two cells of a cell poset $\ca{X}$ such that $\tau<\sigma$ and $\rk\tau=k$, $\rk\sigma=k+1$ for some $k=0,1,\ldots$. Consider the map $f_{\sigma,\tau}$ given by composing the sequence
\begin{equation}\label{eqMapForIncidenceNumber}
H_{k+1}(|C(\sigma)|,|\dd C(\sigma)|;\Zo)\stackrel{\cong}{\longrightarrow} \Hr_{k}(|\dd C(\sigma)|;\Zo)\to H_{k}(|\dd C(\sigma)|,|\dd C(\sigma)\setminus\{\tau\}|;\Zo)\cong H_{k}(|C(\tau)|,|\dd C(\tau)|;\Zo),
\end{equation}
where
\begin{enumerate}
  \item The first map $f_1\colon H_{k+1}(|C(\sigma)|,|\dd C(\sigma)|;\Zo)\to H_{k}(|\dd C(\sigma)|;\Zo)$ is the connecting homomorphism in the long exact sequence of singular homology of the pair $(|C(\sigma)|,|\dd C(\sigma)|)$. This map is an isomorphism since the cone $|C(\sigma)|$ is contractible.
  \item The second map $H_{k}(|\dd C(\sigma)|;\Zo)\to H_{k}(|\dd C(\sigma)|,|\dd C(\sigma)\setminus\{\tau\}|;\Zo)$ is the natural map to relative homology.
  \item The last isomorphism is due to excision property of singular homology.
\end{enumerate}
Since $\ca{X}$ is a cell complex, both groups $H_{k+1}(|C(\sigma)|,|\dd C(\sigma)|;\Zo)$ and $H_{k}(|C(\tau)|,|\dd C(\tau)|;\Zo)$ are isomorphic to $\Zo$, so the composition of~\eqref{eqMapForIncidenceNumber} belongs to $\Hom(\Zo,\Zo)\cong\Zo$. If the orientations $\kappa_\sigma$ and $\kappa_\tau$ of cells $\sigma$ and $\tau$ are fixed, the map $f_{\sigma,\tau}$ becomes a number as well.
\end{con}

\begin{defin}\label{definIncidenceNumber}
Let $\sigma,\tau\in \ca{X}$ be two cells of a cell poset $\ca{X}$ such that $\rk\sigma=\rk\tau+1$, and let $\kappa_\sigma$ and $\kappa_\tau$ be orientations of these cells. \emph{The incidence number} $\inc{\sigma}{\tau}\in\Zo$ is defined as the unique integer such that
\[
f_{\sigma,\tau}(\kappa_\sigma)=\inc{\sigma}{\tau}\kappa_\tau,
\]
--- in the case $\tau<\sigma$. If $\sigma\ngtr \tau$, we set $\inc{\sigma}{\tau}=0$.
\end{defin}

The following statement is the classical property of incidence numbers, which allows to use them in the definition of a differential in a cellular cochain complex.

\begin{lem}[Diamond property]\label{lemDiamondProperty}
Let $\sigma,\tau$ be two cells of a cell poset $\ca{X}$ such that $\sigma_1<\sigma_2$ and $\rk\sigma_2=\rk\sigma_1+2$. Then
\begin{equation}\label{eqDiamond}
\sum_{\tau\colon \sigma_1<\tau<\sigma_2}\inc{\sigma_2}{\tau}\cdot\inc{\tau}{\sigma_1}=0.
\end{equation}
\end{lem}

\begin{proof}
The classical sketch of proof. Let $\rk\sigma_2=k+1$, so that $\rk\sigma_1=k-1$, and consequently $\rk\tau=k$. Since $\inc{\sigma}{\tau}=0$ when $\tau$ is incomparable with $\sigma$, we can replace the required formula~\eqref{eqDiamond} by the equivalent formula
\begin{equation}\label{eqDiamondRewritten}
\sum_{\tau\colon \rk\tau=k}\inc{\sigma_2}{\tau}\cdot\inc{\tau}{\sigma_1}=0.
\end{equation}
The direct sum $\bigoplus_{\sigma\colon \rk\sigma=k+1}H_{k+1}(|C(\sigma)|,|\dd C(\sigma)|;\Zo)$ is naturally identified with $H_{k+1}(|\ca{X}_{k+1}|,|\ca{X}_k|;\Zo)$ (recall the definition of a skeleton in Construction~\ref{conCellPosetSkeleton}). The direct sum of all maps $f_{\sigma,\tau}$ over all $\sigma$ and $\tau$ of ranks $k+1$ and $k$ respectively becomes identified with the map $\delta_{k+1}\colon H_{k+1}(|\ca{X}_{k+1}|,|\ca{X}_k|;\Zo)\to H_{k}(|\ca{X}_{k}|,|\ca{X}_{k-1}|;\Zo)$, which is the connecting homomorphism in the long exact sequence of the triple $(|\ca{X}_{k+1}|,|\ca{X}_{k}|,|\ca{X}_{k-1}|)$. The choice of orientations of the cells defines the basis of both domain and target abelian group, therefore the map $\delta_{k+1}$ becomes an integral matrix $D_{k+1}$ with entries $\inc{\sigma}{\tau}$.

Equality~\eqref{eqDiamondRewritten} is equivalent to $D_kD_{k+1}=0$. This directly follows from the fact that two connecting homomorphisms $\delta_{k+1}\colon H_{k+1}(|\ca{X}_{k+1}|,|\ca{X}_k|;\Zo)\to H_{k}(|\ca{X}_{k}|,|\ca{X}_{k-1}|;\Zo)$ and $\delta_{k}\colon H_{k}(|\ca{X}_{k}|,|\ca{X}_{k-1}|;\Zo)\to H_{k-1}(|\ca{X}_{k-1}|,|\ca{X}_{k-2}|;\Zo)$ compose to zero. Indeed, these connecting homomorphisms are nothing but the 1-st differential $d_1$ of the spectral sequence in homology, associated with the filtration
\[
\varnothing\subset |\ca{X}_0|\subset |\ca{X}_1|\subset|\ca{X}_2|\subset\cdots,
\]
therefore $d_1^2=0$ is trivially satisfied.
\end{proof}

Less involved proofs of the stated fact may be found in any basic book on algebraic topology, when the notion of cell cohomology is introduced, see e.g.~\cite[p.139]{Hatcher}.

\begin{defin}\label{definCWcochainComplex}
For a cellular sheaf $D$ on a cell poset $\ca{X}$, consider the cochain complex $(C^*_{CW}(\ca{X};D),d_{CW})$
\[
C^j_{CW}(\ca{X};D)=\bigoplus_{\sigma\in \ca{X}, \rk\sigma=j}D(s), \qquad d_{CW}\colon C^j_{CW}(\ca{X};D)\to C^{j+1}_{CW}(\ca{X};D),
\]
where $d_{CW}=\bigoplus_{\tau<\sigma, \rk\tau=j, \rk\sigma=j+1}\inc{\sigma}{\tau}F(\tau<\sigma)$. This cochain complex is called \emph{the cellular cochain complex} of the cellular sheaf $D$. The cohomology of $(C^*_{CW}(\ca{X};D),d_{CW})$ are called \emph{cellular cohomology} of a cellular sheaf $D$ and are denoted
\[
H^j_{CW}(\ca{X};D)=H^j(C^*_{CW}(\ca{X};D),d_{CW}).
\]
\end{defin}

\begin{rem}
The most nontrivial part of this definition is the property $d_{CW}^2=0$. It is easily proven from the diamond property of incidence numbers and commutativity of a diagram $D$ itself. Indeed, if we consider the block of the map $d_{CW}^2$ acting from $D(\sigma_1)\subset C^{k-1}(\ca{X};D)$ to $D(\sigma_2)\subset C^{k+1}(\ca{X};D)$ (where $\sigma_1<\sigma_2$), this block equals $D(\sigma_1<\sigma_2)$ multiplied by $\sum_{\tau\colon \sigma_1<\tau<\sigma_2}\inc{\sigma_2}{\tau}\cdot\inc{\tau}{\sigma_1}$ since we summate the compositions of maps over all saturated paths leading from $\sigma_1$ to $\sigma_2$. The latter sum vanishes due to the diamond property, Lemma~\ref{lemDiamondProperty}.
\end{rem}

\begin{thm}\label{thmCWcohomologyIsAGcohomology}
For any cellular sheaf $D$ on a cell poset $\ca{X}$, the cellular cohomology modules are naturally isomorphic to its cohomology:
\[
H^j_{CW}(\ca{X};D)=H^j_{AG}(\ca{X};D).
\]
\end{thm}

We prove both Theorems~\ref{thmRoosIsAGcohomology} and~\ref{thmCWcohomologyIsAGcohomology} with the same general argument in subsection~\ref{subsecMathHonestComputations}. The next remark explains, why Theorem~\ref{thmCWcohomologyIsAGcohomology} is important even if we don't care about higher order cohomology.

\begin{rem}\label{remCompatibilityReducesToCells}
As stated in Proposition~\ref{propCohomologyProperties}, we have $H^0_{AG}(\ca{X};D)\cong \Gamma(\ca{X};D)$, the space of global sections. In the original formulation, the set $\Gamma(\ca{X};D)$ is defined as the solution to the system of compatibility equations, see Definition~\ref{definGlobalSectionsConcrete} or Remark~\ref{remConcreteLimSets}. However, Theorem~\ref{thmCWcohomologyIsAGcohomology} implies that
\[
H^0_{AG}(\ca{X};D)\cong H^0_{CW}(\ca{X};D)=\Ker d_{CW}\colon C^0_{CW}(\ca{X};D)\to C^1_{CW}(\ca{X};D).
\]
The number of equations in the system $d_{CW}x=0$ is less than the number of all compatibility equations (even restricted to the edges of the Hasse diagram of a poset, as in Remark~\ref{remReduceEquationsToGenerators}), because it only involves relations between vertices and edges, with no higher-dimensional cells involved. Informally this means that, in a cell complex, higher-dimensional cells do not store any compatibility information --- it all lives in the 1-skeleton $\ca{X}_1$. This phenomenon reflects the cellular approximation theorem~\cite[p.349]{Hatcher}, known in algebraic topology: any ``proof of coherence'' is a path through a space, and this path can be deformed to a path lying in the 1-skeleton. The same remark holds true for cellular complexes even if a sheaf takes values in non-abelian category, such as $\Sets$: we don't need all compatibility equations~\eqref{eqCoherentStates} to describe global sections, --- equations between 0- and 1-dimensional cells do the job.
\end{rem}

\subsection{Honest cohomology computations}\label{subsecMathHonestComputations}

Let us fix a general poset $S$ and an abelian category $\Vv$. Consider an additive functor $\Fu\colon \Shvs(X_S;\Vv)\to\Cochain(\Vv)$ endowed with a natural transformation $\Gamma(X_S;\cdot)\to \Cochain(\cdot)^0$ which will be called \emph{an augmentation}. The latter means that, for any sheaf $\ca{D}$ on $X_S$ we have natural maps from the global sections $\ca{D}(S)=\lim\limits_{\leftarrow S}D$ to the 0-degree component $\Fu(\ca{D})^0$ of the cochain complex $\Fu(\ca{D})$.

\begin{defin}\label{definHonestlyComputes}
We say that $\Fu$ \emph{honestly computes sheaf cohomology}, if $H^*(\Fu(\ca{D}))$ is naturally isomorphic to $H_{AG}^*(X_S;\ca{D})$.
\end{defin}

Recall that $\low{s}{W}$ is a basic injective sheaf on $S$ defined in Construction~\ref{conInjSheaves}. The next theorem is a reformulation of Grothendieck's delta-functor theorem~\cite[Thm.2.2.2]{grothendieck1957tohoku} which we find more convenient to use in practice.

\begin{thm}\label{thmHonestCohomology}
Assume that the functor $\Fu\colon \Shvs(X_S;\Vv)\to\Cochain(\Vv)$ is exact and for every injective object $W\in\Vv$ we have $H^j(\Fu(\low{s}{W}))=0$ if $j\neq 0$ and the augmentation functor induces the natural isomorphism $W=\Gamma(S;\low{s}{W})\to H^0(\Fu(\low{s}{W}))$. Then $\Fu$ honestly computes sheaf cohomology.
\end{thm}

\begin{proof}
Let $\ca{D}$ be a sheaf on $X_S$ ($D$ a diagram on $S$). Consider an injective resolution~\eqref{eqInjSheaf} of $\ca{D}$:
\[
0\to \ca{D}\to \ca{I}_0\to \ca{I}_1\to\cdots
\]
where each $\ca{I}_j$ is a direct sum of injectives of the form $\low{s}{W}$. Applying the functor $\Fu$ to this injective resolution we get a bicomplex
\begin{equation}\label{eqBicomplex}
\begin{CD}
@.   @. 0 @. 0 @.\\
@. @. @VVV @VVV @.\\
@. @. \ca{I}_0(S) @>>> \ca{I}_1(S) @>>> \cdots\\
@. @. @VVV @VVV @.\\
0 @>>> \Fu(\ca{D})^0 @>>> \Fu(\ca{I}_0)^0 @>>> \Fu(\ca{I}_1)^0 @>>> \cdots\\
@. @VVV @VVV @VVV @.\\
0 @>>> \Fu(\ca{D})^1 @>>> \Fu(\ca{I}_0)^1 @>>> \Fu(\ca{I}_1)^1 @>>> \cdots\\
@. @VVV @VVV @VVV @.\\
@.\vdots @.\vdots @.\vdots @.\ddots
\end{CD}
\end{equation}
Here we added the augmentation map on top of each column except the leftmost one.

In the first nontrivial row of~\eqref{eqBicomplex}, the global sections of the injective resolution of $\ca{D}$ are written. The cohomology of this complex is, by definition, the sheaf cohomology $H_{AG}^*(X_S;\ca{D})$. All other rows are acyclic, since we assumed the functor $\Fu$ is exact.

In the first nontrivial column of~\eqref{eqBicomplex} the cochain complex $\Fu(\ca{D})$ itself is written. All other columns are acyclic, because the augmented complex $0\to \ca{I}(S)\to \Fu(\ca{I})^*$ is assumed acyclic for any injective of the form $\low{s}{W}$, --- hence for any direct sum of such sheaves.

Now, the standard argument of homological algebra (see e.g.\cite[Theorem 2.15]{SpectralSeq}) shows that cohomology of the first row is naturally isomorphic to the cohomology of the first column: $H^*_{AG}(X_S;\ca{D})\cong H^*(\Fu(\ca{D}))$, which was required.
\end{proof}

\begin{rem}\label{remInjectiveAcyclic}
Notice that the statement somehow converse to Theorem~\ref{thmHonestCohomology} holds true: if $\Fu$ honestly computes cohomology, then for any injective sheaf $\low{s}{W}$, the cochain complex $\Fu(\low{s}{W})$ is acyclic in degrees $j>0$, while $H^0(\Fu(\low{s}{W}))\cong W$. Indeed, each injective sheaf is acyclic, see Remark~\ref{remOpinionLeader}, so this statement is trivial.
\end{rem}

The following are the two main classes of examples of honest cohomology computation.

\begin{ex}\label{exRoos}
Let $D$ be a diagram on a poset $S$. Roos complex $C^*_{Roos}(S;D)$ is defined, see Definition~\ref{definRoosComplex}. This construction defines a functor $\Fu_{Roos}\colon \Diag(S;\Vv)\to \Cochain(\Vv)$ --- the functoriality is easily seen from the definition. It is an exact functor, because each module $C^j_{Roos}(S;D)$ is a direct sum of stalks $D(s)$ of a sheaf, and taking a stalk is an exact functor according to Proposition~\ref{propExactSequenceDiagrams}. The augmentation map $\ca{D}(S)\to C^0_{Roos}(S;D)$ is naturally defined since $C^0_{Roos}(S;D)$ is a direct sum of stalks, and $\ca{D}(S)$ in the inverse limit $\lim\limits_{\leftarrow S}D$ of a diagram, hence maps naturally to all its stalks.

In order to apply Theorem~\ref{thmHonestCohomology}, we need to check that, for any injective $W\in\Vv$, the augmented complex $0\to\low{s}{W}(X_S)\to C^*_{Roos}(S;\low{s}{W})$ is acyclic. This is indeed the case, see Example~\ref{exInjectiveIsCone}. The assumptions of Theorem~\ref{thmHonestCohomology} are fulfilled for the Roos functor $\Fu_{Roos}$, hence $H^*_{AG}(X_S;\ca{D})\cong H_{Roos}^*(S;D)$. This proves Theorem~\ref{thmRoosIsAGcohomology}.
\end{ex}

\begin{ex}\label{exCell}
Let $D$ be a cellular sheaf on a cell poset $\ca{X}$. Cellular cochain complex $C^*_{CW}(\ca{X};D)$ is defined, see~\ref{definCWcochainComplex}. This construction defines a functor $\Fu_{CW}\colon \Diag(\ca{X};\Vv)\to \Cochain(\Vv)$ --- the functoriality is easily seen from the definition. It is an exact functor, because each module $C^j_{CW}(\ca{X};D)$ is a direct sum of stalks $D(s)$ of a sheaf, and taking a stalk is an exact functor. The augmentation map $\ca{D}(\ca{X})\to C^0_{CW}(\ca{X};D)$ is naturally defined since $C^0_{CW}(\ca{X};D)$ is a direct sum of stalks, and $\ca{D}(\ca{X})$ in the inverse limit $\lim\limits_{\leftarrow \ca{X}}D$ of a diagram, hence maps to all its stalks.

Let us check that, for any injective $W\in\Vv$, the augmented complex $0\to\low{\sigma}{W}(\ca{X})\to C^*_{CW}(\ca{X};\low{\sigma}{W})$ is acyclic. The augmented cellular cochain complex takes the form
\begin{equation}\label{eqCellCochainOfCell}
0\to W\to \bigoplus_{\tau\leq\sigma, \rk\tau=0}W\to \bigoplus_{\tau\leq\sigma, \rk\tau=1} W\to \bigoplus_{\tau\leq\sigma, \rk\tau=2} W\to \cdots
\end{equation}
Recall that precise definition of the cellular differential required the auxiliary notion of incidence numbers, which came from topological considerations. It is time to dig these considerations up. Notice that the differential complex~\eqref{eqCellCochainOfCell} is by construction the result of application of a functor $\Hom(\cdot,W)$ to the augmented chain complex
\begin{equation}\label{eqChainComplexOfCell}
0\leftarrow W\leftarrow C_0^{CW}(|C(\sigma)|;\Zo)\leftarrow C_1^{CW}(|C(\sigma)|;\Zo)\leftarrow C_2^{CW}(|C(\sigma)|;\Zo)\leftarrow \cdots
\end{equation}
where each $C_j^{CW}(|C(\sigma)|;\Zo)$ is by definition the relative homology group $C_j(|C(\sigma)_j|,|C(\sigma)_{j-1}|;\Zo)$. The homological spectral sequence of the skeletal filtration $\varnothing\subset |C(\sigma)_0|\subset|C(\sigma)_1|\subset\cdots$ degenerates at $E^2$-page by the vanishing argument. Therefore the homology of the complex~\eqref{eqChainComplexOfCell} coincides with the reduced (singular) homology of the space $|C(\sigma)|$ itself. Since $|C(\sigma)|$ is contractible, its reduced homology vanish. Therefore, the cohomology of the cochain complex~\eqref{eqCellCochainOfCell} vanish as well by the universal coefficients' theorem for cohomology.

The assumptions of Theorem~\ref{thmHonestCohomology} are fulfilled for the CW-cochain functor $\Fu_{CW}$. Therefore $H^*_{AG}(\ca{X};D)\cong H_{CW}^*(\ca{X};D)$ which proves Theorem~\ref{thmCWcohomologyIsAGcohomology}.
\end{ex}

Next example is very particular, but it illustrates that the general technique which proves Theorem~\ref{thmHonestCohomology} works pretty well on more concrete examples, which are not covered by Examples~\ref{exRoos} and~\ref{exCell}.

\begin{ex}\label{exMorse}
Let $\ca{Y}$ be a poset described in Example~\ref{exMorseExample}. It is a Morse cell poset with 3 vertices $1,2,3$ and two edges $a,b$, with the edge $a$ attached to vertices $1$ and $2$, and the edge $b$ attached to the middle of $a$ and the vertex $3$. For every sheaf $D$ on $\ca{Y}$ consider the augmented cochain complex
\[
0\to \lim\limits_{\leftarrow \ca{Y}} D\to C^0(\ca{Y};D)\to C^1(\ca{Y};D)\to 0
\]
where
\[
C^0(\ca{Y};D)=D(1)\oplus D(2)\oplus D(3), \qquad C^1(\ca{Y};D)=D(a)\oplus D(b),
\]
the augmentation is a natural map from the inverse limit, and the differential $d\colon C^0(\ca{Y};D)\to C^1(\ca{Y};D)$ is defined by the block matrix
\begin{equation}\label{eqMatrixMorseExample}
\begin{pmatrix}
  D(1<a) & -D(2<a) & 0 \\
  0 & D(2<b) & -D(3<b)
\end{pmatrix}.
\end{equation}
Then the functor $\Fu = C^*(\ca{Y};\cdot)$ honestly computes sheaf cohomology on $\ca{Y}$. Indeed, all the sheaves $\low{s}{W}$ are $\Fu$-acyclic. The most interesting case is $s=b$: in this case acyclicity of $\low{b}{W}$ can be proven by hand.

Informally, the matrix~\eqref{eqMatrixMorseExample} corresponds to first taking ``CW-approximation'' in the spirit of Whitehead theorem (see Fig.~\ref{figElementaryMorse}), and considering incidence numbers of the resulting cellular replacement. Notice however, that on the level of homological algebra, there is no need to take an actual cellular approximation. Instead of matrix~\eqref{eqMatrixMorseExample} one can define the differential to be
\begin{equation}\label{eqMatrixMorseExampleAlternative}
\begin{pmatrix}
  D(1<a) & -D(2<a) & 0 \\
  k_1D(2<b) & k_2D(2<b) & -D(3<b)
\end{pmatrix}
\end{equation}
with $k_1+k_2=1$, and this cochain complex still honestly computes sheaf cohomology on a given poset $\ca{Y}$. This procedure can be considered ``a fuzzy'' CW-approximation.
\end{ex}

The idea behind this example becomes more transparent in the next subsection, where we deal with general Morse posets.

\subsection{One-shot cohomology computations}\label{subsecMathOneShot}

There is a certain similarity shared by Examples~\ref{exRoos}, \ref{exCell}, and~\ref{exMorse}. In all cases the graded components of the complex $\Fu(\ca{D})$, the one which honestly computes cohomology of a sheaf $\ca{D}$, are direct sums of certain stalks $D(s)$ of $\ca{D}$. This motivates the following generalization. 

\begin{defin}\label{definConcreteFunctor}
A functor $\Fu\colon \Shvs(S;\Vv)\to\Cochain(\Vv)$ is called \emph{concrete}, if for any $j\geqslant 0$, the graded component $\Fu(\ca{D})_j$ is a direct sum of stalks $D(s)$ of a sheaf $\ca{D}$.
\end{defin}

Each concrete functor is exact, since taking stalks is exact. Also there exists a natural augmentation map $\ca{D}(S)\to \Fu(\ca{D})_0$, because $\ca{D}(S)=\lim\limits_{\leftarrow} D$ maps to each stalk by the universal property of the inverse limit.

However, it is important that in order for the functor to land in $\Cochain(\Vv)$ one needs to define differentials in a uniform manner. We have already seen in Example~\ref{exCell} that this definition may be tricky: in the case of cell posets ``to define differential in a uniform manner'' reads as ``to construct incidence numbers of cells''. This particular task is outsourced to basic algebraic topology books.

The cellular cochain complex $C^*_{CW}(\ca{X};D)$ is much smaller than Roos complex $C^*_{Roos}(S;D)$, hence more usable in practice. One can summarize this optimality by the following two observations. In the cellular case, each stalk $D(\sigma)$ appears in $C^*_{CW}(\ca{X};D)$ exactly once. In the Roos case, each stalk $D(s)$ appears in $C^*_{Roos}(S;D)$ multiple times. The total number of summands in $C^*_{Roos}(S;D)$ is in the worst case exponential in the cardinality $|S|$. See more precise computations in Examples~\ref{exMultiplicityCW} and~\ref{exMultiplicityRoos} below. This difference motivates the following definition.

\begin{defin}\label{definOneShotComplex}
A concrete functor $\Fu\colon \Shvs(S;\Vv)\to\Cochain(\Vv)$ is said to be \emph{one-shot} if each stalk $D(s)$ of a sheaf $\ca{D}\in\Shvs(S;\Vv)$ appears exactly once as a summand in $\Fu(\ca{D})$.
\end{defin}

Of course we are interested in one-shot concrete functors which provide honest cohomology calculations. We have already seen in Example~\ref{exCell} that cellular cochain complex provides one-shot honest cohomology calculation. On the other hand, Example~\ref{exMorse} shows that non-cell poset may support a functor $\Fu$ that is both one-shot and honestly computes cohomology of sheaves.

\begin{rem}\label{remDeletePoint}
Is it possible to honestly compute cohomology with a concrete functor $\Fu$ without using some stalk $D(s)$ at all? Somehow surprisingly, the answer is yes. Whenever $|S_{<s}|$ is acyclic, the element $s\in S$ can be dropped from the poset, and cohomology (of all sheaves) do not change, namely, $H^j(S;D)\cong H^j(S\setminus\{s\};D|_{S\setminus\{s\}})$. This can be seen as a particular case of the theorem of Oberst which holds in more general categorical setting; see~\cite{Oberst} for derived direct limits of diagrams or~\cite[Thm.3.10]{Husainov} for the cohomological version, or the paper~\cite[Prop.3.5]{SepaFR} where this result is reproved in the context of group homology and cohomology.

In the recent paper of Malko~\cite{Malko} this result is connected to beat-removals which were introduced in algebraic topology in the classical work of Stong~\cite{Stong1966FiniteTS}. The idea of beats' removal recently found an application~\cite{BoissonnatEtAl}: it can be used as a preliminary data simplification step before computing persistent homology, and it works impressively well on the data originating from ML applications. We believe, that a similar preprocessing can be applied for more general sheaf cohomology calculations. For example, if $s$ is a downbeat, i.e. there is a unique edge of the Hasse diagram, which enters $s$, then $|S_{<s}|$ is a cone, hence contractible, hence acyclic, hence $s$ can be deleted without any effect on sheaf cohomology.
\end{rem}

On the other hand, if $|S_{<s}|$ is not acyclic, then it is impossible to not to use $D(s)$ in the computations.

\begin{con}\label{conDiracDiagram}
Consider \emph{the Dirac's delta sheaf} $\delta_s(W)$: the sheaf which has a nonzero value $W$ concentrated in a single point $s$. It will be seen in the subsequent arguments, that whenever $|S_{<s}|$ is not acyclic, the cohomology of $\delta_s(W)$ are also nontrivial, hence cannot be honestly computed if $\delta_s(W)(s)=W$ is not present in the cochain complex.
\end{con}

Next we formulate and prove a new result, which indicates the precise class of posets, suitable for one-shot honest cohomology computations. 

\begin{thm}[One-shot theorem]\label{thmOneShotTheorem}
Let $\Vv$ be the category $\Abel$ of abelian groups or the category $\ko\Vect$ of vector spaces over arbitrary field. The following two conditions on a finite poset $S$ are equivalent:
\begin{enumerate}
  \item There exists a one-shot concrete functor $\Fu\colon\Shvs(S;\Vv)\to\Cochain(\Vv)$, which honestly computes cohomology.
  \item $S$ is a homology Morse cell poset (see Definition~\ref{definMorseHomologyCellPoset}).
\end{enumerate}
\end{thm}

We prove two implications separately. In both cases the arguments work in slightly bigger generality, and may be useful in their own right.

\begin{prop}\label{propOneShotOnePosition}
Assume that a concrete functor $\Fu\colon\Shvs(S;\Vv)\to\Cochain(\Vv)$ honestly computes cohomology of sheaves, and for an element $s\in S$, the stalk $D(s)$ contributes exactly once in $\Fu(\ca{D})$, namely, for some $k\geqslant 0$, the stalk $D(s)$ is a direct summand of the graded component $\Fu(\ca{D})^k$. Then $|S_{<s}|$ has the same homology as a sphere $S^{k-1}$.
\end{prop}

\begin{proof}
For convenience, we restrict considerations to the abelian category $\Vv=\Abel$, the case $\Vv=\ko\Vect$ being similar, and even a bit easier. Let us denote the cochain complex $\Fu(\ca{D})$ by $C^*(S;D)$. Let $W\in\Vv$ be an arbitrary injective abelian group. Consider two sheaves: the first one is $\low{s}{W}=W_{S_{\leq s}}$ defined in Construction~\ref{conInjSheaves}, and another one is
\begin{equation}\label{eqCutLow}
\low{\bar{s}}{W}=W_{S_{<s}}.
\end{equation}
We have a short exact sequence of diagrams on $S$:
\[
0\to \delta_s(W)\to \low{s}{W}\to \low{\bar{s}}{W}\to 0
\]
where $\delta_s(W)$ is the Dirac sheaf, whose stalk in $s$ is $W$ and zero otherwise. The functor $\Fu$ is concrete, hence exact, hence we have a short exact sequence of cochain complexes
\begin{equation}\label{eqShortExactCochainDelta}
0\to C^*(S;\delta_s(W))\to C^*(S;\low{s}{W})\to C^*(S;\low{\bar{s}}{W})\to 0.
\end{equation}
By zig-zag Lemma (Lemma~\ref{lemZigZag}), we have an induced long exact sequence in cohomology
\[
\cdots\to H^j(C^*(S;\low{s}{W}))\to H^j(C^*(S;\low{\bar{s}}{W}))\to  H^{j+1}(C^*(S;\delta_s(W)))\to H^{j+1}(C^*(S;\low{s}{W}))\to\cdots
\]
Since $\Fu$ honestly computes cohomology, the groups $H^j(C^*(S;\low{s}{W}))\cong H_{AG}^j(X_S;\low{s}{W})$ vanish, since $\low{s}{W}$ is injective (it doesn't vanish in degree 0, where the map $H^j(C^*(S;\low{s}{W}))\to H^j(C^*(S;\delta_s(W)))$ is an identity isomorphism $W\to W$, such cases should be considered separately, as usual). From this vanishing, we have an isomorphism~\footnote{It can be considered a sheaf-theoretical analogue of the suspension isomorphism in singular homology.}
\begin{equation}\label{eqConeIsomorphism}
H^{j+1}(C^*(S;\delta_s(W)))\cong H^j(C^*(S;\low{\bar{s}}{W}))\mbox{ for }j>0.
\end{equation}
Now notice that, by assumption, the stalk at $s$ appears only once in the components of $\Fu$ of degree $k$, and $\delta_s(W)$ is concentrated in the point $s$. Therefore $C^k(S;\delta_s(W))=W$, while other components of this cochain complex vanish. Passing to cohomology we get $H^k(S;\delta_s(W))=W$ and all other cohomology groups vanish. Applying isomorphism~\eqref{eqConeIsomorphism} and remembering that $\Fu$ honestly computes sheaf cohomology we get that
\begin{equation}\label{eqSphereLikeVanishing}
H^{k-1}_{AG}(X_S;\low{\bar{s}}{W})\cong W, \mbox{ and } H^{j}_{AG}(X_S;\low{\bar{s}}{W})=0\mbox{ for }j\neq 0,k-1.
\end{equation}
It remains to notice that $\low{\bar{s}}{W}$ is a sheaf supported on subposet $S_{<s}$ which is a lower order ideal. Corollary~\ref{corLowerSheafIsSingularCohomology} implies that $H^*_{AG}(X_S;\low{\bar{s}}{W})$ is isomorphic to singular cohomology of $|S_{<s}|$ with coefficients in $W$. From the vanishing condition~\eqref{eqSphereLikeVanishing} it follows that
\[
\Hr^j(|S_{<s}|;W)=\begin{cases}
                    W, & \mbox{if } j=k-1 \\
                    0, & \mbox{otherwise}
                  \end{cases}
\]
for reduced cohomology groups of the geometrical realization. This shows that the space $|S_{<s}|$ is a ``$W$-cohomological $(k-1)$-sphere'' for each injective group $W$. We need to show that this latter condition is the same as $|S_{<s}|$ being a $\Zo$-homological sphere. This last step is a bit technical, and can be omitted for vector spaces $\ko\Vect$.

Applying the universal coefficient theorem in cohomology~\cite[p.195]{Hatcher} to the space $X=|S_{<s}|$ we get the short exact sequence
\[
0\to\Ext^1_{\Zo}(H_{j-1}(X;\Zo),W)\to \Hr^j(X;W)\to \Hom(\Hr_j(X;\Zo);W)\to 0.
\]
Since $W$ is injective, the $\Ext$-term vanishes, so there is an isomorphism $\Hr^j(X;W)\cong \Hom(\Hr_j(X;\Zo);W)$. For $j\neq k-1$, the module $\Hr^j(X;W)$ vanishes for any $W$, therefore $\Hr_j(X;\Zo)$ vanishes as well (a nontrivial group can always be nontrivially mapped to its injective hull). For $j=k-1$, we have $\Hom(\Hr_j(X;\Zo);W)\cong \Hr^j(X;W)\cong W$ for any injective $W$. In general there exist various abelian groups $G$ satisfying the property
\begin{equation}\label{eqHomProperty}
\Hom(G,W)\cong W \mbox{ for any injective } W,
\end{equation}
not only $G=\Zo$ (for example $G=\Qo$ works as well). Luckily for us, the space $X=|S_{<s}|$ is a finite simplicial complex, so its homology $\Hr_k(X;\Zo)$ with integral coefficients is a finitely generated abelian group. Among such groups, only $\Zo$ satisfies the property~\eqref{eqHomProperty}. This argument shows that $\Hr_*(|S_{<s}|;\Zo)\cong \Hr_*(S^{k-1};\Zo)$ as required.
\end{proof}

Proposition~\ref{propOneShotOnePosition} proves the implication $1\Rightarrow 2$ in Theorem~\ref{thmOneShotTheorem}. If every stalk $D(s)$ is used once by an honest functor $\Fu$, then every space $|S_{<s}|$ has homology of a sphere, which is the definition of a homology Morse cell poset.

\begin{rem}\label{remOneCohomologyOneBullet}
Informally, the proof of implication $1\Rightarrow 2$ in Theorem~\ref{thmOneShotTheorem} can be characterized by the following mantra. ``Let us treat stalks as bullets. If we want to shoot all cohomology with these bullets, we should be able in particular to shoot the subposets $S_{<s}$. To be able to do this in one shot, the total reduced homology of $S_{<s}$ should be 1-dimensional, i.e. $|S_{<s}|$ has homology of a sphere.''
\end{rem}

Somehow surprisingly the proof of the opposite implication is a bit more subtle. This proof utilizes certain intuitions from the proof of Proposition~\ref{propOneShotOnePosition} and eventually leads to a precise algorithm for cohomology computation for general posets to be described in the next subsection.

\begin{con}\label{conIncidenceRequirements}
Let $\ca{X}$ be a homology Morse poset. If $\sigma\in\ca{X}$ is a cell, then there exists $k\geq 0$ such that $|\dd C(\sigma)|=|\ca{X}_{<\sigma}|$ has homology of $S^{k-1}$. In Subsection \ref{subsecMathCellVariations} we called such $k$ the dimension of a cell $\sigma$. Notice that dimensions defined this way do not form a grading on a poset, they are not even monotonic: a situation may occur that $\sigma_1<\sigma_2$ while $\dim\sigma_1\geq \dim\sigma_2$ as demonstrated in Example~\ref{exMorseExample}. Nevertheless, we can try to construct a Morse chain complex $(C_*^{M}(\ca{X};\Zo),\dd)$ which computes singular homology of $|\ca{X}|$ in a geometrically meaningful way (to be made precise below). We set by definition
\[
C_j^{M}(\ca{X};\Zo)=\bigoplus_{\dim\sigma=j}\Zo,
\]
with $j$-dimensional cells $\sigma$ forming the distinguished basis of this module. The nontrivial part of the construction is to define the chain differential $\dd\colon C_j^{M}(\ca{X};\Zo)\to C_{j-1}^{M}(\ca{X};\Zo)$, i.e. the incidence numbers $\inc{\sigma}{\tau}\in\Zo$ such that $\dd(\sigma)=\sum_{\dim\tau=\dim\sigma-1}\inc{\sigma}{\tau}\tau$. We cannot just initialize incidence numbers at random if want to use them in honest cohomology calculations. They way how incidence numbers are used will tell us which properties they should satisfy.

Let's pretend that we already have some incidence numbers. Then we can construct a functor $\Fu_M\colon\Shvs(\ca{X},\Vv)\to\Cochain(\Vv)$ similar to cellular cochain complex:
\begin{equation}\label{eqMorseCochain}
\Fu_M(D)=C^*_M(\ca{X};D),\quad C^*_M(\ca{X};D)=\bigoplus_{\dim\sigma=j}D(\sigma),\quad d_M=\bigoplus_{\dim\tau=\dim\sigma-1}\inc{\sigma}{\tau} D(\tau<\sigma).
\end{equation}
\end{con}

In order for $\Fu_M$ to be well-defined and honestly compute sheaf cohomology we need three conditions to be satisfied.

\begin{req}\label{reqOnIncidenceNumbers}
\begin{enumerate}
  \item $\inc{\sigma}{\tau}=0$ if $\tau\nless\sigma$. Otherwise there are undefined summands in the differential $d_M$ in formula~\eqref{eqMorseCochain}.
  \item $\Fu_M$ should land in cochain complexes. This means $d_M^2=0$. Equivalently, incidence numbers should satisfy the diamond relation~\eqref{eqDiamond}. Equivalently, we should have $\dd^2=0$, so $C^M_*(\ca{X};\Zo)$ should be a chain complex.
  \item The augmented complex $0\to I(\ca{X})\to \Fu_M(I)$ is acyclic for any injective sheaf $I$ of the form $\low{\sigma}{W}$ with injective group $W$. This latter condition is equivalent to vanishing of the reduced Morse homology of any subposet $C(\sigma)=\ca{X}_{\leq\sigma}$ according to the universal coefficient theorem.
\end{enumerate}
\end{req}

We will use the requirements~\ref{reqOnIncidenceNumbers} to define incidence numbers inductively in Construction~\ref{conMorseIncidenceNumbers} below. To justify the induction step we need a technical construction and a lemma, which resembles a topological analogue of Theorem~\ref{thmHonestCohomology}.

\begin{con}\label{conMorseChainSubposet}
Recall from Remark~\ref{remSubposet} that (Morse) cell subposet of a (Morse) cell poset $\ca{Y}$ is any lower order ideal of $\ca{Y}$. Notice that whenever a Morse chain complex $(C_*^M(\ca{Y};\Zo), \dd_M)$ is defined on $\ca{Y}$ in the way that requirements~\ref{reqOnIncidenceNumbers} are fulfilled, it can be restricted, in a straightforward manner, to any cell subposet $\ca{Z}\subseteq\ca{Y}$. We denote the restricted chain complex by $(C^M_*(\ca{Z},\Zo),\dd_M)$ --- the incidence numbers are inherited from those on $\ca{X}$.
\end{con}

\begin{lem}\label{lemHonestHomology}
Assume that incidence numbers are defined on a homology Morse poset $\ca{Y}$ and satisfy the requirements~\ref{reqOnIncidenceNumbers}. Then, for any cell subposet $\ca{Z}\subseteq\ca{Y}$, the homology of the chain complex $(C^M_*(\ca{Z},\Zo),\dd_M)$ are isomorphic to the integral homology of the geometrical realization of $\ca{Z}$:
\[
H_i((C^M_*(\ca{Z},\Zo),\dd_M))\cong H_i(|\ca{Z}|;\Zo).
\]
\end{lem}

\begin{proof}
We prove the lemma by clumsily reducing it to something already proven~\footnote{Right in the spirit of anecdotes about mathematicians.}. Let $W$ be any injective abelian group. Consider the sheaf $W_{\ca{Z}}$ constantly supported on $\ca{Z}$ as in Construction~\ref{conConstantSheafOnLowerIdeal}. We have a sequence of isomorphisms
\begin{equation}\label{eqSequenceOfHomo}
  \begin{split}
\Hom(H_*(|\ca{Z}|;\Zo),W) & \stackrel{1}{\cong} H^*(|\ca{Z}|;W)\\
 & \stackrel{2}{\cong} H^*_{AG}(\ca{Y};W_Z)\\
 & \stackrel{3}{\cong} H^*(C^*_M(\ca{Y};W_{\ca{Z}}))\\
 & \stackrel{4}{\cong} H^*(\Hom((C_*^M(\ca{Z};\Zo),\dd_M),W))\\
 & \stackrel{5}{\cong} \Hom(H_*(C_*^M(\ca{Z};\Zo),\dd_M),W),
\end{split}
\end{equation}
where (1) is the universal coefficients isomorphism due to $W$ being injective (hence the functor $\Hom(\cdot,W)$ is exact), (2) is by Corollary~\ref{corLowerSheafIsSingularCohomology}, (3) is by the fact that Morse complex functor $\Fu_M$ honestly computes cohomology (since incidence numbers on $\ca{Y}$ are assumed to satisfy requirements~\ref{reqOnIncidenceNumbers}, we can apply Theorem~\ref{thmHonestCohomology} to this functor), (4) is just rewriting the definition of $C^*_M(\ca{Y};\ca{Z}_W)$, and, finally, (5) is once again an exactness of $\Hom(\cdot,W)$ for injective $W$.

If two groups $A,B$ are finitely generated and satisfy $\Hom(A,W)\cong\Hom(B,W)$ for any injective $W$, then $A\cong B$. Both groups $H_*(|\ca{Z}|;\Zo)$ and $H_*(C_*^M(\ca{Y};\Zo),\dd_M)$ are finitely generated, hence the claim follows from~\eqref{eqSequenceOfHomo}.
\end{proof}

Now we are finally ready to construct incidence numbers for any finite Morse cell poset.

\begin{con}\label{conMorseIncidenceNumbers}
At first, choose a topological sorting $\sigma_1,\ldots,\sigma_m$ of all cells of the Morse cell poset $\ca{X}$, i.e. a linear order in which, for any $\sigma_i<\sigma_j$ we have $i<j$ (all cells which are smaller than $\sigma_j$ in the partial order, appear earlier in the list). Let $\ca{X}_{(j)}$ denote the subposet $\{\sigma_1,\ldots,\sigma_j\}$ of $\ca{X}$, with the induced partial order. We also formally set $\ca{X}_{(0)}=\varnothing$. All posets $\ca{X}_{(j)}$ are Morse cell posets. Indeed, $(\ca{X}_{(j)})_{<\sigma}=\ca{X}_{<\sigma}=\dd C(\sigma)$ by the defining property of a topological sorting. The poset $\ca{X}_{(j)}$ is obtained from the poset $\ca{X}_{(j-1)}$ by adding a cell $\sigma_j$. Topologically, the space $|\ca{X}_{(j)}|$ is obtained from $|\ca{X}_{(j-1)}|$ by attaching the cone $|C(\sigma_j)|$ over $\dd|C(\sigma_j)|$ with the apex in $\sigma_j$. We construct incidence numbers $\inc{\sigma_j}{\tau}$ inductively for $j\in\{1,\ldots,m\}$. The base of induction $j=0$ is superfluous: there are no cells in the empty poset, hence no need to construct incidence numbers.

Assume that incidence numbers $\inc{\sigma_i}{\tau}$ are already defined for $i\leq j-1$ in a way that requirements~\ref{reqOnIncidenceNumbers} are satisfied on the poset $\ca{X}_{(j-1)}$. We need to extend this definition to the next filtration term $\ca{X}_{(j)}$. This means that we need to define $\inc{\sigma_j}{\tau}$ for the newly added cell $\sigma_j$. Let $k=\dim\sigma_j$. If $k=0$, there is no need to construct incidence numbers $\inc{\sigma_j}{\tau}$ since there are no cells $\tau$ of dimension $-1$ in the poset. Henceforth, in the following we assume $k\geq 1$.

The first natural requirement tells that $\inc{\sigma_j}{\tau}=0$ if $\tau\nless\sigma_j$. So far, we may restrict to only those cells $\tau$ which satisfy $\tau<\sigma_j$ and $\dim\tau=k-1$. Consider the subposet $\dd C(\sigma_j)=\ca{X}_{<\sigma_j}$ --- it is a cell subposet of $\ca{X}_{j-1}$. Since incidence numbers are already constructed on $\ca{Y}=\ca{X}_{j-1}$, we are in position to apply Lemma~\ref{lemHonestHomology} to the cell subposet $\dd C(\sigma_j)$, which gives an isomorphism
\[
H_*(C^M_*(\dd C(\sigma_j);\Zo),\dd_M)\cong H_*(|\dd C(\sigma_j)|;\Zo)
\]
Now, by assumption, the space $|\ca{X}_{<\sigma_j}|$ has the same integral homology as a sphere $S^{k-1}$, i.e. the homology are concentrated in degrees $0$ and $k-1$. Let us augment the chain complex $C^M_*(\dd C(\sigma_j);\Zo)$ with the term $\Zo$ and the map
\[
\dd\colon C^M_0(\dd C(\sigma_j);\Zo)=\bigoplus_{\dim\sigma=0}\Zo\to C^M_{-1}(\dd C(\sigma_j);\Zo)=\Zo,
\]
which sends all the chosen generators $\sigma$, $\dim\sigma=0$, to $1$. Then reduced homology $\Hr_*$ (i.e. the homology of the augmented complex $0\leftarrow\Zo\leftarrow C^M_*(\dd C(\sigma_j);\Zo)$) are concentrated in a single degree $k-1$. We have
\[
\Zo\cong \Hr_{k-1}(C^M_\ast(\dd C(\sigma_j);\Zo))=\dfrac{Z_{k-1}}{B_{k-1}}=\dfrac{\Ker\dd\colon C^M_{k-1}(\dd C(\sigma_j);\Zo)\to C^M_{k-2}(\dd C(\sigma_j);\Zo)}{\im \dd\colon C^M_k(\dd C(\sigma_j);\Zo)\to C^M_{k-1}(\dd C(\sigma_j);\Zo)}.
\]
Take any chain $c\in Z_{k-1}$ of the Morse chain complex which represents a generator of $\Zo\cong\Hr_{k-1}(C_*^M(\dd C(\sigma_j);\Zo))$. By construction, $c$ is an element of
\[
C^M_{k-1}(\dd C(\sigma_j);\Zo)=\bigoplus_{\tau\colon \dim\tau=k-1, \tau<\sigma_j}\Zo
\]
hence
\[
c=\sum\limits_{\tau\colon \dim\tau=k-1, \tau<\sigma_j}b_\tau \tau.
\]
Finally, we set $\inc{\sigma_j}{\tau}=b_\tau$.
\end{con}

\begin{lem}\label{lemIncNumbersInductionStep}
The incidence numbers on $\ca{X}_{(j)}$ defined by Construction~\ref{conMorseIncidenceNumbers} satisfy the requirements~\ref{reqOnIncidenceNumbers}.
\end{lem}

\begin{proof}
We do not need to check anything for the incidence numbers $\inc{\sigma_i}{\tau}$ with $i<j$, since these numbers were inherited from $\ca{X}_{(j-1)}$ and satisfy the requirements~\ref{reqOnIncidenceNumbers} by inductive hypothesis. So we only need to check the properties for $\inc{\sigma_j}{\tau}$. The first requirement is satisfied by assumption: we have $\inc{\sigma_j}{\tau}=0$ unless $\tau<\sigma_j$. The second requirement is also satisfied: by construction we have
\[
\dd_M(\sigma_j)=\sum\limits_{\tau\colon \dim\tau=k-1, \tau<\sigma_j}b_\tau \tau=c\in Z_{k-1}=\Ker\dd_M,
\]
hence $\dd_M^2(\sigma_j)=0$. Finally, the third property is about the subposet $\ca{X}_{\leq\sigma}$ being Morse-acyclic. It holds by construction: the Morse complex of the poset $C(\sigma)$ differs from the Morse complex of $\dd C(\sigma)$ in one term $\Zo$ of degree $k=\dim\sigma$. When passing to homology, this term kills the unique $(k-1)$-homology of $\dd C(\sigma)$ (without any torsion left, since $[c]$ was chosen to be a generator of $H_{k-1}(\dd C(\sigma);\Zo)$), --- and dies itself.
\end{proof}

Lemma~\ref{lemIncNumbersInductionStep} completes the induction step in the construction of incidence numbers on a Morse cell poset $\ca{X}$. Using these incidence numbers, we define the functor $\Fu_M\colon \Shvs(\ca{X};\Vv)\to\Vv$ by~\eqref{eqMorseCochain}. It is a concrete one-shot functor which honestly computes cohomology of sheaves on $\ca{X}$. The implication $2\Rightarrow 1$ of Theorem~\ref{thmOneShotTheorem} is proven.

\begin{rem}
It is possible to consider diagrams over arbitrary finite categories: cohomology can be defined as right derived functors similar to the case of posets. In this generality, one-shot computations of sheaf cohomology can be extended to cellular categories as defined in~\cite{CellCat}.
\end{rem}

\subsection{Structure agnostic cohomology computations}\label{subsecMathMinimalComputations}

One might wonder, what prevents us from extending Remark~\ref{remDeletePoint} (about redundant stalks) and Theorem~\ref{thmOneShotTheorem} (about stalks used exactly once) to more general classes of posets? Literally nothing! We develop more general approach in this subsection and provide an algorithm for sheaf cohomology computations which satisfies two properties.
\begin{enumerate}
  \item The algorithm is applicable to sheaves over all posets. Its complexity (see Definition~\ref{definMultiplicityOfStalk}) is lower than complexity of Roos complex, and is proven to be the minimal possible.
  \item On the class of homology Morse cell posets (in particular on cell posets) this algorithm becomes the classical algorithm of cohomology calculation via cellular (or Morse) cochain complex: in this case the algorithm computes incidence numbers.
\end{enumerate}
The ideas of this subsection are similar to the idea of local homology sheaf applied in stratification learning~\cite{brown2021sheaf}, although the ultimate goals are different. While stratification learning aims to build a convenient filtration on a geometrical structure by utilizing sheaf theory~\footnote{As well as, implicitly Zeeman--McCrory spectral sequence~\cite{McCrory}}, our goal is to construct an optimized algorithm for general sheaf cohomology computation. Our considerations also have much in common with those of~\cite{CIANCI20171}, where similar arguments were applied to constant sheaves on posets.

\begin{defin}\label{definMultiplicityOfStalk}
Assume that a concrete functor $\Fu\colon \Shvs(X_S;\Vv)\to\Cochain(\Vv)$ honestly computes cohomology of sheaves on $S$. The number of times the stalk $D(s)$ appears as a direct summand in the $k$-th graded component $\Fu(\ca{D})^k$ is called \emph{the multiplicity of} $s$, and is denoted $\mult_k(s,\Fu)$. The total number
\[
c(\Fu)=\sum_{s\in S, k\in\Zo}\mult_k(s,\Fu)
\]
is called \emph{the complexity of the functor} $\Fu$.
\end{defin}

\begin{ex}\label{exMultiplicityCW}
For a cell poset $\ca{X}$ and a cellular cochain functor $\Fu_{CW}\colon \Shvs(\ca{X};\Vv)\to\Cochain(\Vv)$ we have
\[
\mult_k(\sigma,\Fu_{CW})=\begin{cases}
                           1, & \mbox{if } k=\dim\sigma \\
                           0, & \mbox{otherwise}.
                         \end{cases}
\]
for any cell $\sigma$. Notice that, whenever $\dim\sigma=0$, we necessarily have $\ca{X}_{<\sigma}=\varnothing$, and formally $\Hr_{-1}(\varnothing;\ko)=\ko$ (this homology comes from the augmentation term). Therefore $0$-dimensional cells are not exceptions in this example. The complexity $c(\Fu_{CW})$ equals the number of cells of $\ca{X}$.
\end{ex}

\begin{ex}\label{exMultiplicityMorse}
For a homology Morse poset $\ca{X}$ and a Morse cochain functor $\Fu_{M}$, similar to the previous example, we have $\mult_k(\sigma,\Fu_M)=1$ if $\dim\sigma=k$, and $0$ otherwise. The complexity $c(\Fu_{M})$ again equals the number of cells of $\ca{X}$. These are precisely all one-shot functors according to Theorem~\ref{thmOneShotTheorem}.
\end{ex}

\begin{ex}\label{exMultiplicityRoos}
For the Roos cochain functor $\Fu_{Roos}$ we have
\[
\mult_{\Fu_{Roos}}(s)=f_{k-1}(\ord(S_{<s})),
\]
where $f_{j}(\ca{K})$ denotes the number of $j$-dimensional simplices of a nonempty simplicial complex $\ca{K}$, while it is formally assumed that $f_{-1}(\ca{K})=1$ for any simplicial complex $\ca{K}$, even for $\ca{K}=\varnothing$. This formula follows easily from the count of chains $s_1<\cdots<s_k<s$ of $S$ which end in a given element $s$. The complexity $c(\Fu_{Roos})$ equals the number of simplices of $\ord(S)$.
\end{ex}

To avoid complications with torsion, for the rest of this subsection we work with the target category $\Vv=\ko\Vect$, the category of vector spaces over a fixed field $\ko$.

\begin{thm}\label{thmMorseBound}
For any functor $\Fu$ which honestly computes cohomology of sheaves on a poset $S$, and for any element $s\in S$, the multiplicity $\mult_k(s,\Fu)$ is at least the total reduced Betti number $\br_{k-1}(|S_{<s}|)=\dim\Hr_{k-1}(|S_{<s}|;\ko)$ of the down-set of $s$.
\end{thm}

\begin{proof}
The proof is completely similar to Proposition~\ref{propOneShotOnePosition}. Consider a short exact sequence of sheaves on $S$:
\begin{equation}\label{eqShortExactDirac}
0\to \delta_s(\ko)\to\low{s}{\ko}\to \low{\bar{s}}{\ko}\to 0,
\end{equation}
where $\low{s}{\ko}$ is the basic injective sheaf defined in Construction~\ref{conInjSheaves}, $\low{\bar{s}}{\ko}$ is defined in~\eqref{eqCutLow}, and $\delta_s(\ko)$ is the Dirac diagram supported on $s$, see Construction~\ref{conDiracDiagram}. Notice that $\dim \Fu(\delta_s(\ko))^k=\mult_k(s,\Fu)$ --- because stalks in all points except $s$ vanish, and $\delta_s(\ko)(s)=\ko$ contributes to the $k$-th graded component of the cochain complex $\mult_k(s,\Fu)$ many times by construction.

Taking long exact sequence of cohomology induced by the short exact sequence~\eqref{eqShortExactDirac}, and noticing that the sheaf $\low{s}{\ko}$ is acyclic, we see that $H_{AG}^*(X_S;\delta_s(\ko))\cong \Hr_{AG}^{*-1}(X_S;\low{\bar{s}}{\ko})$. However, by Corollary~\ref{corLowerSheafIsSingularCohomology} we know that $\Hr_{AG}^{*-1}(X_S;\low{\bar{s}}{\ko})$ is isomorphic to the singular cohomology $\Hr^{*-1}(|S_{<s}|;\ko)$. On the other hand, $H_{AG}^*(X_S;\delta_s(\ko))\cong H^*(\Fu(\delta_s(\ko)))$, since $\Fu$ honestly computes cohomology. Since cohomology is a subquotient of the cochain complex, we have
\[
\br_{k-1}(|S_{<s}|)=\dim H^{k}(\Fu(\delta_s(\ko)))\leq \dim \Fu(\delta_s(\ko))^k=\mult_k(s,\Fu),
\]
as needed.
\end{proof}

\begin{rem}\label{remEulerChar}
It follows from the proof of Theorem~\ref{thmMorseBound} that $\sum_{k=0}^{\infty}(-1)^k\mult_{k,\Fu}=\sum_{k=0}^{\infty}(-1)^k\br_{k-1}(|S_{<s}|)$, whenever the first sum makes sense. Indeed, the first expression is the Euler characteristic of the cochain complex $\Fu(\delta_s(\ko))^k$, while the second one is the Euler characteristic of its cohomology. It is well-known that Euler characteristics of any cochain complex and that of its cohomology coincide (see e.g.~\cite[Thm.2.44]{Hatcher} or~\cite[Lm.5.1]{robinson2014topological}. Together with Theorem~\ref{thmMorseBound}, this remark allows to derive Proposition~\ref{propOneShotOnePosition} in the category $\Vv=\ko\Vect$.
\end{rem}

The next statement asserts that the estimation given by~\ref{thmMorseBound} is exact. There exists a functor $\Fu_{\min}$ which realizes the lower theoretical bound on multiplicity for all elements $s\in S$ at once, hence on the total complexity.

\begin{thm}\label{thmExactBound}
For any finite poset $S$, there exists a concrete functor
\[
\Fu_{\min}\colon\Shvs(X_S;\Vv)\to\Cochain(\Vv)
\]
which honestly computes cohomology and satisfies $\mult_k(s,\Fu_{\min})=\br_{k-1}(|S_{<s}|)$ for any $s\in S$.
\end{thm}

To prove this statement we provide an algorithm which constructs the required functor $\Fu_{\min}$ given a finite poset $S$. First we give some specifications and definitions to simplify the implementation of this algorithm.

\begin{con}\label{conChainCpxDataStructure}
The main step, given an arbitrary poset $S$, is to construct a chain complex $(C_*(S;\ko),\dd)$ which stores information on ``incidence numbers'' (more preciesly, their analogues when no assumption on topology of $|S_{<s}|$ is posed). 

To encode a (finite-dimensional) chain complex $(C_*,d)$ over a field $\ko$ we use the data structure $\ChCpx$ which consists of
\begin{enumerate}
  \item a finite set $A$, denoting a set of homogeneous generators of the total space $\bigoplus_{j\geq 0}C_j$;
  \item a grading $\deg\colon A\to \Zg$;
  \item a map $M\colon A\times A\to \ko$, which denotes the matrix of an operator $\dd$ in the chosen basis $A$ (first component stands for input=indicates the columns, the second stands for the output and indicates rows).
\end{enumerate}
The contract on $\ChCpx$ (assumed but not checked explicitly) is that (1) for each $a,c\in A$ we have $\sum_{b\in A}M(a,b)M(b,c)=0$ (i.e. $d^2=0$), (2) for each $a,b\in A$ with $\deg(a)\neq \deg(b)+1$, $M(a,b)=0$ (the differential lowers degree by $1$).

We also have the augmented version of this structure, $\AChCpx$, where the grading takes values in $\{-1\}\sqcup\Zg$. The cochain complex $\CochCpx$ can be defined by transposing the matrix $M$, i.e. inverting the order of arguments (the only difference from chain complexes is that cohomological differential raises degree by $1$).
\end{con}

Two auxiliary procedures will be used in the algorithm, whose contracts are described below.

\begin{enumerate}
  \item A function HomologyBasis consumes $\ChCpx$ (resp., $\AChCpx$) and outputs representing chains of generators of its homology, i.e. an indexing set $H$ together with the degree function $\deg\colon H\to \Zg$ (resp. $\{-1\}\sqcup\Zg$), and the representative chains $RC\colon H\to (A\to \ko)$, satisfying the condition that $\deg(a)=\deg(i)$ for any nonzero coefficient $RC(i)(a)\neq 0$. The procedure HomologyBasis has standard realizations based on variations of Gauss algorithm, see~\cite{KACZYNSKI199859} and the implementation available in~\cite{Sage}.
  \item A function Restrict consumes an augmented chain complex $(C_*,d)$, on a set $A$ of generators, and a subset $B\subset A$, and outputs the restricted complex, that is simply $B\times B$-block of the matrix of the differential. We assume the restriction thus defined makes sense, i.e. the differential preserves the vector subspace spanned by $B$.
  \item A function TopoSort consumes a poset $S$ and outputs an ordered list $(s_1,\ldots,s_m)$ of its elements which is a linearization of $S$.
\end{enumerate}

Returning back to the complex associated with a poset, a vector space $C_*(S;\ko)$ will have type $\AChCpx$ and consist of the triple $(G^a,\deg,I)$, where $G^a$ is the set of generators and $I$ is the matrix of ``incidence numbers''. All generators from $G^a$ except the augmentation term will be stitched to specific elements of $S$, so we fix a decomposition $G^a=\{\varnothing\}\sqcup G$, and $G=\bigsqcup_{s\in S}G_s$.
%

The set $G^a$ of generators and the matrix $I$ are constructed inductively. We first choose a topological sorting of $S$ and proceed by computing and killing homology of all down-sets $|S_{<s}|$ one by one, according to the following algorithm. The values of $I(a,b)$ which are not explicitly specified are set equal to $0$. Whether this procedure is actually needed depends on the representation format of matrices. For example, if the matrices are represented in sparse format, this procedure is unnecessary. However, in the subsequent algorithm of sheaf cohomology computation, it will be more convenient to work with matrices represented as dense matrices, or double arrays.

\begin{algorithm}
    \caption{Incidence matrix computation}\label{algIncMatrixMain}
    \begin{algorithmic}[1]
    \REQUIRE Poset $S$.
    \STATE \textbf{procedure} $S_{sorted}$=TopoSort($S$)
    \STATE \textbf{create} generator $\varnothing$ with $\deg(\varnothing)=-1$
    \STATE \textbf{define} $C=(\{\varnothing\},\deg,I)$ of type $\AChCpx$
    \FOR{each $s$ in $S_{sorted}$}
        \STATE \textbf{define} $G_s=\{\}$
        \STATE \textbf{define} $S_{<s}=\{t\in S\mid t<s\}$
        \STATE \textbf{define} $G|_{<s}=\{\varnothing\}\sqcup\bigsqcup_{t\in S_{<s}}G_t$
        \STATE \textbf{procedure} $C|_{<s}$=Restrict($C$, $G|_{<s}$)
        \STATE \textbf{procedure} $(H,\deg,RC)$=HomologyBasis($C|_{<s}$)
        \FOR{each homology generator $i$ in $H$}
            \STATE \textbf{create} generator $g$ with $\deg(g)$ equal to $\deg(i)+1$
            \STATE \textbf{append} $g$ to $G_s$
            \FOR{each $h$ in $G$}
                \STATE \textbf{set} $I(h,g)$ equal to $0$
                \IF{$h$ lies in $G|_{<s}$}
                    \STATE \textbf{set} $I(g,h)$ equal to $RC(i)(h)$
                \ELSE
                    \STATE \textbf{set} $I(g,h)$ equal to $0$
                \ENDIF
            \ENDFOR
        \ENDFOR
        \STATE \textbf{append} $G_s$ to $G^a$
    \ENDFOR
    \STATE \textbf{output} set of generators $G^a$, matrix $I$
    \end{algorithmic}
\end{algorithm}

\begin{proof}[Proof of Theorem~\ref{thmExactBound}]
As long as the set of generators $G^a$ and the incidence matrix $I$ are computed by Algorithm~\ref{algIncMatrixMain}, they can be used to compute cohomology of any sheaf $D$ on $S$. Since we have a decomposition $G^a=\{\varnothing\}\sqcup G$ and $G=\bigsqcup_{s\in S}G_s$, it will be convenient to adopt a convention from dependent type theory: if $g\in G$, then $\pr_S(g)\in S$ denotes the index of the corresponding component of the union, and $\pr_G(g)\in G_{\pr_S(g)}$ denotes the corresponding element of this component.

Recall that $I(g,h)$ is the element of the matrix $I$, lying in $g$-th column and $h$-th row. Mathematically, given a diagram $D$ on $S$ (a sheaf $\ca{D}$ on the corresponding Alexandrov topological space $X_S$) we define the cochain complex $(C^*_{\min}(S;D), d_{\min})$ by the formulas
\begin{equation}\label{eqMinComplex}
C^j_{\min}(S;D)=\bigoplus_{g\in G,\deg g=j}D(\pr_S(g)), \qquad d_{\min}^j\colon C^j_{\min}(S;D)\to C^{j+1}_{\min}(S;D)
\end{equation}
\begin{equation}\label{eqMinComplexDifferential}
d_{\min}^j=\bigoplus_{\substack{g_1,g_2\in G}{\deg(g_1)=j,\deg(g_2)=j+1}}I(g_2,g_1)D(\pr_S(g_1)<\pr_S(g_2)).
\end{equation}
for any $j=0,1,\ldots$. We set $\Fu_{\min}(D)=(C^*_{\min}(S;D), d_{\min})$. This construction satisfies the following properties.

\begin{enumerate}
  \item The differential is well-defined since $\inc{g_2}{g_1}$ may be nonzero only if $\pr_S(g_1)<\pr_S(g_2)$, as can be seen directly from Algorithm~\ref{algIncMatrixMain}.
  \item The differential $d_{\min}$ satisfies $d_{\min}^2=0$. The proof is similar to Construction~\ref{conMorseIncidenceNumbers}. Hence $\Fu_{\min}$ lands in $\Cochain(\Vv)$.
  \item Since all summands in~\eqref{eqMinComplex} are functorial, $\Fu_{\min}$ is actually a functor from $\Diag(S;\Vv)$ to $\Cochain(\Vv)$.
  \item Finally, we have $H^*(W\to\Fu_{\min}(\low{s}{W}))=0$ for any $s\in S$ and any vector space $W=\ko^n$ (which is automatically injective). Indeed, we have $\Hr_*(C_*(S_{\leq s};\ko),\dd)=0$ by construction of the incidence matrix in Algorithm~\ref{algIncMatrixMain}. Applying the exact functor $\Hom(\cdot,W)$ we get
      \begin{multline*}
      H^*(W\to\Fu_{\min}(\low{s}{W}))=H^*(\Hom(C_*(S_{\leq s};\ko)\to\ko,W))= \\
      \Hom(H_*(C_*(S_{\leq s};\ko)\to\ko),W)=0.
      \end{multline*}
\end{enumerate}

The assumption of Theorem~\ref{thmHonestCohomology} are satisfied, therefore $\Fu_{\min}$ honestly computes cohomology of sheaves on $S$. The Algorithm~\ref{algIncMatrixMain} can be backtracked and it is seen that the multiplicity
\[
\mult_j(s,\Fu_{\min})=\#G_s
\]
of a stalk $D(s)$ in the vector space $C^j_{\min}(S;D)$ equals the rank $\dim \Hr_j'$ where
\[
\Hr_j'=\Hr_j(C|_{<s}),
\]
is the homology of the complex restricted to generators associated with elements of $S_{<s}$. Notice that $S_{<s}$ is a subset of $S'=\{s_1,\ldots,s\}$, the interval of topological sorting which have already been constructed. By induction, the incidence numbers already defined on $S'$ give honest sheaf cohomology computation for sheaves on $S'$. According to Corollary~\ref{corLowerSheafIsSingularCohomology} (also see details in the proof of Lemma~\ref{lemIncNumbersInductionStep}), the singular homology of lower order ideals of $S'$ can be honestly computed with the same incidence numbers (the restricted chain complex). This is applied to the lower order ideal $S_{<s}\subset S'$. Therefore, it is proven by induction, that
\[
\Hr_j(C|_{<s})\cong \Hr_j(|S_{<s}|;\ko).
\]
Therefore $\mult_j(s,\Fu_{\min})=\dim \Hr_j(|S_{<s}|;\ko)$ which is the predicted lower bound for the multiplicity. Theorem is finally proved.
\end{proof}

\begin{rem}\label{remExternalSolverIntegral}
Notice, that the problem of computing chains representing a basis of homology is outsourced to an external solver HomologyBasis. If the solver can deal with integral matrices, and if it is known that homology of $|S_{<s}|$ are torsion-free, then it is natural to require that the solver outputs not just any basis of homology, but an integral basis, i.e. the set of integral generators. However, things become more tricky if $|S_{<s}|$ has torsion in homology: in this case more generators are required to kill the homology. It is still possible to construct an integral analogue of the functor $\Fu_{\min}$ algorithmically, however, we postpone such description for the future research.
\end{rem}

\begin{rem}\label{remSparsity}
In the case of coefficients in a field, it is natural to define some optimization priors on the work of the homology solver. Notice that the solver should output the generators of a homology group $H_j(C|_{<s})$ which are classes of chains of the corresponding complex. There is a freedom to choose the representatives. A natural prior would be to require the sparsity of the resulting representative chains. The more sparse is the incidence matrix, the easier is the subsequent (co)homology computation. This leads to an interesting practical problem in applied linear algebra.
\end{rem}

\begin{probl}\label{problSparseSolution}
Consider a couple of vector subspaces $B\subset Z\subset C=\ko^n$. Find vectors $c^{(1)},\ldots,c^{(k)}\in C$ such that
\begin{enumerate}
  \item the set $\{[c^{(1)}],\ldots,[c^{(k)}]\}$ is a basis of $Z/B$,
  \item the total number $\sum_{i=1}^{k}\#\{j\in[n]\mid c_j^{(i)}\}$ of vanishing components is maximized.
\end{enumerate}
\end{probl}

An algorithm, which effectively solves Problem~\ref{problSparseSolution} should be preferred in a system which computes cohomology of sheaves.

\begin{rem}\label{remMorseExampleAndApproximation}
Recalling Example~\ref{exMorse}, we see a freedom in the definition of the incidence numbers, namely, we can set $\dd(b)=k_1\mathbf{1}+k_2\mathbf{2}-\mathbf{3}$ with any choice of parameters $(k_1,k_2)\in\Zo^2$ such that $k_1+k_2=1$, and it leads to the honest cohomology computations. According to Remark~\ref{remSparsity}, we should prefer sparse solutions, which are either $(k_1,k_2)=(1,0)$ or $(k_1,k_2)=(0,1)$. Notice that these two ways of defining a differential neatly correspond to the two geometrically meaningful ways of turning a Morse poset $\ca{Y}$ into a cell poset, see Fig.~\ref{figElementaryMorse}. In certain sense, the requirement of sparsity of solutions at each step somehow resembles cellular approximation theorem --- in a homological language. 
\end{rem}

\begin{rem}
It follows from Algorithm~\ref{algIncMatrixMain} that, whenever $|S_{<s}|$ is acyclic, there are no generators $g$ related to the element $s$. This means that the stalk $D(s)$ does not contribute to $C_{\min}^*(S;D)$ at all. In particular, it follows that cohomology of any sheaf $D$ do not depend on the stalk $D(s)$ with acyclic downset $D_{<s}$. This is well-aligned with Remark~\ref{remDeletePoint}.
\end{rem}

For convenience we formulate an algorithm to compute cohomology of arbitrary diagram $D$ on a finite poset $S$.

\begin{con}
First of all, we encode a diagram $D$ by specifying the following data structures
\begin{enumerate}
  \item for each $s\in S$, a finite set $N_s$ is specified. This corresponds to a basis of the stalk $D(s)$;
  \item for each pair $s,t\in S$ such that $s<t$ in $S$, a function $D_{st}\colon N_s\times N_t\to \ko$ is specified. This function encodes the matrix of the operator $D(s<t)$ written in the chosen bases.
\end{enumerate}
Notice that there are many redundant matrices in this structure: since the compositionality holds for a sheaf, it is sufficient to define the matrices $D_{st}$ only for the edges $(s,t)$ of the Hasse diagram of $S$ (see Construction~\ref{conHasseDiag}), while all other matrices can be recovered as their products. However, it will be more convenient to work with this redundant structure: the several-hop matrices
\[
D(s_0<\cdots<s_k)=D(s_{k-1}<s_k)\cdots D(s_0<s_1)
\]
may appear explicitly in the calculation.
\end{con}

It will be assumed that the set $G=\bigsqcup_{s\in S}G_s$ of generating elements, and the incidence matrix $I\colon G\times G\to\ko$ are already precomputed by Algorithm~\ref{algIncMatrixMain}. The following algorithm is mathematically straightforward: we form the structure $C'$ of type $\CochCpx$ representing the cochain complex $(C^*_{\min}(S;D),d_{\min})$, send send it to HomologyBasis which outputs the set of cochains representing cohomology generators together with their degrees. Recall from Construction~\ref{conChainCpxDataStructure} that $C'$ is a triple $(N,\deg',\Dif)$, where $N$ is a finite set of basic vectors, $\deg\colon N\to\Zg$ is the degree function, and $\Dif\colon N\times N\to \ko$ is the matrix of the differential, assumed to shift degrees by $+1$.

As follows from the definition of the cochain complex, we have $N=\bigsqcup_{g\in G}N_{\pr_S(g)}$. Again, given $n\in N$, the notation $\pr_G(n)\in G$ stands for the corresponding index, and $\pr_N(n)\in N_{\pr_S(g)}$ the element of the corresponding component.

\begin{algorithm}
    \caption{Sheaf cohomology computation}\label{algMainAlg}
    \begin{algorithmic}[1]
    \STATE \textbf{set} $N$ equal to $\bigsqcup_{g\in G}N_{\pr_S(g)}$
    \FOR{each $n$ in $N$}
        \STATE \textbf{set} $\deg(n)$ equal to $\deg(\pr_G(n))$
    \ENDFOR
    \FOR{each $(n_1,n_2)$ in $N\times N$}
        \STATE \textbf{set} $\Dif(n_1,n_2)$ equal to $I(\pr_G(n_2),\pr_G(n_1))\cdot D_{\pr_S(\pr_G(n_1))\pr_S(\pr_G(n_2))}(\pr_N(n_1),\pr_N(n_2))$
    \ENDFOR
    \STATE \textbf{define} $C'=(N,\deg,\Dif)$ of type $\CochCpx$
    \STATE \textbf{procedure} $(H',\deg,RC')$=HomologyBasis($C'$)
    \STATE \textbf{output} $(H,\deg,RC)$
    \end{algorithmic}
\end{algorithm}

Notice that the matrix $\Dif$ is well defined. The factor
\[
D_{\pr_S(\pr_G(n_1))\pr_S(\pr_G(n_2))}(\pr_N(n_1),\pr_N(n_2))
\]
is undefined when $\pr_S(\pr_G(n_1))\nless\pr_S(\pr_G(n_2))$. However, in this case the factor
\[
I(\pr_G(n_2),\pr_G(n_1))
\]
vanishes, as guaranteed by Algorithm~\ref{algIncMatrixMain}, so their product is set equal to $0$.

\begin{rem}\label{remAgnosticComputations}
Notice that Theorems~\ref{thmMorseBound} and~\ref{thmExactBound} do not guarantee that, given a poset $S$ and a sheaf $D$, the complex $\Fu_{\min}(D)$ is the most optimal way to compute cohomology of $D$. As stupid example: if $D$ identically vanishes, the cohomology vanish as well, so there is no need to compute matrices $I$ and $\Dif$. Instead, the theorems assert that, given a poset $S$, the complex $\Fu_{\min}(D)$ is the most optimal for all possible diagrams $D$ on $S$. So the application of Algorithms~\ref{algIncMatrixMain} and~\ref{algMainAlg} makes sense in the context when $S$ is given, and there is a need to compute cohomology of several possible sheaves on $S$; the matrix $I$ depends only on $S$ hence can be reused in the calculations.

Certainly, in specific situation, the matrix $I$ can be found without using Algorithm~\ref{algIncMatrixMain}. For example, if $S$ is known to be a simplicial complex, the incidence numbers are defined in the alternating fashion, as in Construction~\ref{conStandardSimplicialIncNumbers}, which doesn't require the calculation of homology of the down-sets. Algorithm~\ref{algIncMatrixMain} is applicable to any finite poset, hence we call it structure agnostic.

The development of algorithms which efficiently compute cohomology for a single diagram $D$ on $S$ is a challenging problem, which is a subject of the future work.
\end{rem}

\section{Laplacians, diffusion, and beyond}\label{secMathLaplacians}

In this section we restrict to the category $\Ro\Vect$ of real vector spaces. The field $\Ro$ has two important features, distinguishing it from general fields $\ko$ and making it suitable for practical applications. (1) It is ordered. (2) It is a complete metric space. By utilizing the first property one can define the general combinatorial Hodge theory. The second property leads to heat diffusion on sheaves which underlies the usage of sheaves in the design of neural networks. Most mathematical claims about Laplacians follow from the basic linear algebra reviewed in the next subsection. 

\subsection{Linear algebra preliminaries}\label{subsecMathLinAlg}

A euclidean space is a real vector space $V\in\Ro\Vect$ with the chosen inner product $\langle\cdot,\cdot\rangle$. Each (finite-dimensional) euclidean space $V$ is naturally identified with its dual $V^*$ via the isomorphism $v\mapsto \langle v,\cdot\rangle$. Any map $f\colon V\to W$ between euclidean spaces induces the conjugate map $f^*\colon W=W^*\to V=V^*$. If a map $f$ is written as a matrix $F$ in some orthogonal basis, then its conjugate $F^*$ is written with the transpose matrix $F^t$. If $V=\Ro^n$ is the arithmetical space, then we assume the scalar product is chosen in a standard way $\langle v,u\rangle=\sum_{i=1}^{n}v_iu_i$.

The next lemma is the classical statement, to be found in any elementary book on linear algebra.

\begin{lem}\label{lemKerFFisKerF}
For a linear map $f\colon V\to W$, we have $\Ker f=\Ker f^*f$.
\end{lem}

\begin{proof}
If $fv=0$, then $f^*fv=0$, so the inclusion $\Ker f\subseteq\Ker f^*f$ is obvious. On the converse, if $f^*fv=0$, then
$0=\langle f^*fv,v\rangle=\langle fv,fv\rangle$, which implies $fv=0$ since the inner product is chosen non-degenerate.
\end{proof}

\begin{con}\label{conSelfAdjoint}
Recall that an operator $A\colon V\to V$ on euclidean space is called self-adjoint (symmetric matrix) if $A^*=A$. Every operator of the form $f^*f$ or $gg^*$ is obviously self-adjoint. The sum of self-adjoint operators is self-adjoint. The spectral theorem asserts that any self-adjoint operator can be diagonalized, i.e. there exists an orthogonal basis of $V$ in which $A$ is written as a diagonal matrix $\Lambda=\diag(\lambda_1,\ldots,\lambda_n)$. In other words, for any real symmetric matrix $A$ there exists a spectral decomposition $A=Q^t\Lambda Q$ where $Q$ is orthogonal, and $\Lambda$ is diagonal. The diagonal entries of $\Lambda$ are eigenvalues of $A$, and the columns of $Q$ are the eigenvectors of $A$. The problem of finding $Q$ and $\Lambda$ for a given symmetric matrix $A$ is called the symmetric diagonalization problem. It can be solved by approximate algorithms of numerical linear algebra with floating point, see e.g.~\cite{NumLinAlg}.
\end{con}

\begin{con}\label{conNonNegative}
An operator $A\colon V\to V$ is called nonnegative~\footnote{We use this term as a shorter version of a more standard ``positive semidefinite''.}, if $\langle Av,v\rangle\geq 0$. The proof of Lemma~\ref{lemKerFFisKerF} shows that any operator of the form $f^*f$ or $gg^*$ is nonnegative. The sum of nonnegative operators is easily seen to be nonnegative. If $A$ is self-adjoint, then $A$ is nonnegative if and only if all its eigenvalues $\lambda_i$ are nonnegative real numbers. It can be seen from the spectral theorem, that if $A$ is nonnegative self-adjoint, then
\begin{equation}\label{eqKernelIsQuadraticMinimum}
\Ker A=\{v\in V \mid \langle Av,v\rangle=0\}.
\end{equation}
An important observation follows from formula~\eqref{eqKernelIsQuadraticMinimum}. With any nonnegative symmetric matrix $A$, we can associate a function $Q_A\colon V\to\Ro$, $Q_A(x)=\langle Ax,x\rangle$, which is (1) quadratic, (2) nonnegative, (3) convex, (4) most importantly,
\begin{equation}\label{eqQuadraticMinimization}
Q_A(x)=0 \Leftrightarrow x\in\Ker A.
\end{equation}
Hence the problem of finding a solution to the equation $Ax=0$ can be faithfully reformulated as a problem of minimization of $Q_A$. 

Notice that, whenever $A=f^*f$, we get 
\begin{equation}\label{eqStupidObservation}
Q_A(x)=\langle x,f^*fx\rangle = \|fx\|^2.
\end{equation}
\end{con}

\begin{rem}\label{remKernelsIntersect}
If $A_1,A_2$ are two nonnegative self-adjoint operators, then so is $A_1+A_2$. Since $Q_{A_1+A_2}=Q_{A_1}+Q_{A_2}$, it easily follows from the considerations of Construction~\ref{conNonNegative}, that
\[
\Ker(A_1+A_2)=\Ker A_1\cap \Ker A_2,
\]
whenever $A_1,A_2$ are nonnegative self-adjoint.
\end{rem}

With all this been written, the next important lemma is not much more complicated than Lemma~\ref{lemKerFFisKerF}. It constitutes the basis of Hodge theory in the finite-dimensional case.

\begin{lem}\label{lemQuotientIsHarmonic}
For a sequence of two linear maps $U\stackrel{g}{\rightarrow} V\stackrel{f}{\rightarrow} W$ between euclidean spaces define an operator $L=f^*f+gg^*\colon V\to V$. Then we have an isomorphism
\[
\Ker L\cong \Ker f/(\Ker f\cap \im g).
\]
\end{lem}

\begin{proof}
Both $L_1=gg^*$ and $L_2=f^*f$ are nonnegative self-adjoint operators on $V$. We have
\begin{equation}\label{eqLinAlg}
\Ker L=\Ker(L_2+L_1)=\Ker L_2\cap \Ker L_1=\Ker f\cap \Ker g^*=\Ker f\cap (\im g)^\bot.
\end{equation}
Whenever $A,B$ are two subspaces of the euclidean space $V$, we have an isomorphism $A\cap B^{\bot}\cong A/(A\cap B)$, given by $a\mapsto [a]$. Therefore, the rightmost space in~\eqref{eqLinAlg} is naturally isomorphic to $\Ker f/(\Ker f\cap \im g)$.
\end{proof}

\subsection{Laplacians and combinatorial Hodge theory}\label{subsecMathHodge}

Recall the definition of cochain complex from subsection~\ref{subsecMathCochainRecap}.

\begin{defin}\label{definLaplaceGeneralCochain}
Consider a connective cochain complex $(C^*,d)$ in the category $\Ro\Vect$:
\[
0\to C^0\stackrel{d_0}{\to} C^1 \stackrel{d_1}{\to} C^2 \stackrel{d_2}{\to}\cdots
\]
Assume an inner product $\langle\cdot,\cdot\rangle$ is somehow chosen on every vector space $C_i$. Then we can define the sequence of nonnegative operators:
\[
\Delta_j=d_j^*d_j+d_{j-1}d_{j-1}^*\colon C^j\to C^j,
\]
called \emph{Laplace operators (or Laplacians)}. The vector subspace $\ca{H}_j=\Ker\Delta_j$ is called \emph{the harmonic subspace} of $C^j$, and its elements --- \emph{harmonic cochains}.
\end{defin}

We have a natural corollary of Lemma~\ref{lemQuotientIsHarmonic}

\begin{cor}\label{corHarmonicIsCohomology}
The $j$-th harmonic subspace of the euclidean cochain complex $(C^*,d)$ is (non-naturally) isomorphic to its $j$-th cohomology module of $(C^*,d)$:
\begin{equation}\label{eqHarmonicIsCohomology}
\ca{H}_j\cong H^j(C^*,d).
\end{equation}
\end{cor}


\begin{con}\label{conInnerProductCellCpx}
Let $D$ be a cellular sheaf defined on a cell poset $\ca{X}$, see Definitions~\ref{definCellPoset} and~\ref{definCellularSheaf}. Assume an inner product is chosen on each stalk $D(\sigma)$. This gives a canonical inner product on the cellular cochain spaces $C^j_{CW}(\ca{X};D)=\bigoplus_{\dim\sigma=j}D(\sigma)$, where direct summands are assumed orthogonal.
\end{con}

\begin{defin}\label{definCellSheafLaplacian}
The Laplacian $\Delta_j$ of the cochain complex $(C^j_{CW}(\ca{X};D),d_{CW})$ is called the \emph{$j$-th cellular sheaf Laplacian}.
\end{defin}

The assumptions needed to define euclidean structure on the cellular cochain complex can be generalized in a natural way.

\begin{defin}\label{definEuclideanSheaf}
A \emph{euclidean sheaf} on a poset $S$ is a sheaf (diagram) $D$ together with a choice of inner product on each stalk $D(s)$, $s\in S$.
\end{defin}

\begin{con}\label{conEuclideanSheaf}
Let $D$ be a euclidean sheaf on $S$. Let $\Fu\colon\Shvs(S;\Ro\Vect)\to\Cochain(\Ro\Vect)$ be a concrete functor in the sense of Definition~\ref{definConcreteFunctor}. Then the cochain complex $(C^*,d)=\Fu(D)$ naturally attains the inner product. Indeed, each graded component $C^j$ is a direct sum of stalks, taken with some multiplicities:
\begin{equation}\label{eqJthComponentOfF}
C^j=\Fu(D)^j=\bigoplus_{s\in S}D(s)^{\oplus \mult_j(s,\Fu)}.
\end{equation}
The inner product on $C^j$ is defined by setting all direct summands pairwise orthogonal, while the inner product in each summand is summand $D(s)$ is given by assumption.

Since we have a euclidean structure on $(C^*,d)$, we can define \emph{$\Fu$-based Laplacians} in the natural way:
\begin{equation}\label{eqFuValuedLaplacians}
\Delta_j^{\Fu}\colon C^j\to C^j,\qquad \Delta_j^{\Fu}=d^*d+dd^*.
\end{equation}
\end{con}

Recall that a concrete functor $\Fu\colon\Shvs(S;\Ro\Vect)\to\Cochain(\Ro\Vect)$ is said to honestly compute sheaf cohomology, if cohomology of $\Fu(D)$ is isomorphic to the sheaf cohomology, see Definition~\ref{definHonestlyComputes}. This definition and Corollary~\ref{corHarmonicIsCohomology} easily imply the next statement.

\begin{prop}\label{propFuBasedLaplacianKernel}
Let $\Fu\colon\Shvs(S;\Ro\Vect)\to\Cochain(\Ro\Vect)$ be a concrete functor that honestly computes sheaf cohomology, and let $D$ be a euclidean sheaf on a poset $S$. Then
\[
\Ker\Delta_j^{\Fu} \cong H^j(S;D).
\]
\end{prop}

\begin{rem}\label{remFuLaplacianToOptimization}
It follows from Construction~\ref{conNonNegative}, that the linear subspace $\Ker\Delta_j^{\Fu}$ coincides with the zero-set of the nonnegative quadratic function $Q_{\Delta_j^{\Fu}}$ defined on the space $\Fu(D)^j$, described in~\eqref{eqJthComponentOfF}. Harmonic $j$-cochains can be found by minimizing this function. In particular, the space of global sections $\Gamma(S;D)\cong H^0(\Fu(D))$ can be found by minimizing the function
\[
Q_{\Delta_0^{\Fu}}(x)=\|d_0x\|^2,
\]
according to~\eqref{eqStupidObservation}. Here $d_0\colon C^0(S;D)\to C^1(S;D)$ is the leftmost differential in the cochain complex $\Fu(D)$.
\end{rem}

\subsection{Normalization}\label{subsecMathNormalize}

Computing cohomology reducing to finding $\Ker L_j$ which can be done by diagonalizing the matrix $L$. However, when it comes to the study of Laplacians, the nonnegative symmetric matrices $L_j$ are valuable and informative in their own right. Sometimes, the matrix $L_j$ is replaced with another matrix $L_j'$, which is better behaved from computational perspective, but preserves the main properties of $L_j$, namely, (1) $\Ker L_j'=\Ker L_j$, (2) $L_j'$ is nonnegative symmetric. Such matrix $L_j'$ are usually called \emph{normalizations} of $L_j$. This notion originated in spectral graph theory~\cite{Chung}, where many properties of general graphs are more naturally formulated in terms of eigenvalues of $L_j'$ rather than eigenvalues of $L_j$. It should be noticed however, that in the case of general sheaves, a normalization can be performed in a variety of ways, which we review below in the case of cellular sheaves.

\begin{con}\label{conNormalizationOfLaplacian}
Consider a cellular sheaf $D$ on a cell poset $\ca{X}$, and let $\sigma\in\ca{X}$ be a cell. Let $k_\sigma$ denote the dimension of a stalk $D(\sigma)$. Assume that an orthogonal basis $v_{\sigma,1},\ldots,v_{\sigma,k_\sigma}$ is chosen in each stalk $D(\sigma)$. Then the collection of vectors $\bigsqcup_{\sigma\colon \dim\sigma=j}\{v_{\sigma,1},\ldots,v_{\sigma,k_\sigma}\}$ becomes an orthogonal basis of $C^j_{CW}(\ca{X};D)$, for each $j\geq 0$. In this basis, the Laplacian $\Delta_j$ is written as a symmetric nonnegatively defined matrix, which we denote by $L_j$ in order to distinguish from the corresponding operator. Notice, that $L_j$ has a natural block structure; given two cells $\sigma,\tau$ of dimension $j$, we denote by $(L_j)_{\sigma,\tau}$ the block of $L_j$ formed by the intersection of columns indexed by $(\sigma,1),\ldots,(\sigma,k_\sigma)$ and rows indexed by $(\tau,1),\ldots,(\tau,k_\tau)$ (i.e. the matrix corresponding to $\Delta_j$ restricted to $D(\sigma)$ in the domain, and $D(\tau)$ in the target space).
\begin{enumerate}
  \item Naive coordinate-wise normalization. Consider the diagonal matrix $K$ filled with the values
  \[
  K_{(\sigma,i),(\sigma,i)}=\begin{cases}
                              (L_j)_{(\sigma,i),(\sigma,i)}^{-1}, & \mbox{if } (L_j)_{(\sigma,i),(\sigma,i)}\neq 0 \\
                              1, & \mbox{otherwise},
                            \end{cases}
  \]
  --- i.e. the pseudoinverse of the diagonal part of $L_j$. Then the matrix
  \begin{equation}\label{eqWeaklyNormalized}
    L_j'=K^{1/2}L_jK^{1/2}
  \end{equation}
  is symmetric nonnegative, has the same kernel as $L_j$, but the diagonal of $L_j'$ is only filled with $0$'s and $1$'s. We call $L_j'$ the weakly normalized sheaf Laplacian.
  \item Stalk-wise normalization. Let $\sigma$ be a cell of $\ca{X}$, and $A_\sigma=(L_j)_{\sigma,\sigma}$ be the corresponding diagonal block of $L_j$. Since $\Delta_j$ is a nonnegative operator, so is its restriction to the subspace $D(\sigma)$, therefore the symmetric matrix $A_\sigma$ is a matrix of a nonnegative operator. By the orthogonal diagonalization, there exists an orthogonal matrix $Q_\sigma$ such that $Q_\sigma^tA_\sigma Q_\sigma$ is a diagonal matrix $\diag(\lambda_{(\sigma,1)},\ldots,\lambda_{(\sigma,k_\sigma)})$ with nonnegative entries. Again, consider the pseudoinverse $N_\sigma=\diag(n_{(\sigma,1)},\ldots, n_{(\sigma,k_\sigma)})$, where $n_{(\sigma,i)}=1/\lambda_{(\sigma,i)}^{-1}$ if the denominator is nonzero, and $1$ otherwise. Let us form a block diagonal orthogonal matrix $Q$ with blocks $Q_\sigma$ and diagonal matrix $N$ with blocks $N_\sigma$. Then the matrix
  \begin{equation}\label{eqStronglyNormalized}
  L''_j=N^{1/2}Q^tLQ
  \end{equation}
  is symmetric nonnegative, has the same kernel as $L_j$, but all the diagonal blocks $(L''_j)_{\sigma,\sigma}$ are diagonal matrices of the form $\diag(1,\ldots,1,0,\ldots,0)$. We call $L''_j$ the strongly normalized Laplacian.
  \item We may take a matrix $M$ which is pseudoinverse of $L_j$ itself. Then $L_j'''=K^{1/2}L_jK^{1/2}$ is a matrix of the form $\diag(1,\ldots,1,0,\ldots,0)$. The computation of $M$ is equivalent to diagonalization of $L_j$, so this case is usually not considered a normalization, despite the fact it is completely analogous to the previous two examples. However, we find it amusing that the study of sheaf Laplacians has explicit hierarchical structure.
\end{enumerate}
\end{con}

\subsection{Heat diffusion}\label{subsecMathDiffusion}

Assume a nonnegative self-adjoint operator $A\colon V\to V$ is defined on a euclidean vector space $V$. According to Construction~\ref{conNonNegative} the question of finding $\Ker A$ is equivalent to minimizing the corresponding quadratic function $Q_A\colon V\to\Rg$. The latter can be done e.g. by means of gradient descent.

\begin{con}\label{conGradientHeat}
Consider a function $Q_A(x)=\langle Ax,x\rangle$ defined on a finite-dimensional euclidean vector space. Its gradient equals
\begin{equation}\label{eqGradientOfQuadratic}
\nabla Q_A(x_0)=(A+A^*)x_0.
\end{equation}
In the following, it will be assumed that $A$ is self-adjoint, so that $\nabla Q_A=2A$. Henceforth, the continuous gradient descent flow of $Q_A$ with a parameter $\eta>0$ takes the form
\begin{equation}\label{eqGradFlowContinuous}
\dot{x}=-2\eta Ax,\qquad x(0)=x_0,
\end{equation}
this is a linear system of differential equations. Its solution is $x(t)=\exp(-2\eta At)x_0$. If $A$ is nonnegative, then there exists $x_{+\infty}=\lim_{t\to+\infty} x(t)$ and it lies in $\Ker A$.

Similarly, the discrete gradient descent for $Q_A(x)$ has the form
\begin{equation}\label{eqGradFlowDiscrete}
x_{k+1}=x_k-2\eta Ax_k\qquad k=0,1,2,\ldots.
\end{equation}
The solution of this cascade is $x_k=(1-2\eta A)^kx_0$. Again, for nonnegative $A$, there exists $x_{+\infty}=\lim_{k\to+\infty} x_k$ and this vector lies in $\Ker A$.
\end{con}

Estimation of the convergence rate of either continuous or discrete gradient flow is a straightforward exercise, but it emphasizes the role of eigenvalues of $A$.

\begin{con}\label{conSolutionDiagonalized}
If $A=Q^t\Lambda Q$ is a spectral decomposition with the diagonal matrix $\Lambda=\diag(\lambda_1,\ldots,\lambda_n)$, as in Construction~\ref{conSelfAdjoint}, then the solution to equation~\eqref{eqGradFlowContinuous} has the form
\[
x(t)=\exp(-2\eta At)x_0=Q^t\diag(e^{-2\eta\lambda_1t},\ldots, e^{-2\eta\lambda_nt})Qx_0.
\]
Assume that $A$ is nonnegative, i.e. $\lambda_i\geq 0$ for all $i$. As $t\to+\infty$, the exponent $e^{-2\eta \lambda_it}$ is constant whenever $\lambda_i=0$, and decreases to $0$ when $\lambda_i>0$. The exponent with the least nonzero $\lambda_i$ has the slowest convergence rate.

Similarly, the solution to a cascase~\eqref{eqGradFlowDiscrete} has the form
\[
x_k=Q^t\diag((1-2\eta\lambda_1)^k,\ldots, (1-2\eta\lambda_n)^k)Qx_0,
\]
and the main term of the asymptotics is given by the one with the smallest positive $\lambda_i$. These considerations prove the following statement.
\end{con}

\begin{prop}\label{propConvergenceRate}
Let $A\neq 0$ be nonnegative self-adjoint operator, and $\lambda_{\min}>0$ be its least positive eigenvalue. Then the convergence rate of the equation~\eqref{eqGradFlowContinuous} is asymptotically equivalent to $e^{-2\eta\lambda_{\min}t}$, for a generic initial value $x_0$. The rate of convergence of the cascade~\eqref{eqGradFlowDiscrete} is asymptotically equivalent to $(1-2\eta\lambda_{\min})^k$.
\end{prop}

It remains to apply these standard arguments to sheaf Laplacians or their normalized versions. Let $\Fu$ be a concrete functor which honestly computes cohomology, let $\Delta_j^{\Fu}$ be the $\Fu$-based Laplacian of a euclidean sheaf $D$ defined in Construction~\ref{conEuclideanSheaf}, let $L_j$ be a matrix of this Laplacian (in some orthogonal basis), and $L_j'$ --- its normalized version.

\begin{defin}\label{definHeatDiffusion}
The function $Q_j(x)=\langle L_j'x,x\rangle$ is called \emph{the ($\Fu$-based, $j$-th order, normalized, Dirichlet) energy} of the euclidean sheaf $D$. Its gradient descent flow with parameter $\eta$, --- either continuous
\[
\dot{x}=-2\eta L_j'(x),\qquad x(0)=x_0,
\]
or discrete
\[
x_{k+1}=(1-2\eta L_j')x_k,\qquad k=0,1,2,\ldots
\]
--- is called a \emph{heat diffusion} on a sheaf $D$.
\end{defin}

By definition, the heat diffusion is the process that minimizes the sheaf's energy, which is physically meaningful. Proposition~\ref{propConvergenceRate} implies the following.

\begin{cor}\label{corSpeedOfHeatDiffusion}
The convergence rate of the heat diffusion on a sheaf $D$ is determined asymptotically by the least nonzero eigenvalue $\lambda_{\min}$ of the corresponding (normalized) sheaf Laplacian $L_j'$.
\end{cor}

\begin{rem}\label{remOtherEigenvalues}
The zero eigenspace of $L_j'$ coincides with $\Ker\Delta_j^{\Fu}$ (see Construction~\ref{conNormalizationOfLaplacian}) which is isomorphic to $H^j(S;D)$ (see Proposition~\ref{propFuBasedLaplacianKernel}). Therefore, the multiplicity of $0$-th eigenvalue of $L_j'$ equals the $j$-th Betti number of a sheaf $D$, i.e. $\dim H^j(S;D)$. The smallest positive eigenvalue $\lambda_{\min}$ measures ``the worst'' asymptotical term of the heat diffusion process. Other eigenvalues of $L_j'$ contain information on lower terms of the asymptotics for the heat diffusion process. It is believed that these eigenvalues reflect some geometrical properties of the underlying poset $S$, which is supported by a variety of results in spectral graph theory~\cite{Chung}. It should be noticed however, that there are no mathematical results relating eigenvalues of higher order Laplacians with higher dimensional discrete structures, such as simplicial complexes and cell posets.
\end{rem}


\subsection{Examples of diffusion}\label{subsecMathDiffuseExamples}

We end up with two main examples which demonstrate how the terminology and results introduced in the previous subsections motivate particular constructions used in topological deep learning. The definitions of sheaf diffusion and message passing were given in the main part of the text, see subsection~\ref{subsecSheafLearning}. Here we use cellular cochain complex and Roos cochain complex described in subsections~\ref{subsecMathCohomologyCellular} and~\ref{subsecMathCohomologySimplicial} respectively.

\textbf{Diffusion in cellular sheaves.} Let $S$ be a cell poset, e.g. a poset given by a simple graph, as in Example~\ref{exGraphToPoset}. In this case, cellular cochain complex $\Fu_{CW}(D)=(C^*_{CW}(S;D),d_{CW})$ is defined and honestly computes cohomology.

\begin{con}\label{conCellComplexDiffusion}
By Theorem~\ref{thmCWcohomologyIsAGcohomology}, we have
\[
\Gamma(S;D)\cong H^0(S;D)\cong H^0(C^*_{CW}(S;D),d_{CW})=\Ker d_{CW}\colon C^0_{CW}(S;D)\to C^1_{CW}(S;D).
\]
This formula shows that global sections of $D$ can be defined as a subspace of $C^0_{CW}(S;D)=\bigoplus_{\rk s=0}D(s)$. For any element $s\in S$ of rank 1 (that is an edge), the poset $S_{<s}$ is by definition homeomorphic to a $0$-dimensional sphere, i.e. a pair of points. In other words, in a cell poset, each edge $s$ has two endpoints\footnote{This should not surprise anyone who ever seen an interval.}, vertices $s'$ and $s''$. Coherence relations reduce to the linear system $d_{CW}x=0$ given by
\[
f_{s's}(x_{s'})-f_{s''s}(x_{s''})=0 \mbox{ for each edge }s\mbox{ with endpoints }s',s'',
\] 
where $f_{ts}=D(t<s)$ denote the structure maps of the diagram $D$. The two signs in the above formula alternate, because only this choice of signs (incidence numbers) guarantees that $\Fu_{CW}$ honestly computes cohomology, in particular in degree 0. A global section can be found by minimizing the sheaf energy function $Q_{\Delta_0^{\Fu_{CW}}}$ corresponding to the functor $\Fu_{CW}$; we denote it by $Q$ for simplicity:
\[
Q\colon C^0_{CW}(S;D)\to \Ro,\qquad Q(x)=\sum_{s\colon \rk s=1}\|f_{s's}(x_{s'})-f_{s''s}(x_{s''})\|^2.
\]
See Remark~\ref{remFuLaplacianToOptimization} for the derivation of this expression. 

Notice that, even in the case $\dim |S|>1$, i.e. if a cell poset $S$ has elements other than vertices and edges, the global sections of a sheaf are defined through vertices and edges alone, see Remark~\ref{remCompatibilityReducesToCells}. The variables used in the system live over the vertices of $S$, while the edges are needed to define relations, and to perform linear message passing. 
\end{con}

\begin{rem}\label{remDiffusionOnMorse}
The situation with Morse cell posets described in subsection~\ref{subsecMathOneShot} is completely similar, but the functor $\Fu_{M}$ is used instead of $\Fu_{CW}$. This gives a definition of sheaf diffusion on non-graphical structures similar to the one shown on Fig.~\ref{figElementaryMorse}.
\end{rem}

\textbf{Diffusion in binary relations.} Let $R\subseteq A\times B$ be a binary relation (or a bipartite graph), which can be considered as a poset $S_R$ on $A\sqcup B$, see Example~\ref{exBinRelToPoset}. In particular, $R$ may encode a hypergraph, where $A$ is the set of vertices, $B$ is the set of hyperedges, and $R$ is the relation of inclusion. 

\begin{con}\label{conHypergraphDiffusion}
A diagram $D$ on $S_R$ is, by the general definition, an assignment, to each vertex $a\in A$, a real vector space $D(a)$, to each vertex $b$, a space $D(b)$, and, whenever $aRb$, a linear map $f_{ab}\colon D(a)\to D(b)$. No compositionality restrictions are posed on these maps since $\dim |S_R|=1$.

The graded components of the Roos complex $C^*_{Roos}(S;D)$ have the form:
\[
C^0_{Roos}(S_R;D)=\bigoplus_{a\in A}D(a)\oplus\bigoplus_{b\in B}D(b),\quad C^1_{Roos}(S;D)=\bigoplus_{ab\in R}D(b),
\]
while the differential $d_{Roos}\colon C^0_{Roos}(S_R;D)\to C^1_{Roos}(S_R;D)$ is defined by
\[
d_{Roos}(x_a)=-f_{ab}(x_a) \mbox{ for }x_a\in D(a);\quad d_{Roos}(x_b)=x_b \mbox{ for }x_b\in D(b).
\]
All components $C^j_{Roos}(S;D)$ with $j>1$ vanish. The global sections $\Gamma(S_R;D)\cong \Ker d_{Roos}$ are the vectors $\bigoplus_{a\in A}x_a\oplus \bigoplus_{b\in B}x_b\in C^0_{Roos}(S_R;D)$ such that $f_{ab}(x_a)=x_b$ for any $aRb$. These are simply the coherence relations for the diagram $D$: this is not surprising due to Remark~\ref{remRoos0globalSec}. 

The theory developed in subsection~\ref{subsecMathHodge} claims that $\Ker d_{Roos}=\Ker \Delta$ where $\Delta=d_{Roos}^*d_{Roos}$. The heat diffusion is given by minimization of Dirichlet's energy $Q(x)=\langle x,\Delta x\rangle$ given by
\begin{equation}\label{eqRelationDiriclet}
Q(x)=\sum_{ab\in R}\|x_b-f_{ab}(x_a)\|^2,
\end{equation}
according to Remark~\ref{remFuLaplacianToOptimization}. 
\end{con}

\begin{ex}\label{exSingleToManyHypergraph}
Assume that $A=\{a_1,\ldots,a_m\}$, $B=\{b\}$ is a singleton in Construction~\ref{conHypergraphDiffusion}, and the values $x_a\in D(a)$ are forced to stay constant during the heat diffusion. Then minimization of the energy $Q$ given by~\eqref{eqRelationDiriclet} makes $x_b$ tend to the barycenter of $f(x_{a_1}),\ldots,f(x_{a_m})$, since the barycenter minimizes the sum of squared distances to a fixed finite set.
\end{ex}

\begin{rem}\label{remRelationGrouping}
In general, when $B=\{b_1,\ldots,b_k\}$, each letter $b_j$ can be treated as a hyperedge $A_j=\{a\mid aRb_j\}\subset A$. The minimization of $Q$ makes the vectors $\{f_{ab_j}(x_a)\mid a\in A_j\}$ group around their barycenter $x_{b_j}$, for any hyperedge $b_j\in B$. The barycenters $x_b$ can be removed from consideration: instead of minimizing $Q$ given by~\eqref{eqRelationDiriclet}, one can instead minimize the function
\begin{equation}\label{eqHypergraphDirichlet}
Q(x)=\sum_{b_j\in B} \frac{1}{|A_j|}\sum_{a',a''\in A_j}\|f_{a'b_j}(x_{a'})-f_{a''b_j}(x_{a''})\|^2.
\end{equation}
The latter process makes the vectors $f_{ab_j}(x_a)$ group together for any hyperedge $A_j$. The above formula appears in the paper~\cite{duta2024sheaf} introducing hypergraph sheaf neural networks. Essentially, this is the same function as~\eqref{eqRelationDiriclet} (at least if $x_b$ is set equal to the barycenter of the corresponding vectors). Indeed, given $l$ vectors in $\Ro^d$, the sum of their pairwise squared distances equals $l$ times the sum of squared distances to their barycenter. However, the formula~\eqref{eqRelationDiriclet} is a bit more optimal since it has fewer summands. 
\end{rem}

\begin{rem}\label{remInternalStructureOfHyperedges}
The above construction utilizes Roos complex: it is permutation invariant in the sense that the formulas for the energy function $Q(x)$ (either~\eqref{eqRelationDiriclet} or~\eqref{eqHypergraphDirichlet}) do not depend on the order of variables $a\in A_j$ in each hyperedge $A_j$. It is possible to model the heat diffusion by using smaller number of summands. E.g. in Example~\ref{exSingleToManyHypergraph} we could minimize the function $Q(x)=\sum_{i=1}^{m-1}\|f_{a_{i+1}b}(x_{a_{i+1}})-f_{a_ib}(x_{a_i})\|^2$, the energy defined over a minimal functor $\Fu_{\min}$. However, this expression depends on the linear order of $a_i$'s, hence not permutation invariant anymore. Application of such formulas in constructions of neural architectures may be reasonable if the problem's formulation subsumes some internal structure on hyperedges of a hypergraph.
\end{rem}

\textbf{Derived perspective on binary relations.} Treating hypergraphs as binary relations (or bipartite graphs) has a conceptual advantage: the symmetric role of vertices $A$ and hyperedges $B$ becomes more transparent. Given a relation $R\subset A\times B$, we can swap $A$ and $B$ and get a relation $R^{\op}\subset B\times A$, such that $bR^{\op}a\Leftrightarrow aRb$. The corresponding poset $S_{R^{\op}}$ coincides with $S_R^{\op}$ obtained from $S_R$ by reversing the order. As obvious from the definition of geometrical realization, we have $|S_R^{\op}|=|S_R|$. A sheaf theoretical counterpart of this claim should be mentioned.

\begin{rem}\label{remDerivedOpposite}
It was proved in~\cite[Cor.4.18]{LADKANI2008435} that the derived categories $\Der^b(S_R)$ and $\Der^b(S_{R^{\op}})$ are equivalent, see Remark~\ref{remDerivedCatOnPoset}. This can be seen as an instance of a more general fact: Coxeter functors (see~\cite{BernGelfPonomarev}) provide the derived equivalence of the categories of representations~\cite[Thm.3.19]{Longbottom}. In the absence of compositionality restrictions, the category of quiver representations coincides with the category of sheaves on a poset which implies the claim.

The stated claim has a straightforward practical implication. Whenever a hypergraph is studied by homological methods, one can swap the roles of its elements: consider hyperedges as vertices, and vertices as hyperedges. The equivalence of derived categories guarantees that no essential information is lost. Such swap makes sense if calculations with $R^{\op}$ are less costly than those with $R$.    
\end{rem}

\begin{con}\label{conDowker}
The construction of sheaves on a hypergraph is utilizing the poset $S_R$ which is 1-dimensional; therefore all cohomology in degrees $>1$ of all sheaves on $S_R$ vanish, see Corollary~\ref{corVanishing}. However, in algebraic topology there is another way to transform a binary relation into a higher-order structure, the Dowker complex. A binary relation (a hypergraph) can be turned into a simplicial complex by treating hyperedges as simplices and adding their proper subsets. Dowker complex better resembles the intuitions behind applied papers, since hyperedges with $>2$ vertices actually become higher-dimensional structures and there may be nontrivial higher-order cohomology. 

Swapping the arguments of a binary relation leads to a Dowker duality. The simplicial complexes constructed from $R$ and $R^{\op}$ are homotopically equivalent; an elegant functorial proof of this claim can be found in~\cite{BrunSalbu}. Therefore, the homotopy type of a binary relation is well defined and symmetric in the arguments of the relation. It seems an interesting research problem: to define sheaf theory on binary relations in the way that remembers higher order structure of Dowker complex. 
\end{con}

\textbf{Diffusion in general.} For a general poset $S$, the somehow canonical way to define Laplacian and diffusion on sheaves on $S$ is by using Roos complex $\Fu_{Roos}$, since this construction is universal. This idea of combining Laplacians with general posets was proposed in~\cite{HypergraphsSimpSets} motivated by the particular case of binary relations and hypergraphs. 

Another way to introduce Laplacians and diffusion over arbitrary posets is by using a minimal functor $\Fu_{\min}$. Let $S_{low}$ denote the subset of minimal elements of $S$. It can be seen from the inductive construction of $\Fu_{\min}$ presented in subsection~\ref{subsecMathMinimalComputations}, that $C^0_{\min}(S;D)=\bigoplus_{s\in S_{low}}D(s)$, and the number of summands of $C^1_{\min}(S;D)$ equals, by Definition~\ref{definMultiplicityOfStalk}, to the sum of multiplicities $\sum_{s\in S} \mult_1(s,\Fu_{\min})$, i.e. the number of generators of degree $1$ output by Algorithm~\ref{algIncMatrixMain}. The formula for the energy function $Q_{\Delta_0^{\Fu_{\min}}}$ is determined by this generating set.

We finish with the general remark addressed to the specialists in deep learning.

\begin{rem}
In general, the mathematical pipeline described in Appendix sections~\ref{secMathCohomology} and~\ref{secMathLaplacians}, consists of the following steps 
\begin{enumerate}
  \item choice of cochain complex which honestly computes cohomology (or just global sections),
  \item choice of euclidean structure on a sheaf,
  \item choice of normalization for a Laplacian.
\end{enumerate}
Each of this steps subsumes picking real parameters~\footnote{Sometimes some discrete optimization is involved, like choosing the optimal dimensions of stalks, but in all cases there is a bunch of real numbers to optimize.} hence can be learned end-to-end by gradient descent, or constructed by hand based on the specifics of the problem at hand. 
\end{rem}

\begin{ex}
The mechanism of sheaf attention networks mentioned in Section~\ref{secReviewShvsML} is an example of learning the euclidean structure. The standard euclidean structure on stalks of edges is scaled by the learnable attention weights. 
\end{ex}

\begin{ex}
If hyperedges of a hypergraph have an unknown internal structure, then it can be captured by the learnable cochain complex as outlined in Remark~\ref{remInternalStructureOfHyperedges}. The only mathematical restriction on the complex is that it should honestly compute cohomology, otherwise the whole theory doesn't make much sense.
\end{ex}

\section{Connection sheaves}\label{secMathConnection}

In this section we briefly recall the basic notions related to connection sheaves. Here, similar to Appendices~\ref{secMathCohomology} and~\ref{secMathLaplacians}, $\Vv$ is an abelian category, with $\Vv=\Ro\Vect$ being the main example. For a more detailed overview of the mathematical results related to posets we refer to~\cite{Barmak2012GcoloringsOP}.

\begin{defin}\label{definConnectionSheaf}
A $\Vv$-valued diagram $D$ on a poset $S$ (a sheaf $\ca{D}$ on an Alexandrov space $X_S$) is called a \emph{connection sheaf} if all structure maps $D(s_1\leq s_2)$ are isomorphisms in $\Vv$. More generally, let $G$ be a subgroup of the monoid $\Hom_{\Vv}(V,V)$ for some $V\in\Vv$. Then a \emph{$G$-connection sheaf} on $S$ is a diagram $D$ in which all structure maps $D(s_1\leq s_2)$ belong to $G$.
\end{defin}

In classical algebraic topology connection sheaves are known as local systems; they were introduced by Steenrod in~\cite{Steenrod1943HomologyWL}.

\begin{rem}\label{remDifGeomConnections}
Connection sheaf on a finite poset is a suitable discretization for the notion of a bundle over a manifold, such as tangent bundle of a smooth manifold. In a bundle, we have a real vector space of the same dimension attached to each point of a manifold. In order to take derivatives of tensor fields, the spaces attached to infinitely closed points are related by a certain isomorphism; this is formalized in the notion of connection from differential geometry. This basic and extremely powerful object of differential geometry leads to the notions of parallel transport and Riemann curvature tensor.

A straightforward way to discretize this setting is to take a poset $S$, assign, to each element $s\in S$, a real vector space $D(s)$ of the same dimension, and, whenever $s_1<s_2$ (which is treated as closeness of points) assign an isomorphism between $D(s_1)$ and $D(s_2)$. This seemingly leads to Definition~\ref{definConnectionSheaf} of a sheaf. However, the analogy between connection sheaves on posets and differential connections on manifolds may be a bit misleading if a poset is more complicated than just a graph. In sheaf theory compositionality $D(s_1<s_2);D(s_2<s_3)=D(s_1<s_3)$ is required. In differential geometry, the discrepancy between ``various ways to move around a point'' measures the curvature and may be nontrivial. In this sense, connection sheaves may be seen as discrete models of flat manifolds. General quiver representations with isomorphic maps can model arbitrary manifolds.
\end{rem}

Assume that $S$ is connected (meaning that the geometrical realization $|S|$ is a connected topological space). This means that for any two points $s',s''\in S$ there is a path
\[
p=(s'\lessgtr s_1\lessgtr s_2\lessgtr \cdots \lessgtr s_n\lessgtr s'')
\]
where $a\lessgtr b$ denotes either $a\leq b$ or $a\geq b$. The reversed path $p^{-1}$ from $s''$ to $s'$ is defined in a straightforward manner.

\begin{con}\label{conParallelTransport}
Given a connection sheaf $D$ on $S$ and a path $p$ between $s'$ and $s''$ as above, we can define \emph{the parallel transport} operator
\[
T_p\colon D(s')\to D(s''),\qquad T_p=D(s'\lessgtr s_1);D(s'\lessgtr s_1);\cdots;D(s_n\lessgtr s''),
\]
where we set by definition
\[
D(a\lessgtr b)=\begin{cases}
                 D(a\leq b), & \mbox{if } a\lessgtr b \mbox{ reads as }a\leq b\\
                 D(b\leq a)^{-1}, & \mbox{if } a\lessgtr b \mbox{ reads as }a\geq b.
               \end{cases}
\]
We obviously have $T_{p^{-1}}=T_p^{-1}$. The definition of a sheaf~\ref{definSheafDiagramMainText} directly implies $T_{p'}=T_{p''}$ if a path $p''$ is obtained from $p'$ by deletion of fragment $(s_i=s_{i+1})$ or substituting a fragment $(s_{i-1}\leq s_{i}\leq s_{i+1})$ with $(s_{i-1}\leq s_{i+1})$ or similar operation for the reversed relation $\geq$, --- we call such operations triangular reductions. Therefore, it can be assumed without loss of generality that a parallel transport is determined along paths composed of alternating strict inequalities. This argument proves the following.
\end{con}

\begin{lem}\label{lemHomotopicPaths}
If $p',p''$ are two homotopic paths from $s'$ to $s''$ in $S$, then $T_{p'}=T_{p''}$.
\end{lem}

\begin{proof}
Construction~\ref{conParallelTransport} shows that two paths differing by a triangle $(s_0<s_1<s_2)$ in $|\ord S|=|S|$ give the same parallel transport operators. By the cellular approximation theorem, every homotopy between $s'$ to $s''$ lies in a 2-skeleton of $|S|$. Since all 2-cells of $|S|$ have the form $(s_0<s_1<s_2)$ a homotopy is a sequence of triangular reductions or their inverses, hence preserves the parallel transport operator.
\end{proof}

\begin{rem}\label{remMonodromy}
If the initial and end-points coincide, a path from $s$ to itself is called a loop, and its homotopy class is, by definition, the element of the fundamental group $\pi_1(|S|,s)$. Lemma~\ref{lemHomotopicPaths} implies that the parallel transport map $T_p$ provides a well-defined homomorphism from $\pi_1(|S|,s)$ to $\Iso(D(s))$, which is called a \emph{monodromy} of a connection sheaf. In case $D$ is $G$-connection sheaf, the monodromy is valued in the subgroup $G$. A connection sheaf is called \emph{flat} if its monodromy is a trivial representation. This means that traversing any loop is identical on $D(s)$.
\end{rem}

More generally, let $\ConDiag(S,\Vv)$ denote the full subcategory of connection sheaves in $\Diag(S,\Vv)$. The next proposition and its proof are subsumed by the previous remark, and seem to be folklore. 

\begin{prop}\label{propFundGrpRepresentations}
For a connected poset $S$, the category $\ConDiag(S,\Vv)$ is equivalent to the category $\Repr(\pi_1(|S|),\Vv)$ of $\Vv$-valued representations of the fundamental group $\pi_1(|S|)$.
\end{prop}

\begin{rem}
By McCord's theorem we have $\pi_1(|S|)\cong \pi_1(X_S)$, see Remark~\ref{remMcCord}. The claim of Proposition~\ref{propFundGrpRepresentations} can thus be formulated as an equivalence of the categories $\ConDiag(S,\Vv)$ and $\Repr(\pi_1(X_S),\Vv)$.
\end{rem}

In view of Remark~\ref{remDifGeomConnections} on relation between connection sheaves and connections in differential geometry, this proposition may be considered as a discrete version of Riemann--Hilbert correspondence (rather naive though). To prove it, we choose a base point $s_0\in S$ and a spanning tree $G$ of the 1-skeleton $|(\ord S)_1|$. For any point $s\in S$, there is a unique path $p_s$ from $s_0$ to $s$ lying in $G$. These data will be assumed fixed. The following two constructions emulate the procedure of collapsing/decolapsing a spanning skeleton into one point, and prove Proposition~\ref{propFundGrpRepresentations}.

\begin{con}\label{conConSheafToRepr}
Given a connection sheaf $D$ on $S$, we construct a representation of $\pi_1(|S|,s_0)$ on the space $D(s_0)$. For an element $g=[p]\in\pi_1(|S|,s_0)$ represented by a loop $p$ from $s_0$ to itself, set the action of $g$ on $D(s_0)$ defined by $T_p$. This is well-defined on homotopy classes by Lemma~\ref{lemHomotopicPaths}. This is a group homomorphism from $\pi_1(|S|,s_0)$ to $\End(D(s_0))$ by the definition of the parallel transport.
\end{con}

\begin{con}\label{conReprToConSheaf}
Given a fundamental group representation $R\colon \pi_1(|S|,s_0)\to \End(V)$, we construct a connection sheaf $\bar{V}$ on $S$. Set $\bar{V}(s)=V$ for any $s$. Then for $s_1<s_2$ set
\[
\bar{V}(s_1<s_2)=R([p_{s_1};(s_1<s_2);p_{s_2}^{-1}])\colon V=\bar{V}(s_1)\to V=\bar{V}(s_2),
\]
where $[p_{s_1};(s_1<s_2);p_{s_2}]\in \pi_1(|S|,s_0)$ is the homotopy class of the loop obtained by concatenation of the chosen path $p_{s_1}$ from $s_0$ to $s_1$, the shortcut from $s_1$ to $s_2$, and the reversed chosen path $p_{s_2}^{-1}$ from $s_2$ to $s_0$. It is a simple exercise to prove that $\bar{V}$ is a well-defined connection sheaf.
\end{con}

The proof that Constructions~\ref{conConSheafToRepr} and~\ref{conReprToConSheaf} are inverses to each other up to natural isomorphism (given the tree $G$) is rather straightforward and left as an exercise. Of course, a more conceptual and mathematically correct way to state Proposition~\ref{propFundGrpRepresentations} is via the fundamental groupoid of $|S|$.

\begin{rem}\label{remConnectionCohomology}
It should be noted that the category $\Repr(\pi_1(|S|),\Vv)$ (and hence $\ConDiag(S,\Vv)$ according to Proposition~\ref{propFundGrpRepresentations}) is an abelian category itself, and it has enough injectives if so does $\Vv$. However, the subcategory $\ConDiag(S,\Vv)$ does not inherit all its structure from $\Diag(S;\Vv)$. Injective sheaves on $S$ may fail to be connection sheaves, and injective connection sheaves may fail to be injective sheaves of $\Diag(S,\Vv)$. In particular, the notion of cohomology in $\Diag(S,\Vv)$ and $\ConDiag(S,\Vv)$ vary. When speaking about cohomology of connection sheaves, one should be extremely careful in terminology.
\end{rem}

\begin{ex}\label{exConstantCohomologyConnection}
As an example, consider the constant sheaf $\overline{\Ro^1}$ on $S$. Its cohomology, as defined in the category $\Diag(S,\Ro\Vect)$, coincides with the singular cohomology of the geometrical realization $H^*(|S|;\Ro)$, as discussed in Example~\ref{exConstantSheaf}. However, $\overline{\Ro^1}$ is actually a connection sheaf, $\overline{\Ro^1}\in\ConDiag(S,\Vv)$. In terms of Proposition~\ref{propFundGrpRepresentations} it corresponds to the trivial $\pi_1(|S|)$-module $\Ro^1$, so its cohomology coincides with the group cohomology $H^*(\pi_1(|S|);\Ro^1)$. The modules $H^*(|S|;\Ro)$ and $H^*(\pi_1(|S|);\Ro^1)$ may be different, e.g. if $|S|$ is homeomorphic to 2-sphere.

Similarly, if $D$ is an arbitrary connection sheaf, then $H^*(S;D)$ computed in $\Diag(S;\Vv)$ is the cohomology of $S$ with local coefficients $D$, while $H^*(S;D)$ computed in $\ConDiag(S,\Vv)$ is the cohomology of the $\pi_1(S,s_0)$-module $D(s_0)$. In a sense, passage from the category $\Diag(S;\Vv)$ to the subcategory $\ConDiag(S,\Vv)$ forgets higher homotopical information encoded in $S$.
\end{ex}

We have the following simple corollary of Proposition~\ref{propFundGrpRepresentations} which may seem counterintuitive at first glance.

\begin{cor}\label{corConnSheafOnSimplyConnected}
If $|S|$ is simply connected, then every connection sheaf on $S$ is isomorphic to the constant sheaf with the same stalks' dimension.
\end{cor}

\begin{rem}\label{remWhyLearnConnectionSheaf}
The natural question arises: if the connection sheaves are just the constant sheaf, then what's the point in ``learning'' a connection sheaf as described in Subsection~\ref{subsecSheafLearning} and the Section~\ref{secReviewShvsML} of the review? There are several reasonable answers.
\begin{enumerate}
  \item Data structures, on which connection sheaf learning is performed are usually graphs, and graphs are rarely simply connected.
  \item If data structures are highly dimensional, we actually learn connection quiver representations rather than ordinary connection sheaf and simply forget compositionality relations.
\end{enumerate}
Anyhow, the question once again leads to Problem~\ref{problRelations} from the list. There won't be any actual higher-dimensional topology in deep learning, until the relations are taken into account.
\end{rem}

\section{Information capacity of sheaves}\label{secMathSpaceRestoredFromShvs}

In this section we discuss to which extent the category of sheaves on a space remembers the space. All the claims made in this subsection are standard from mathematical perspective, but they rigorify the usage of sheaves in geometrical problems and provide a general fundament for sheaf learning. We couldn't find the mentions of these claims in the literature, and add them in our review for completeness.

\subsection{Topos of sheaves}

Given a topological space $X$ and a chosen target category $\Vv$, we can form a category $\Shvs(X,\Vv)$ of $\Vv$-valued sheaves on $X$. The space $X$ is a bare geometrical structure, no computations are available on $X$ itself. We need to borrow certain ``types of computations'' which exist in the target category $\Vv$, and put them into $\Shvs(S,\Vv)$ in order to be able to say something about $X$. The natural question arises: how much information do we loose, when passing from $X$ to the category $\Shvs(X;\Vv)$? In case $\Vv=\Sets$ this is a classical question and the answer is well-known.

\begin{con}\label{conToposOfSheaves}
We refer to~\cite{MacLaneMoerdijk} for the definition of a topos. The category $\Sets$ is a topos, and, for any topological space $X$, the category $\Shvs(X,\Sets)$ is also a topos. Given two topoi $\ST{T}_1,\ST{T}_2$, a geometrical morphism $F$ from $\ST{T}_1$ to $\ST{T}_2$ consists of a pair of adjoint functors $f_*\colon \ST{T}_1\rightleftarrows \ST{T}_2\colon f^*$ such that $f^*$ preserves finite limits. The class of all toposes with geometrical morphisms between them forms a category, which will be denoted $\Topoi$. A continuous map of topological spaces $f\colon X\to Y$, induces direct and inverse image functors $f_*\colon \Shvs(X;\Sets)\rightleftarrows\Shvs(Y;\Sets)\colon f^*$ as described in subsection~\ref{subsecMathFunctorial}, they form a geometrical morphism in the category $\Topoi$. Hence, passing to a topos of $\Sets$-valued sheaves defines a functor
\begin{equation}\label{eqSpaceToTopos}
\Shvs(\cdot,\Sets)\colon \Top\to \Topoi,\quad X\mapsto \Shvs(X,\Sets).
\end{equation}
\end{con}

Topological space $X$ is called \emph{sober} if, roughly speaking, it can be reconstructed from its lattice, or locale, $\OpSets(X)$ of open subsets. See formal definition e.g. in~\cite[p.477]{MacLaneMoerdijk}.

\begin{prop}\label{propTopoiFullyFaithful}
The functor $\Shvs(\cdot,\Sets)$ is fully faithful on the full subcategory $\Top_{sober}$ of sober topological spaces.
\end{prop}

The proof follows from~\cite[Cor.4,p.481]{MacLaneMoerdijk} and~\cite[Prop.2, p.491]{MacLaneMoerdijk}. It is not surprising that the lattice $\OpSets(X)$ is what actually matters: presheaves were defined as diagrams on $\OpSets(X)^{\op}$.

\begin{cor}\label{corToposIsoThenSoberHomeo}
If $X,Y$ are sober topological spaces, and $\Shvs(X,\Sets)$ is isomorphic as a topos to $\Shvs(Y,\Sets)$, then $X\cong Y$.
\end{cor}

\begin{rem}\label{remRecoverXfromTopos}
Unwinding the formal proof of Proposition~\ref{propTopoiFullyFaithful} we can notice that the whole big category $\Shvs(X,\Sets)$ is not actually needed to recover a sober space $X$. Basically, the lattice $\OpSets(X)$ coincides with the collection of subsheaves of the constant $\ast$-valued sheaf, up to isomorphism. Here $\ast$ is the singleton, the final object of $\Sets$.
\end{rem}

\begin{rem}\label{remFiniteT0SpaceSetSheaves}
Any finite $T_0$-space can be shown to be sober. Corollary~\ref{corToposIsoThenSoberHomeo} gives a strict mathematical ground for the belief that $\Sets$-valued diagrams over finite posets (resp., sheaves over corresponding $T_0$-spaces, according to Proposition~\ref{propPosTop}) store the complete information about the underlying topology. It should be noticed however, that not every Alexandrov $T_0$-space is sober: such space is sober if and only if the corresponding poset is Artinian~\cite[Thm.C.6]{Lindenhovius}.
\end{rem}

\begin{rem}\label{remPreposetsNotNeededLocalic}
If $S$ is just a preordered set, and $\bar{S}$ is its corresponding poset, as in Construction~\ref{conPreposetToPoset}, then we have $\OpSets(X_S)\cong \OpSets(X_{\bar{S}})$. Therefore Alexandrov topology $X_S$ and the corresponding $T_0$-topology $X_{\bar{S}}$ are indistinguishable by means of topoi of sheaves. This is the main conceptual reason why we restrict consideration to posets, and not general preordered sets, see Remark~\ref{remPreposetsNotNeeded}.
\end{rem}

The natural question arises in abelian setting (see definition of abelian category in Appendix~\ref{secMathCohomology}). To which extent the abelian category $\Shvs(X,\ko\Vect)$ remembers a space $X$? This situation differs from the case of a topos: it appears that the abelian categories are not that ``rigid'' objects as topoi as shown in the next remark.

\begin{rem}\label{remToAbCatNotFaithful}
The functor $\Shvs(\cdot,\ko\Vect)\colon \Top\to\AbCat$ is not faithful (even on sober spaces). Indeed, there is exactly one continuous map from a singleton $\ast$ to itself, but there are infinitely many additive endofunctors on $\ko\Vect$ (which is identified with $\Shvs(\ast,\ko\Vect)$). E.g. we have an infinite series $V\mapsto V^{\oplus n}$.

It could be noticed that, for sheaves valued in an abelian category, both direct $f_*$ and inverse image $f^*$ maps are additive functors. We could wonder, if it is possible to make category $\AbCat$ more rigid, by considering adjoint pairs of additive functors $(f_*,f^*)$ as morphisms instead of a single additive morphism, i.e. by ``simulating'' a topos. However, this doesn't resolve the previous counterexample: the functor $f_*\colon V\mapsto V^{\oplus n}$ has an adjoint $f^*\colon V\mapsto V^{\oplus n}$.
\end{rem}

Nevertheless, the category of $\ko\Vect$-diagrams on a poset still remembers a poset. 

\begin{prop}\label{propRecoverXfromAbCat}
The locale $\OpSets(X)$ can be recovered from $\Shvs(X,\ko\Vect)$. The functor $\Shvs(\cdot,\ko\Vect)\colon \Top\to\AbCat$ is full on sober spaces.
\end{prop}

\begin{proof}
We only need to prove the first claim, the second follows from the equivalence between the category of spatial locales and sober topological spaces, see~\cite[Cor.4,p.481]{MacLaneMoerdijk}. There is a distinguished element $\ko\in\ko\Vect$, which is a unique up to isomorphism simple object of the abelian category $\ko\Vect$. Hence we can construct a constant $\ko$-valued sheaf $\bar{\ko}\in\Shvs(X,\ko\Vect)$. Subsheaves of $\bar{\ko}$, defined up to isomorphism, form a poset which is isomorphic to the locale $\OpSets(X)$.
\end{proof}

This proof is similar to Remark~\ref{remRecoverXfromTopos}, but now we use $\ko$ instead of $\ast$ and $0$ instead of $\varnothing$. This observation mimics the boolean logic (inherent to the topos $\Sets$) in terms of linear algebra (inherent to the abelian category $\ko\Vect$).

\begin{rem}\label{remTheoryIs1dimStrange}
Proof of Proposition~\ref{propRecoverXfromAbCat} shows that, from mathematical perspective, the complete information about finite $T_0$-topology can be extracted from sheaves, whose stalks are no more than 1-dimensional. This is somehow orthogonal to the deep learning practice, where it is assumed that a single sheaf with higher-dimensional stalks (or a few such sheaves used in the layers of a network) is already expressive enough to describe the properties of the space, see Section~\ref{secReviewShvsML}. This discrepancy between theory and practice seems interesting and deserves further investigation.
\end{rem}

\begin{rem}\label{remDerivedCatOnPoset}
The topos $\Shvs(X;\Sets)$ is interesting in its own right. However, when working in abelian setting, it is common in algebraic geometry to consider not the category $\Shvs(X;\Abel)$ of sheaves itself, but its derived versions which are better suited for homological applications. For a poset $S$, or the corresponding Alexandrov space $X_S$, the bounded derived category $\Der^b(S)$ of finite-dimensional sheaves is defined. A comprehensive study of derived categories of $\Abel$-valued sheaves on posets was done by Ladkani~\cite{LADKANI2008435}. There exist examples of non-isomorphic finite posets which are derived-equivalent.
\end{rem}

\nocite{*}
\printbibliography

\end{document}